\PassOptionsToPackage{square,sort&compress}{natbib}
\documentclass[10pt,number,preprint,times]{elsarticle}
\usepackage{physics}
\makeatletter

\usepackage{amssymb}
\usepackage{atbegshi}
\AtBeginDocument{\AtBeginShipoutNext{\AtBeginShipoutDiscard}}
\usepackage[utf8]{inputenc} 
\usepackage{alltt}
\usepackage{lipsum}
\usepackage{epstopdf}

\usepackage{amsfonts}
\usepackage{mathtools}
\usepackage{amsmath}
\usepackage{times}
\usepackage{tikz-cd}
\usepackage{manyfoot}%

\usepackage{booktabs}
\usepackage[T1]{fontenc}
\usepackage[english,french]{babel}
\usepackage{lmodern}
\usepackage[a4paper]{geometry}
\usepackage{epstopdf}
\usepackage{mathrsfs}
\usepackage{latexsym,amssymb,amsfonts}
\usepackage{enumerate}
\usepackage{relsize,exscale}
\usepackage{dsfont}
\usepackage{microtype}
\usepackage{xcolor}
\usepackage{enumitem} 
\usepackage{graphicx}
\usepackage{float}
\usepackage{upgreek}
\usepackage{caption}
\usepackage{subfigure}

\usepackage[labelfont=bf]{caption}
\usepackage{xcolor}
\definecolor{mred0}{rgb}{0,0.406,0.596}
\usepackage{algorithm}                   

\definecolor{mcyan}{rgb}{0,0.501,0.501}
\definecolor{mcyan}{rgb}{0,0.501,0.501}
\definecolor{mblue}{rgb}{0,0.406,0.796}
\definecolor{mcgreen}{rgb}{0.180,0.545,0.41}
\definecolor{mred0}{rgb}{0,0.406,0.596}
\definecolor{mred3}{rgb}{0.6,0,0.8}
\definecolor{DarkGreen}{rgb}{0.278,0.701,0.913} 
\definecolor{mred1}{rgb}{0.180,0.545,0.341}

\usepackage{longtable}
\usepackage{algorithm}                   
\usepackage{algpseudocode}               

\numberwithin{equation}{section} 

\usepackage{amsthm}
\newtheorem{thm}{Theorem}[section]
\newtheorem{defi}{Definition}[section]

\newtheorem{lem}{Lemma}[section]
\newtheorem{rem}{Remark}[section]

\usepackage{float}
\usepackage{multicol}
\usepackage[colorlinks=true,linkcolor=blue,citecolor=green,urlcolor=red]{hyperref}
\usepackage[figurename=Fig.]{caption}  

\journal{Elsevier}
\makeatletter
\def\ps@pprintTitle{%
	\let\@oddhead\@empty
	\let\@evenhead\@empty
	\def\@oddfoot{\reset@font\hfil\thepage\hfil}
	\let\@evenfoot\@oddfoot
}
\makeatother

\usepackage[numbers,sort&compress,square]{natbib}

\date{}
\begin{document}
	\maketitle
	\selectlanguage{english}
	\begin{abstract}
Water-borne diseases are still a major public health concern, as there are circumstances under which water could act as a carrier of the pathogen, extending their modeling beyond direct contact between hosts. In the present work, we introduce a new mathematical framework, coupling epidemiological dynamics with fluid motion, in order to understand the spatial spread of such an infection. Our model couples the classical Susceptible–Infected–Recovered (SIR) model with the Navier–Stokes equations describing the motion of fluids, which enhances the existing literature by simultaneously taking into account two aspects: the pathogen being transported by the water currents and the dependence of the effective viscosity of the fluid on the pathogen concentration. We apply the Faedo–Galerkin method and compactness arguments to prove the existence of a global, biologically feasible solution to the coupled SIR–Pathogen–Navier–Stokes (SIRPNS) system. Additionally, we investigate the uniqueness of such solutions in the two-dimensional case. Finally, by constructing a numerical scheme based on the semi-implicit scheme in time and the finite element method in space, we run several numerical simulations to show how infection dispersal, environmental contamination, and hydrodynamic feedback together govern the spatial dynamics, persistence, and eventual decline of waterborne epidemics.
	\\[1ex]
	\end{abstract}
	\begin{frontmatter}
		\title{Waterborne epidemics via a new coupled SIR–Pathogen--Navier–Stokes system: Mathematical modeling, nonlinear analysis and numerical simulation}

		\author{Mohamed Mehdaoui}
		\address{Euromed University of Fez, UEMF, 30 000, Fez, Morocco}
		\ead{m.mehdaoui@ueuromed.org, m.mehdaoui@edu.umi.ac.ma}
		
		\author{Yassine Ouzrour}
		\ead{yassine.ouzrour@usmba.ac.ma} 
		\address{Laboratoire de Mathématiques et Applications aux Sciences de l'Ingénieur, Ecole Normale Supérieure de Fès, Université Sidi Mohamed Ben Abdellah, Morocco}

	\end{frontmatter}

	\def\l({\left(}
	\def\r){\right)}
	\section{Introduction and motivation}\label{intro}
Analyzing the dynamics of infectious diseases through compartmental modeling started with the pioneering work by Kermack and McKendrick in the early twentieth century \cite{kermack1927contribution}. Their classic Susceptible–Infected–Recovered model has since become a standard tool in epidemic modeling, due to its simplicity yet outstanding predicting capacity. Since then, the basic formulation has been generalized in all ways to include an array of biological and environmental factors, closing the gap between theoretical predictions and realistic applications. When it comes to capturing the geographic nature of epidemics, numerous models governed by reaction–diffusion systems have been proposed. However, the existing literature still exhibits a notable limitation when it comes to the incorporation of environments where the pathogen transmission is mediated through a fluid medium. In fact, it is known from medical studies that pathogens of epidemics such as cholera, dysentery, or influenza can survive in the ambient environment and occasionally proliferate before infecting new hosts \cite{morse1995factors,lipp2002effects,gu2025distinct}, leading to the infection cycle illustrated by Fig. \ref{cyclefig}. Consequently, this results into the question of determining a suitable way to incorporate fluid dynamics into mathematical epidemic models, in the aim of increasing their practicality.

In comparison to the existing literature, when it comes to the modeling of environmental pathogens, the concentration is mathematically considered as either predetermined or regarded as passive scalars that are susceptible to diffusion or decay. We refer for example to \cite{capasso1978global,el2020mathematical,dong2022global,mehdaoui2024well,mehdaouii2023analysis,mehdaoui2023dynamical,sabbar2025refining,sabbar2024probabilistic,zagour2024time,mehdaoui2024optimal,wang2016reaction,hau2006wind,lacitignola2016backward,buonomo2012forces,shuai2013global,wang2025effect,mehdaoui2024new,zhang2018spatial,yang2025reaction}. When pathogen accumulation itself alters the fluid's physical characteristics, the intricate feedback that may occur is not taken into consideration. More specifically, variations in density, flow patterns, or viscosity brought on by pathogen concentrations are rarely taken into account.

Taking the above discussion into account, in this paper, our main goal is to get around the restriction of earlier works by creating a novel mathematical framework that combines fluid mechanics and epidemiological dynamics. In particular, we enhance the classical SIR model \cite{mehdaoui2023optimal} to explicitly interact with the incompressible Navier-Stokes equations \cite{bendahmane2025mathematical,bendahmane2024mathematical2,bendahmane2024mathematical1,bendahmane2026mathematical,ouzrour2025well}. Our approach results in a novelty exhibited by two aspects. First, the model captures the crucial role of fluid transport in regulating the spatial patterns of infection by taking into account the advection of both host densities and pathogen concentrations by the fluid velocity field. Second, a two-way coupling between epidemiological and hydrodynamic processes is established by treating the fluid's viscosity as a function of pathogen concentration.

\begin{figure}[h!]
\centering
\includegraphics[height=10cm]{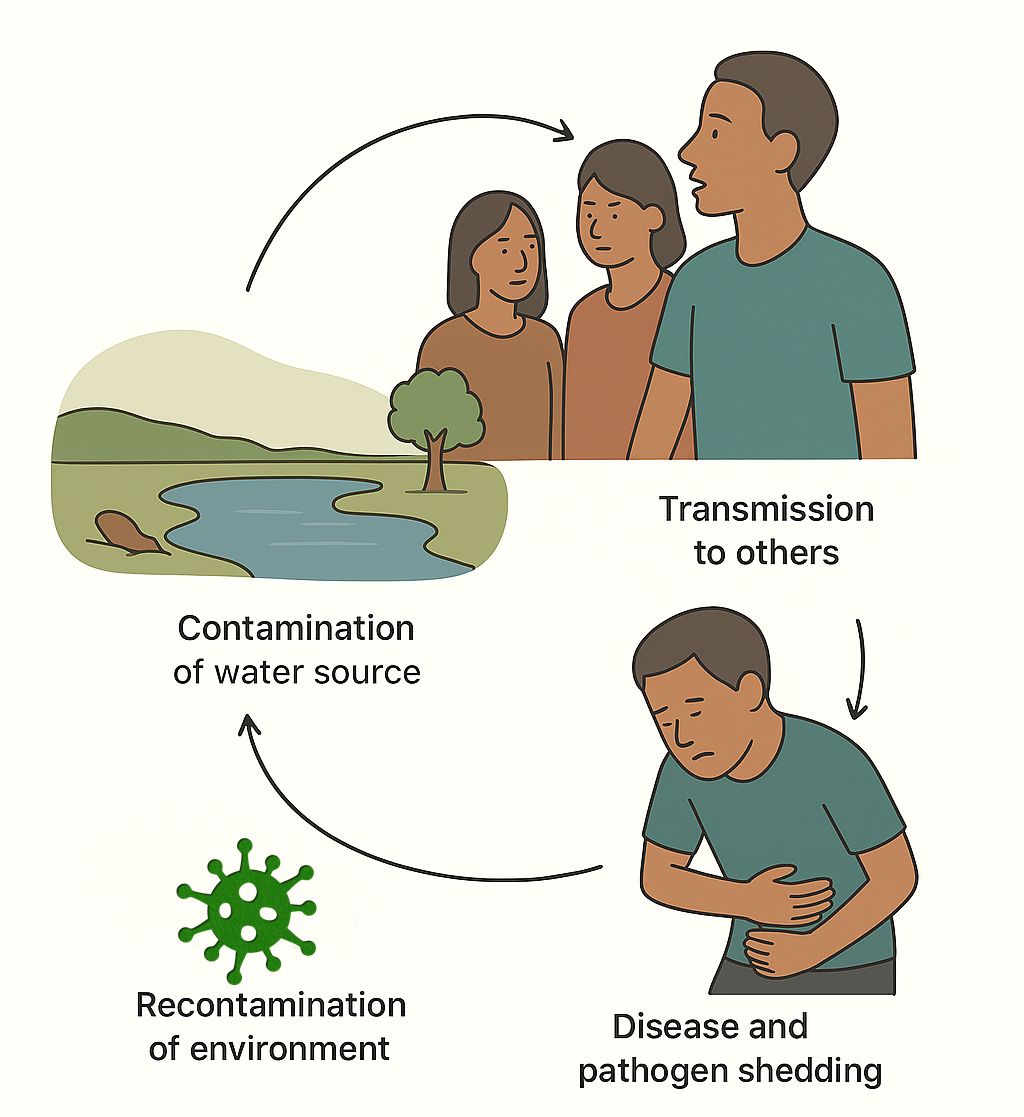}
\caption{The cycle of waterborne epidemics transmission.}
\label{cyclefig}
\end{figure}

The resulting coupled system of nonlinear partial differential equations poses a challenge when it comes to ensuring the existence, uniqueness as well as the regularity of the established solutions, mainly caused by the coupling between reaction-diffusion dynamics and Navier-Stokes flows through the nonstandard nonlinearity in the viscosity term, which complicates theoretical analysis and numerical simulation. The remainder of this paper is organized as follows. We establish the biological context and develop the model in Section~\ref{model}. The mathematical analysis of the system is covered in Section~\ref{analysis}, where we establish results on global existence, uniqueness, and the long-term behavior of solutions. We provide numerical simulations in Section~\ref{numerics} that demonstrate the qualitative aspects of the model and demonstrate how pathogen–viscosity feedback and fluid transport can change the course of epidemics in various scenarios. Section~\ref{conclusion} concludes by summarizing the key findings and suggesting possible lines of inquiry for further study.

\section{Model formulation}\label{model}

We present the mathematical formulation of the proposed framework. Let the habitat of interest be represented by a domain $\Omega \subset \mathbb{R}^d$ with $d \in \{2,3\}$, which is assumed to be open and bounded. By $\partial \Omega$ we denote the boundary of $\Omega$ and assume it hereafter to be smooth enough. Our aim is to describe, over a fixed time interval $(0,T)$ with $T>0$, the coupled dynamics of host populations, pathogen transport, and fluid motion in the context of waterborne epidemics. The model is built on the following biologically motivated assumptions:

\begin{enumerate}[label=\textcolor{blue}{({A${{_\arabic*}}$})}]
	\item \label{a1} \textbf{Indirect, environmental transmission.}  
	The transmission of pathogens through environmental reservoirs stands in contrast to direct host-to-host transmission which occurs with measles and influenza. The pathogens \emph{Vibrio cholerae} and enteric viruses serve as examples of such pathogens. The infection rate of susceptible people depends on the local pathogen concentration $C$ in the fluid through the function $\beta(C)$. The dose-response relationship in environmental epidemiology shows that infection risk grows with the strength of exposure according to this model.
	
	\item \textbf{Pathogen shedding by infected hosts.}  
	The environmental pathogen pool receives contributions from infected hosts through their pathogen release at a rate $\alpha$. 
	
	\item \textbf{Environmental persistence and decay.}  
	The external environment causes pathogens to experience mortality and inactivation and removal through natural processes including predation and chemical breakdown and sedimentation. The model includes exponential decay at rate $\lambda$ to represent the typical duration pathogens survive in fluid environments. 
	
	\item \textbf{Host mobility.}  
	The host population consists of three compartments S, I and R which move randomly throughout the habitat $\Omega$. The movement of hosts through their habitat occurs because of their natural behavior and social activities and water-seeking behavior. The model uses diffusion terms to describe host movement at a population density level instead of following specific individual paths.
	
	\item \textbf{Fluid-mediated transport.}  
	The model represents the surrounding fluid as an incompressible substance which moves according to the velocity field $\boldsymbol{U}$. The  density of hosts and pathogen concentration $C$ experience movement through this fluid flow. The model demonstrates how water currents and air ventilation systems distribute pathogens while determining the spread of infectious diseases.
	
	\item \label{a6} \textbf{Pathogen–fluid feedback.}  
	The model introduces a new mechanism which enables epidemiological processes to influence hydrodynamic operations and vice versa. The fluid viscosity $\nu(C)$ depends on pathogen concentration because microbial growth and organic matter accumulation modify water and mucus-like environments. 
\end{enumerate}

By combining Assumptions \ref{a1}–\ref{a6}, we obtain a new epidemic-pathogen-fluid system in which host dynamics, pathogen transport, and fluid motion are tightly interconnected. In what follows, we mathematically incorporate the given assumptions by relying on parabolic partial differential equations.

	\subsection{{Fluid dynamics}}
	
	\noindent
	The velocity field $\boldsymbol{U}$ of the surrounding medium is described by the following incompressible Navier–Stokes equation:

	\begin{equation}
		\begin{cases}
			\begin{split}
				&\partial_t \boldsymbol{U} +(\boldsymbol{U}\cdot\nabla)\boldsymbol{U} = -\nabla p + \operatorname{div}(\nu(C)\nabla\boldsymbol{U} )+ \boldsymbol{f}, &&\text{ in } Q_T:=\Omega \times (0,T), \\
				&\operatorname{div}(\boldsymbol{U})=0, &&\text{in } Q_T,\\
				&\boldsymbol{U}=0, && \text{ on } \Sigma_T:=\partial \Omega \times (0,T),\\
				&\boldsymbol{U}(.,0)=\boldsymbol{U}_0,  && \text{ in } \Omega,
			\end{split}
		\end{cases}
	\end{equation}
	where $\rho$ denotes the fluid density, $\nu$ its viscosity, and $\boldsymbol{f}$ represents external forces (e.g. gravity). The velocity field affects both the transport of pathogens and the spatial redistribution of individuals. In addition, we allow $\nu$ to depend on the pathogen concentration, thereby introducing a feedback from epidemiological to hydrodynamical dynamics.
	
	\subsection{{Pathogen concentration dynamics}}
	\noindent
	The concentration of pathogens in the environment is described by
	\begin{equation}
		\begin{cases}
			\begin{split}
				&\partial_t C + \boldsymbol{U}\cdot\nabla C = D_C\Delta C + \alpha I(x,t)(x) - \lambda C, &&\text{in } Q_T,\\
				&C=0, && \text{on } \Sigma_T,\\
				&C(.,0)=C_0, && \text{in } \Omega.
			\end{split}
		\end{cases}
	\end{equation}
	Here, $D_C$ stands for the diffusion coefficient of the pathogen, $\alpha$ the rate at which infected individuals shed pathogens into the environment. The last term accounts for pathogen clearance and natural decay at rate $\lambda$. The advection term $\boldsymbol{U}\cdot\nabla C$ reflects the transport of pathogens by the fluid flow.
	
	\subsection{{Host dynamics}}
	\noindent
	The host population is divided into three compartments, as follows:  The susceptible $S$, the infected $I$, and the recovered $R$. Each compartment evolves according to a reaction–diffusion–advection equation:
	\begin{equation}
		\begin{cases}
			\begin{split}
				&\partial_t S= D_S\Delta S +\Lambda- \beta(C)\frac{SI}{N}-\eta S, &&\text{in } Q_T,\\
				&\partial_t I= D_I\Delta I + \beta(C)\frac{SI}{N} - \gamma I-\eta I, &&\text{in } Q_T,\\
				&\partial_t R= D_R\Delta R + \gamma I-\eta R, &&\text{in } Q_T,\\
				&\nabla S \cdot \overrightarrow{n}=\nabla I \cdot \overrightarrow{n}=\nabla R \cdot \overrightarrow{n}=0, &&\text{on } \Sigma_T,\\
				&(S(.,0),I(.,0),R(.,0))=(S_0,I_0,R_0), &&\text{in } \Omega,
			\end{split}
		\end{cases}
	\end{equation}
	where $D_S$, $D_I$, and $D_R$ denote the diffusion coefficients of the respective compartments, $N=S+I+R$ is the total population, $\Lambda$ is the birth rate, $\eta$ is the death rate, $\beta(C)$ is the transmission rate modulated by the local pathogen concentration $C$, and $\gamma$ is the recovery rate. The term $\beta(C)\frac{SI}{N}$ models new infections, while $\gamma I$ accounts for recovery transitions.
	
	\subsection{{The resulting coupled system}}
	\noindent
	Collecting the above equations, the full SIR--Pathogen--Navier-Stokes (SIRPNS) model reads as follows:
	\begin{equation}\label{SIRCNS}
		\begin{cases}
			\begin{split}
				&\partial_t \boldsymbol{U} +(\boldsymbol{U}\cdot\nabla)\boldsymbol{U} = -\nabla p + \operatorname{div}(\nu(C)\nabla\boldsymbol{U} ) + \boldsymbol{f}, &&\text{in } Q_T,\\
				&\operatorname{div}(\boldsymbol{U})=0, &&\text{in } Q_T,\\
				&\partial_t C + \boldsymbol{U}\cdot\nabla C= D_C\Delta C + \alpha I  - \lambda C, &&\text{in } Q_T,\\
				&\partial_t S= D_S\Delta S - \beta(C)\frac{S I}{N}+\Lambda - \eta S, &&\text{in } Q_T,\\
				&\partial_t I= D_I\Delta I + \beta(C)\frac{S I}{N} - \gamma I - \eta I, &&\text{in } Q_T,\\
				&\partial_t R= D_R\Delta R + \gamma I- \eta R, &&\text{in } Q_T,\\
				&\boldsymbol{U}=C=\nabla S \cdot \overrightarrow{n}=\nabla I \cdot \overrightarrow{n}=\nabla R \cdot \overrightarrow{n}=0,  && \text{on } \Sigma_T,\\
				&(\boldsymbol{U}(.,0),C(.,0),S(.,0),I(.,0),R(.,0))=(\boldsymbol{U}_0,C_0,S_0,I_0,R_0) ,  && \text{in }\Omega.
			\end{split}
		\end{cases}
	\end{equation}

	\noindent
	In the aim of illustrating the interactions captured by our model,  Figure \ref{fig:model_schematic} gives a diagram of the coupled dynamics and summarizes how the host compartments (susceptible, infected, recovered), the environmental pathogen concentration, and the fluid flow are interconnected through both direct biological processes and physical transport. 
\begin{figure}[H]
	\centering
	\begin{tikzpicture}[
		node distance = 3.5cm and 2.5cm,
		box/.style = {draw, rounded corners, rectangle, align=center, 
			minimum width=3.2cm, minimum height=1.6cm, 
			text width=2.8cm, font=\small, line width=0.7pt,
			inner sep=6pt},
		arrow/.style = {->, >=stealth, line width=0.8pt},
		label/.style = {font=\footnotesize, align=center, inner sep=3pt}
		]
		
		\node (host) [box, fill=blue!10] at (-2,4) {Host Population \\ {\footnotesize $S(x,t)$, $I(x,t)$, $R(x,t)$}};
		\node (pathogen) [box, fill=red!10] at (2,-1) {Pathogen Concentration \\ {\footnotesize $C(x,t)$}};
		\node (fluid) [box, fill=green!10] at (6.7,6.6) {Fluid Dynamics \\ {\footnotesize Velocity $\boldsymbol{U}(x,t)$} \\ {\footnotesize Viscosity $\nu(C)$}};
		
		\draw [arrow, blue!80!black] ([yshift=5pt] host.south) 
		to [bend left=20] 
		node [right, pos=0.7, align=left, xshift=8pt, label] {\footnotesize Shedding\\{\footnotesize $\alpha I$}} 
		([yshift=5pt] pathogen.north);
		
		\draw [arrow, blue!60!black]
		(host.north west) to [bend left=40]
		node [above left, align=right, label] {\hspace{-2.9cm} \footnotesize Natural death\\\hspace{-2.9cm} {\footnotesize $\eta$}}
		++(-2,1.2);
		
		\coordinate (birthNode) at (-2,8);
		\draw [arrow, blue!60!black]
		(birthNode) to [bend right=20]
		node [above, align=center, yshift=3pt, label] {\footnotesize \hspace{-2cm} Natural birth\\\hspace{-2cm} {\footnotesize $\Lambda$}}
		(host.north);
		
		\draw [arrow, red!80!black] ([yshift=-5pt] pathogen.north) 
		to [bend left=20] 
		node [left, pos=0.7, align=right, xshift=-8pt, label] {\footnotesize Infection\\{\footnotesize $\dfrac{\beta(C)SI}{N}$}} 
		([yshift=-5pt] host.south);
		
		\draw [arrow, red!80!black, line width=1pt] (pathogen.east) 
		to 
		node [above, pos=0.5, align=center, yshift=5pt,xshift=-40pt, label] {\footnotesize Modulates\\{\footnotesize Viscosity $\nu(C)$}} 
		(fluid.west);
		
		\draw [arrow, green!80!black, dashed] ([xshift=90pt,yshift=-22pt] fluid.west) 
		to [bend left=120] 
		node [below, pos=0.6, align=center,xshift=0.25cm ,yshift=-18pt, label] {\footnotesize Advection\\{\footnotesize $\boldsymbol{U} \cdot \nabla C$}} 
		([yshift=-8pt] pathogen.east);
		
		\draw [arrow, green!80!black, dashed] ([yshift=8pt,xshift=-1pt] fluid.west) 
		to [bend right=60] 
		node [above, pos=0.6, align=center, yshift=10pt,xshift=-19pt, label] {\footnotesize Host Transport\\{\footnotesize (Advection)}} 
		([xshift=2pt] host.east);
		
		\draw [arrow, blue!60!black] (host.west) 
		to [bend right=25] 
		node [left, align=right, xshift=-10pt, label] {\footnotesize Recovery\\{\footnotesize $\gamma I$}} 
		++(-1.5,-0.6);
		
		\draw [arrow, red!60!black] (pathogen.west) 
		to [bend left=25] 
		node [left, align=right, xshift=-10pt, label] {\footnotesize Decay\\{\footnotesize $-\lambda C$}} 
		++(-1.5,0.6);
		
		\draw [arrow, green!60!black] (fluid.east) 
		to [out=10, in=-10, looseness=1.5] 
		node [right, align=left, xshift=-2pt, label] {\footnotesize Self-advection\\{\footnotesize $(\boldsymbol{U}\cdot\nabla)\boldsymbol{U}$}} 
		([yshift=12pt] fluid.east);
		
	\end{tikzpicture}
	\caption{\textbf{Schematic diagram of the SIR–Pathogen–Navier–Stokes (SIRPNS) model couplings.} The diagram illustrates the two-way coupling between epidemiological dynamics (blue), pathogen transport (red), and fluid mechanics (green).}
	\label{fig:model_schematic}
\end{figure}
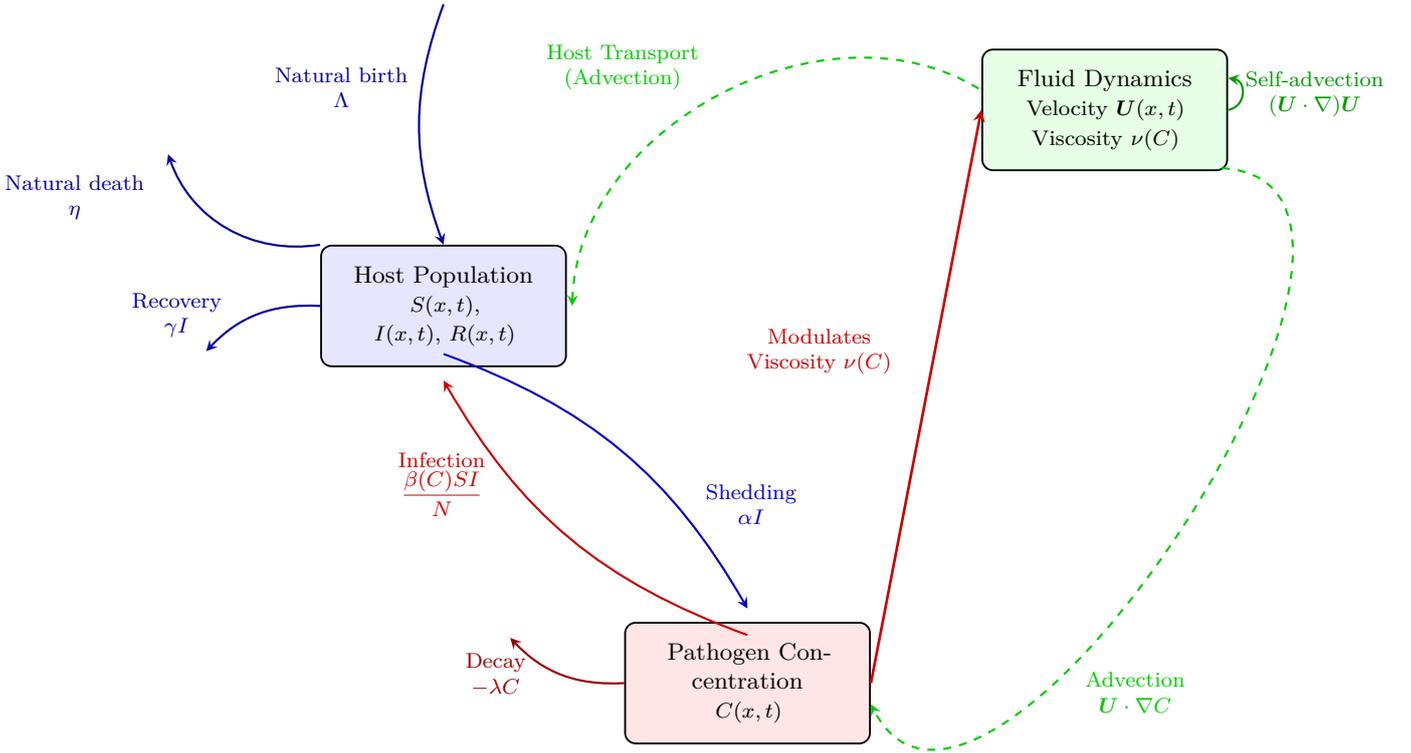

\section{Mathematical analysis of the SIRPNS model}\label{analysis}

\subsection{Functional framework and assumptions}
Before addressing the well-posedness of the SIRPNS model in the mathematical sense, we first establish its functional framework. We begin by denoting $L^p(\Omega)$ and $W^{m,p}(\Omega)$ as the standard Lebesgue and Sobolev spaces, respectively, which are defined on the open bounded set $\Omega$ for $1 \leq p \leq +\infty$ and $m \in \mathbb{N}$. For consistency, we denote the norm of $L^p(\Omega)$  by $\|\cdot\|_{L^p(\Omega)}$ for $1 \leq p \leq +\infty$. If $p = 2$, we employ the usual notation $H^m(\Omega) := W^{m,p}(\Omega)$. Given a Banach space $X$, $X'$ stands for the dual space of $X$, and $\norm{\cdot}_X$ refers to the norm in $X$ and any power of $X$. For all $1 \leq p \leq +\infty$, $L^p(0,T;X)$ represents the space of all measurable functions $u : (0,T) \to X$ for which $t \mapsto \|u(t)\|_X$ belongs to $L^p(0,T)$. 
Furthermore, we define the vector-valued spaces:
	\[
	\mathbf{L}^2(\Omega) = \big(L^2(\Omega)\big)^d, \quad 
	\mathbf{H}^1(\Omega) = \big(H^1(\Omega)\big)^d, \quad 
	\mathbf{H}_0^1(\Omega) = \big(H_0^1(\Omega)\big)^d.
	\]
Finally, we introduce the spaces $\mathcal{W}$, $\mathbf{H}$, and $\mathbf{V}$ as follows:
	\[
	\mathcal{W}:= \{\boldsymbol{U}\in \mathcal{D}(\Omega) \mid \operatorname{div} \boldsymbol{U} = 0\}, \quad 
	\mathbf{H}:= \overline{\mathcal{W}}^{\mathbf{H}^1_0(\Omega)}, \quad 
	\mathbf{V}:= \overline{\mathcal{W}}^{\mathbf{L}^2(\Omega)}.
	\]
	
	\noindent  We now state some preliminary results that will be needed in subsequent sections. The following well-known Gagliardo-Nirenberg inequality is essential (see~\cite{ladyvzenskaja1968linear}).
	\begin{align}
		\|\boldsymbol{Z}(t)\|_{\mathbf{L}^{4}(\Omega_b)} &\leq c\|\boldsymbol{Z}(t)\|_{\mathbf{H}^{1}(\Omega_b)}^{\zeta}\|\boldsymbol{Z}(t)\|_{\mathbf{L}^{2}(\Omega_b)}^{1-\zeta}, && \text{for all } \boldsymbol{Z} \in \mathbf{H}^{1}(\Omega_b), \label{Gagliardo-Nirenberg-U} \\
		\|\psi(t)\|_{L^{4}(\Omega)} &\leq c\|\psi(t)\|_{H^{1}(\Omega)}^{\zeta}\|\psi(t)\|_{L^{2}(\Omega)}^{1-\zeta}, && \text{for all } \psi \in H^{1}(\Omega) \label{Gagliardo-Nirenberg-C},
	\end{align}
	where $\zeta = d/4$.
	 On the other hand, our analysis will occasionally build Young's inequality with small parameter $\delta>0$:
	\begin{equation}\label{inq-Young}
		ab \leq \delta a^{p} + c(\delta) b^{q},
	\end{equation}
	with $a, b > 0$, $\delta > 0$, $1 < p, q < \infty$, $1/p + 1/q = 1$, and $c(\delta) = (\delta p)^{-q/p} q^{-1}$.

The last result that we recall in Aubin-Lions compactness lemma \cite{simon1986compact}, which reads as follows.
	
\begin{lem}\label{lemma-Aubin-Lions} Let $X_0, X$, and $X_1$ be three Banach spaces with $X_0 \subset X \subset X_1$. Suppose that $X_0$ is compactly embedded in $X$ and that $X$ is continuously embedded in $X_1$. Then:
	\begin{enumerate}
		\item 
	 If $G$ is bounded in $L^p\left(0, T ; X_0\right)$ for $1 \leq p<\infty$, and $\frac{\partial G}{\partial t}$ is bounded in $L^1\left(0, T ; X_1\right)$, then $G$ is relatively compact in $L^p(0, T ; X)$.
	\item If $F$ be bounded in $L^{\infty}\left(0, T ; X_0\right)$ and $\frac{\partial F}{\partial t}$ is bounded in $L^p\left(0, T ; X_1\right)$ with $p>1$, then $F$ is relatively compact in $C(0, T ; X)$.
	\end{enumerate}		
	\end{lem}
	
%
	
To prove the existence of weak solutions, we impose some assumptions in regards to the parameters of the model (\ref{SIRCNS}).
\begin{enumerate}
			\item[\textbf{(A1).}] The functions  $\nu=\nu(\cdot)$ and $\beta=\beta(\cdot)$ are continuous such that
			\begin{eqnarray}
				0<\nu_{1} \leq \nu(s) \leq \nu_{2} \qquad & \forall s\in \mathbb{R},\label{cond-mu}\\	
				0<\beta_{1} \leq \beta(s) \leq \beta_{2} \qquad & \forall s\in \mathbb{R}, \label{cond-beta}
			\end{eqnarray}
			where $\nu_1$, $\nu_2$, $\beta_1$ and  $\beta_2$  are positive constants.
			\item[\textbf{(A2).}]  $ \boldsymbol{U}_{0} \in \mathbf{L}^{2}(\Omega)$ and $ (C_0,S_0,I_0,R_0) \in (L^{2}(\Omega))^4$, and it holds that $C_0,S_0,I_0,R_0\geq0$.
			\item[\textbf{(A3).}] $\boldsymbol{f}\in L^2(0,T;\mathbf{L}^2(\Omega))$
\end{enumerate}

	\begin{defi}\label{defSol} A quintet  $(\boldsymbol{U},C,S,I,R)$ is sait to be a weak solution to Model (\ref{SIRCNS}), if it satisfies the following assertions: 
		\begin{equation*}
			\begin{split}
				& \boldsymbol{U} \in L^\infty(0,T; \mathbf{V})\cap L^2(0,T; \mathbf{H}),\;\;\partial_t \boldsymbol{U} \in L^2(0,T;(\mathbf{H})^\prime),
				\\
				&C \in L^\infty(0,T; L^2(\Omega))\cap L^2(0,T;  H^{1}_{0}(\Omega)),\;\;\partial_t C\in L^2(0,T;( H^{1}_{0}(\Omega))^\prime), 
				\\
				&S \in  L^\infty(0,T; L^2(\Omega))\cap L^2(0,T;  H^{1}_{0}(\Omega)),\;\;\partial_t S\in L^2(0,T;( H^{1}_{0}(\Omega))^\prime), 
				\\
				&I \in  L^\infty(0,T; L^2(\Omega))\cap L^2(0,T;  H^{1}_{0}(\Omega)),\;\;\partial_t I\in L^2(0,T;( H^{1}_{0}(\Omega))^\prime),
				\\
				&R \in  L^\infty(0,T; L^2(\Omega))\cap L^2(0,T;  H^{1}_{0}(\Omega)),\;\;\partial_t R\in L^2(0,T;( H^{1}_{0}(\Omega))^\prime)
			\end{split}
		\end{equation*}
		and the following identities hold
		\begin{eqnarray}\label{eqvar1}
			\begin{split}
				\int_{0}^{T}		\left\langle\partial_t\boldsymbol{U}, \boldsymbol{Z}\right\rangle\, dt+\iint_{Q_T}	\nu(C) \nabla \boldsymbol{U} : \nabla \boldsymbol{Z}\, d\mathbf{x}\, dt+\iint_{Q_T}	(\boldsymbol{U}\cdot\nabla)\boldsymbol{U} \cdot  \boldsymbol{Z}\, d\mathbf{x}\, dt\\
				-\iint_{Q_T}\boldsymbol{f}\cdot \boldsymbol{Z}\, d\mathbf{x}\, dt&=0,
				\\  
				\int_{0}^{T}	\left\langle\partial_tC, \psi_C \right\rangle\, dt+ \iint_{Q_T}  \boldsymbol{U}\cdot\nabla C\, \psi_C\, d\mathbf{x}\, dt+ D_C\iint_{Q_T}\nabla C \cdot\nabla \psi_C\, d\mathbf{x}\, dt\\ - \alpha\iint_{Q_T} I\, \psi_C\, d\mathbf{x}\, dt + \lambda \iint_{Q_T} C \,\psi_C \, d\mathbf{x}\, dt
				&=0,
				\\ 	  		
				\int_{0}^{T}	\left\langle\partial_t S, \psi_S \right\rangle\, dt+ D_S\iint_{Q_T}\nabla S \cdot\nabla \psi_S\, d\mathbf{x}\, dt+\iint_{Q_T}\beta(C) \frac{SI}{N} \, \psi_S\, d\mathbf{x}\, dt\\-\iint_{Q_T} \Lambda \,\psi_S \, d\mathbf{x}\, dt + \eta \iint_{Q_T} S \,\psi_S \, d\mathbf{x}\, dt
				&=0,
				\\ 	  		
				\int_{0}^{T}	\left\langle\partial_t I, \psi_I \right\rangle\,  dt+ D_I\iint_{Q_T}\nabla I \cdot\nabla \psi_I\, d\mathbf{x}\, dt-\iint_{Q_T}\beta(C) \frac{SI}{N} \, \psi_I\, d\mathbf{x}\, dt + (\eta+\gamma) \iint_{Q_T} I \,\psi_I \, d\mathbf{x}\, dt
				&=0,
				\\ 	  		
			\int_{0}^{T}	\left\langle\partial_t R, \psi_R \right\rangle\,  dt+ D_R\iint_{Q_T}\nabla R \cdot\nabla \psi_R\, d\mathbf{x}\, dt-\gamma\iint_{Q_T} I \,\psi_R \, d\mathbf{x}\, dt+ \eta\iint_{Q_T} R \,\psi_R \, d\mathbf{x}\, dt
				&=0,
			\end{split} 
		\end{eqnarray}   
		for all test functions $\boldsymbol{Z} \in \mathbf{L}^2(0,T; \mathbf{H})$, $\psi_C \in L^2(0,T;  H^1_0)$, $\psi_S \in L^2(0,T;  H^1_0)$, $\psi_I \in L^2(0,T;  H^1_0)$ 
		and $\psi_R \in L^2(0,T;  H^1_0)$, and
		\begin{eqnarray}\label{eqvar2}
			\begin{split}
				\boldsymbol{U}(0,\mathbf{.}) &=\boldsymbol{U}_{0},  & \hbox { in }   \Omega, \\
				C(0,\mathbf{.}) &=C_{0},  & \hbox { in }   \Omega,
				\\
				S(0,\mathbf{.}) &=S_{0},  & \hbox { in }   \Omega, \\
				I(0,\mathbf{.}) &=I_{0},  & \hbox { in }   \Omega, \\
				R(0,\mathbf{.}) &=R_{0},  & \hbox { in }   \Omega,  
			\end{split}
		\end{eqnarray}
	\end{defi}

	\subsection{Existence of weak solutions}
	
	Our first main result pertains to the existence of weak solutions to Model (\ref{SIRCNS}) and is stated as follows.
	\begin{thm}\label{theo-existence}
		Assume the assumptions {\bf(A1)}-{\bf(A3)}  hold. 
		Then the problem (\ref{SIRCNS}) has a weak solution in the sense of Definition \ref{defSol}.
	\end{thm}
	Let us briefly outline the steps of the proof of our first main result, which we will explore in details in the upcoming subsubsections:
	\begin{enumerate}
		\item \textbf{Construction of a sequence of Faedo-Galerkin solutions}. 
		In this step, we formulate a finite dimensional counterpart (of dimension $n \in \mathbb{N}^*$)  of Model (\ref{SIRCNS}), by projecting the solution into a finite-dimensional space that is constructed from an orthonormal basis. This step is concluded by establishing that  a sequence $(\boldsymbol{U}^n,C^n,S^n,I^n,R^n)$ of solutions to the finite dimensional counterpart exists.
		\item \textbf{A priori estimates of the sequence of Faedo-Galerkin solutions}.  In this part, we  establish several estimates for the sequence $(\boldsymbol{U}^n,C^n,S^n,I^n,R^n)$. These allow us to conclude that the weak solution is bounded in the appropriate chosen Banach spaces.
		\item \textbf{Passage to the limit}. In this step, we gather all the results obtained in the previous steps to let $n\rightarrow +\infty$ and recover a weak solution to the model (\ref{SIRCNS}). 
	\end{enumerate}
	
	\subsubsection{Construction of a sequence of Faedo-Galerkin solutions}
	Let, $\left\{ (\xi_k, \ell_k) \right\}_{k=1,2, \ldots} \subset \boldsymbol{H} \times H^1_0$ be an orthonormal basis of $\boldsymbol{V} \times H_0^1$ and let $\boldsymbol{\Pi}_n$ and  $\Pi_n$ denote the orthogonal projection operators from $\boldsymbol{L}^2(\Omega)$ and $L^2(\Omega)$, respectively
	(endowed with the usual inner product $\langle \cdot , \cdot \rangle$) onto the finite-dimensional subspaces
	\[
	\begin{aligned}
		\boldsymbol{X}^{n}  &:= \operatorname{span}\{\xi_1, \dots, \xi_n\},\\
		Y^n &:= \operatorname{span}\{\ell_1, \dots, \ell_n\}.
	\end{aligned}
	\]
	For every $\boldsymbol{\phi} \in \boldsymbol{L}^2(\Omega)$,  $\phi \in L^2(\Omega)$, and $n \in \mathbb{N}^*$, this projection is given by
	\[
	\boldsymbol{\Pi}_n \boldsymbol{\phi} (t,\mathbf{x}) = \sum_{k=1}^n \langle  \boldsymbol{\phi} , \xi_k \rangle(t) \xi_k\l(\mathbf{x}\r) \text{ and } \Pi_n \phi (t,\mathbf{x}) = \sum_{k=1}^n \langle \phi , \ell_k \rangle(t) \ell_k\l(\mathbf{x}\r).
	\]
	
	\medskip
	\noindent
	We now construct the \emph{Faedo--Galerkin approximations}. 
	Fix $n \in \mathbb{N}^*$ and consider an approximate solution of the form
	\[
	\boldsymbol{U}^n(t) := \sum_{k=1}^n \boldsymbol{u}_k^n(t) \, \xi_k, \; C^n(t) := \sum_{k=1}^n c_k^n(t) \, \ell_k, \;  S^n(t) := \sum_{k=1}^n d_k^n(t) \, \ell_k, \; I^n(t) := \sum_{k=1}^n e_k^n(t) \, \ell_k, \; R^n(t) := \sum_{k=1}^n f_k^n(t) \, \ell_k,
	\]
	where $\{\boldsymbol{u}_k^n\}_{k=1}^n$, $\{c_k^n\}_{k=1}^n$, $\{d_k^n\}_{k=1}^n$, $\{e_k^n\}_{k=1}^n$ and $\{f_k^n\}_{k=1}^n$ are scalar functions yet to be determined. The initial condition is projected accordingly as
	\[
	\boldsymbol{U}^n_0=\sum_{\ell=0}^{n}\langle \boldsymbol{U}_0, \xi_{\ell} \rangle\xi_\ell, \; C^n_0 := \sum_{k=1}^n \langle C_0, \ell_k \rangle \ell_k, \;  S^n_0 := \sum_{k=1}^n \langle S_0, \ell_k \rangle \ell_k, \; I^n_0 := \sum_{k=1}^n \langle I_0, \ell_k \rangle \ell_k, \; R^n_0 := \sum_{k=1}^n \langle R_0, \ell_k \rangle \ell_k.
	\]
	In addition, the second member $\boldsymbol{f}$ is approximated within the same finite-dimensional space of $\boldsymbol{U}_0^n$:
	\[
	\boldsymbol{f}^n\l(t,x\r)=\sum_{k=0}^{n}\langle\boldsymbol{f},\xi_k\rangle (t)\xi_k\l(\mathbf{x}\r). \; 
	\]
	The functions  $\{\boldsymbol{u}_k^n\}_{k=1}^n$, $\{c_k^n\}_{k=1}^n$, $\{d_k^n\}_{k=1}^n$, $\{e_k^n\}_{k=1}^n$ and $\{f_k^n\}_{k=1}^n$ are determined so that, for each 
	$k \in \{1,\dots,n\}$ and every $n \in \mathbb{N}^*$, the following system of equations is satisfied:
	\begin{equation}\label{eq:approx-eqn-integrated}
		\begin{split}
			\frac{d}{dt} \boldsymbol{U}^n(t)&=
			 \boldsymbol{\Pi}_n \left[\nabla \cdot \left(\nu(C) \nabla \boldsymbol{U} \right) \right] - \boldsymbol{\Pi}_n \left[(\boldsymbol{U}\cdot\nabla)\boldsymbol{U} \right]
			+ \boldsymbol{\Pi}_n[\boldsymbol{f}]
			,
\\
			\frac{d}{dt} C^n(t) &=
			\Pi_n \left[\nabla \cdot \left(D_C \nabla C \right) \right]- \Pi_n \left[\boldsymbol{U} \cdot  \nabla C  \right]+\Pi_n\left[\alpha I-\lambda C\right]
		,
\\
			  \frac{d}{dt} S^n(t)&=
			  \Pi_n \left[\nabla \cdot \left(D_S \nabla S \right) \right] -\Pi_n \left[ \beta(C)\frac{SI}{N} \right]
               + \Pi_n\left[\Lambda-\eta S\right]	
			,
\\ 
			 \frac{d}{dt} I^n(t)&= \Pi_n \left[\nabla \cdot \left(D_I \nabla I \right) \right] +  \Pi_n \left[ \beta(C)\frac{SI}{N} \right]
			- \Pi_n\left[(\gamma+\eta)I \right]	
			,
\\ 
			  \frac{d}{dt} R^n(t)&= \Pi_n \left[\nabla \cdot \left(D_R \nabla R \right) \right]
			  + \Pi_n\left[\gamma I-\eta R\right]	
		,
		\end{split}
	\end{equation}
	with initial data
	\begin{equation}\label{eq:initial-data-approx-eqn}
		\begin{split}
			\boldsymbol{U}^n(0)	&= \boldsymbol{U}_{0}^n,\\
			C^n(0) &= C_{0}^n,\\
			S^n(0) &=S_0^n,\\
			I^n(0) &= I_{0}^n,\\
			R^n(0) &= R_{0}^n.
		\end{split}
	\end{equation}
	Let $\mathcal{U}^n:=(\boldsymbol{U}^n,C^n,S^n,I^n,R^n)$, then we can rewrite (\ref{eq:approx-eqn-integrated})-(\ref{eq:initial-data-approx-eqn}) more explicitly as a compact system of ordinary differential equations:
	\begin{equation}\label{Cauchy-problem}
		\begin{split}
			\begin{cases}
				\frac{d}{dt}\mathcal{U}^n&=\mathcal{F}^n(\mathcal{U}),\\
				\mathcal{U}^n(0)&=\mathcal{U}_0,
			\end{cases}
		\end{split}
	\end{equation}
	where  $\mathcal{U}_0=(\boldsymbol{U}_{0}^n,C_{0}^n,S_{0}^n,I_{0}^n,R_{0}^n,)$ and 
	$$\mathcal{F}^n(\mathcal{U})
	=\begin{pmatrix}
		\boldsymbol{\Pi}_n \left[\nabla \cdot \left(\nu(C) \nabla \boldsymbol{U} \right) \right]-\boldsymbol{\Pi}_n \left[(\boldsymbol{U}\cdot\nabla)\boldsymbol{U} \right]+ 	\boldsymbol{\Pi}_n\left[\boldsymbol{f}\right]
		\\
		\Pi_n \left[\nabla \cdot \left(D_C \nabla C \right) \right]-\Pi_n \left[\boldsymbol{U}\cdot\nabla C \right]+	\Pi_n\left[\alpha I-\lambda C\right]
		\\
		\Pi_n \left[\nabla \cdot \left(D_S \nabla S \right) \right]-\Pi_n \left[ \beta(C)\frac{SI}{N} \right]+	\Pi_n\left[\Lambda -\eta S\right]
		\\
		\Pi_n \left[\nabla \cdot \left(D_I \nabla I \right) \right]+\Pi_n \left[ \beta(C)\frac{SI}{N} \right]+	\Pi_n\left[(\gamma +\eta) I\right]
		\\
		\Pi_n \left[\nabla \cdot \left(D_R \nabla R \right) \right]+	\Pi_n\left[\gamma I-\eta R\right]
	\end{pmatrix}.
	$$
\noindent
	From the assumptions on the data of the model, the functions $\mathcal{F}^n$ is Caratheodory functions. Therefore, according to the standard theory of ordinary differential equation, there exists an absolutely continuous solution $\left\{\mathcal{U}^n_{k}\right\}_{k=1}^{n}$ satisfying the above equation. Consequently, a weak local solution exists for all $t\in (0,t_0)$ with $0<t_0<T$. Moreover, $\mathcal{U}^n=(\boldsymbol{U}^n,C^n,S^n,I^n,R^n)$
	satisfies the following weak formulation:
	\begin{eqnarray}\label{eq:var-approched}
			\begin{split}
			\left\langle\partial_t\boldsymbol{U}^n, \boldsymbol{Z}\right\rangle+\int_{\Omega}	\nu(C^n) \nabla \boldsymbol{U}^n : \nabla \boldsymbol{Z}\, d\mathbf{x}+\int_{\Omega}	(\boldsymbol{U}^n\cdot\nabla)\boldsymbol{U}^n \cdot  \boldsymbol{Z}\, d\mathbf{x}\, d\mathbf{x}
	-\int_{\Omega}\boldsymbol{f}\cdot \boldsymbol{Z}\, d\mathbf{x}&=0,
	\\  
		\left\langle\partial_tC^n, \psi_C \right\rangle+ \int_{\Omega}  \boldsymbol{U}^n\cdot\nabla C^n\, \psi_C\, d\mathbf{x}+ D_C\int_{\Omega}\nabla C^n \cdot\nabla \psi_C\, d\mathbf{x}
		- \alpha\int_{\Omega} I^n\, \psi_C\, d\mathbf{x} + \lambda \int_{\Omega} C^n \,\psi_C \, d\mathbf{x}
	&=0,
	\\ 	  		
		\left\langle\partial_t S^n, \psi_S \right\rangle+ D_S\int_{\Omega}\nabla S^n \cdot\nabla \psi_S\, d\mathbf{x}+\int_{\Omega}\beta(C^n) \frac{S^n\,I^n}{N_n} \, \psi_S\, d\mathbf{x}
		-\int_{\Omega} \Lambda \,\psi_S \, d\mathbf{x} + \eta \int_{\Omega} S^n \,\psi_S \, d\mathbf{x}
	&=0,
	\\ 	  		
		\left\langle\partial_t I^n, \psi_I \right\rangle+ D_I\int_{\Omega}\nabla I^n \cdot\nabla \psi_I\, d\mathbf{x}-\int_{\Omega}\beta(C^n) \frac{S^n\,I^n}{N_n} \, \psi_I\, d\mathbf{x} + (\eta+\gamma) \int_{\Omega} I^n \,\psi_I \, d\mathbf{x}
	&=0,
	\\ 	  		
		\left\langle\partial_t R^n, \psi_R \right\rangle+ D_R\int_{\Omega}\nabla R^n \cdot\nabla \psi_R\, d\mathbf{x}-\gamma\int_{\Omega} I^n \,\psi_R \, d\mathbf{x}+ \eta\int_{\Omega} R^n \,\psi_R \, d\mathbf{x}
	&=0,
\end{split} 
	\end{eqnarray}   
	for all test functions $\boldsymbol{Z} \in D([0,T); \mathbf{H})$, $\psi_C \in D([0,T);  H^1_0)$, $\psi_S \in D([0,T);  H^1_0)$, $\psi_I \in D([0,T);  H^1_0)$ 
	and $\psi_R \in D([0,T);  H^1_0)$.
	
	\subsubsection{A priori estimates of the sequence of Faedo-Galerkin solutions}
	Now, we derive $n$-independent a priori estimates of the sequence of Faedo-Galerkin solutions $(\boldsymbol{U}^n,C^n,S^n,I^n,R^n)$. The first estimate reads:
	\begin{lem}\label{lem-1} 
		The sequence  $(\boldsymbol{U}^n)_n$  satisfies 
		\begin{itemize}
			\item[(i)]\quad $(\boldsymbol{U}^n)_n$  bounded in
			$L^2(0,T;\mathbf{H})\cap L^\infty(0,T;\mathbf{V})$.
			\item[(ii)]\quad 	$(\partial_t \boldsymbol{U}^n)_n$  uniformly bounded in $L^2(0,T; \mathbf{H}^\prime)$.
			\item[(iii)]\quad $(\boldsymbol{U}^n)_n$ relatively compact in $(L^2(Q_T))^2$.		
		\end{itemize}
	\end{lem}
	\begin{proof}
	The proof is an adaptation of the one presented in \cite[Lemma 3]{bendahmane2025mathematical}, and so we omit it here for brevity.
	\end{proof}
	
	Hereinafter, we use $c$, $c_1$, $c_2$, $c_3$, $\ldots$ to denote unspecified positive constants that may change from line to line but do not depend on $n$.
    For the sequence $(C^n,S^n,I^n,R^n)$, we derive the following lemma:
	\begin{lem}\label{lem-2}
		If $C_0$, $S_0$, $I_0$, $R_0$ are nonnegative,  then the sequence $(C^n,S^n,I^n,R^n)$ is
		\begin{itemize}
			\item[(i)] nonnegative;
			\item[(ii)] bounded in
			$(L^2(0,T;H^1(\Omega))\cap L^\infty(0,T;L^2(\Omega)))^4$;
			\item[(iii)] relatively compact in $(L^2(Q_T))^4$.
		\end{itemize}
	\end{lem}
	\begin{proof}
We divide the proof into several parts.

\textit{Proof of $(i)$.} 
The key idea is to replace the term $ \beta(C^n) \dfrac{S^n I^n}{C^n+S^n+I^n}
$
with
$
\dfrac{S^{n,+} I^{n,+}}{C^{n,+}+S^{n,+}+I^{n,+}},
$, the term $\alpha I^{n}$ with $\alpha I^{n,+}$; and the term $\gamma  I^{n}$ with $\gamma  I^{n,+}$
where $u^{+}$ stands for the positive part of a given function $u$. Note that these modifications does not alter the existence of the sequence of Faedo-Galerkin solutions as the same previous theoretical techniques can be used. For simplicity, we keep the same notation for this corresponding sequence. That is,  $(\bold{U}^n,C^n,S^n,I^n,R^n)$. 

\noindent
Now, we consider the particular choices of test functions in (\ref {eq:var-approched}) as follows:

$$\psi_C=-C^{n,-},\;  \psi_S=-S^{n,-},\; \psi_I=-I^{n,-}\; \text{ and } \psi_R=-R^{n,-},$$
	where $u^-$ stands for the negative part of a given function $u$.

\noindent
Indeed, by a direct computation we obtain		
		\begin{eqnarray}\label{post-1}
			\begin{array}{ll}
				&\displaystyle\frac{1}{2}\frac{d}{\,dt}\int_\Omega \abs{C^{n,-}}^2  \, d\mathbf{x}
				- \int_\Omega \boldsymbol{U}^n \cdot\nabla C^n \, C^{n,-}  \, d\mathbf{x}
				+D_C\int_\Omega \abs{\nabla C^{n,-}}^2 \, d\mathbf{x}+\lambda\int_{\Omega} \abs{C^{n,-}}^2\, d\mathbf{x}\\
				& \displaystyle\qquad +\alpha\int_\Omega I^{n,+} C^{n,-}  \, d\mathbf{x}=0,
			\end{array}
		\end{eqnarray} 
		\begin{eqnarray}\label{post-2}
			\begin{array}{ll}\displaystyle
				\displaystyle\frac{1}{2}\frac{d}{\,dt}\int_\Omega \abs{S^{n,-}}^2  \, d\mathbf{x}
				&+\displaystyle D_S\int_\Omega \abs{\nabla S^{n,-}}^2 \, d\mathbf{x}+ \displaystyle \eta\int_\Omega |S^{n,-}|^2\, d\mathbf{x}+\int_\Omega \Lambda\, S^{n,-}\, d\mathbf{x}\\
				&\displaystyle=-\int_\Omega \beta(C^n) \frac{S^{n,+} I^{n,+}}{C^{n,+}+S^{n,+}+I^{n,+}} S^{n,-}  \, d\mathbf{x},
		\end{array}\end{eqnarray} 
		\begin{eqnarray}\label{post-3}
			\begin{array}{ll}
				&\displaystyle\frac{1}{2}\frac{d}{\,dt}\int_\Omega \abs{I^{n,-}}^2  \, d\mathbf{x}
				
				+D_I\int_\Omega \abs{\nabla I^{n,-}}^2 \, d\mathbf{x}-(\eta+\gamma)\int_\Omega \abs{I^{n,-}}^2  \, d\mathbf{x}\\
				& \displaystyle\qquad \qquad \qquad \qquad=\int_\Omega \beta(C^n) \frac{S^{n,+} I^{n,+}}{C^{n,+}+S^{n,+}+I^{n,+}} I^{n,-}  \, d\mathbf{x},
		\end{array}\end{eqnarray}
	and
			\begin{eqnarray}\label{post-4}
		\begin{array}{ll}
			&\displaystyle\frac{1}{2}\frac{d}{\,dt}\int_\Omega \abs{R^{n,-}}^2  \, d\mathbf{x}
			+D_R\int_\Omega \abs{\nabla R^{n,-}}^2 \, d\mathbf{x}-\eta\int_{\Omega} \abs{R^{n,-}}^2\, d\mathbf{x}\\
			& \displaystyle\qquad +\gamma\int_\Omega I^{n,+} R^{n,-}  \, d\mathbf{x}=0.
		\end{array}
	\end{eqnarray} 
	On the other hand, we have
$$\operatorname{div}\boldsymbol{U}^n=0 \text{ in } Q_T$$ 
and 
$$\boldsymbol{U}^n=0 \text{ on } \Sigma_T.$$ 
Hence
$$
\int_\Omega \boldsymbol{U}^n \cdot\nabla C^n\, C^{n,-}  \, d\mathbf{x}
=\frac{1}{2}\int_\Omega \boldsymbol{U}^n \cdot \nabla  \abs{C^n}^2\, d\mathbf{x}
=-\frac{1}{2}\int_\Omega \div \boldsymbol{U}^n \, \abs{C^n}^2\, d\mathbf{x}
+\frac{1}{2}\int_{\partial \Omega}\abs{C^n}^2  \boldsymbol{U}^n \cdot \boldsymbol{\eta} d\sigma=0.
$$
Now, recall that for a given function $u$, it holds that
		$$
		\operatorname{supp} \cap \operatorname{supp}{u^-}=\emptyset,
		$$
		where $\operatorname{supp}{u^+}$ denotes the support of a given function $u$.
		
		\noindent
		Hence, 
		$$\displaystyle\frac{1}{2}\frac{d}{\,dt}\int_\Omega \abs{C^{n,-}}^2  \, d\mathbf{x} \leq 0,\; \displaystyle\frac{1}{2}\frac{d}{\,dt}\int_\Omega \abs{S^{n,-}}^2  \, d\mathbf{x} \leq 0;\; \displaystyle\frac{1}{2}\frac{d}{\,dt}\int_\Omega \abs{I^{n,-}}^2  \, d\mathbf{x} \leq 0, \; \text{and } \displaystyle\frac{1}{2}\frac{d}{\,dt}\int_\Omega \abs{R^{n,-}}^2  \, d\mathbf{x} \leq 0.$$
		Since the data $C_0$, $S_0$, $I_0$ and $R_0$, is nonnegative, we deduce that $$C^{n,-}=0,\;  S^{n,-}=0,\; I^{n,-}=0,\; \text{and } R^{n,-}=0 \; \; \text{in } Q_T.$$

		\textit{Proof of $(ii)$.}  
		Let $\psi_C=C^n$, $\psi_S=S^n$, $\psi_I=I^n$, $\psi_R=R^n$. From (\ref {eq:var-approched}), we obtain by using Green's formula and Cauchy-Schwarz inequality that:
		\begin{equation}\label{est-fluid1}
			\begin{array}{l}
				\displaystyle	\frac{1}{2} \frac{d}{dt}\norm{C^n(t)}_{L^2(\Omega)}^2 +D_C    \int_\Omega | \nabla C^n|^2 \, d\mathbf{x}  +\lambda \norm{C^n(t)}_{L^2(\Omega)}^2 +  \int_{\Omega} \boldsymbol{U}^n\cdot \nabla C^n \, C^n  \, d\mathbf{x}  
				\leq
				\alpha \Big|\int_{\Omega} I^n C^n  \, d\mathbf{x}\Big| \\ \\
				\displaystyle	\frac{1}{2} \frac{d}{dt}\norm{S^n(t)}_{L^2(\Omega)}^2 +D_S    \int_\Omega | \nabla S^n|^2 \, d\mathbf{x} +\eta \norm{S^n(t)}_{L^2(\Omega)}^2 +  \int_{\Omega} \Lambda C^n  \, d\mathbf{x}  
				\leq 
				\Big| \int_{\Omega} \beta(C^n) \frac{S^n I^n }{N_n} S^n\, d\mathbf{x} \Big|,\\ \\
				\displaystyle	\frac{1}{2} \frac{d}{dt}\norm{I^n(t)}_{L^2(\Omega)}^2 +D_I    \int_\Omega | \nabla I^n|^2 \, d\mathbf{x}  +(\gamma+\eta) \norm{I^n(t)}_{L^2(\Omega)}^2  \leq \Big| \int_{\Omega} \beta(C^n) \frac{S^n I^n }{N_n}I^n \, d\mathbf{x} \Big|,\\ \\
				\displaystyle	\frac{1}{2} \frac{d}{dt}\norm{R^n(t)}_{L^2(\Omega)}^2 +D_R    \int_\Omega | \nabla R^n|^2 \, d\mathbf{x}  +\eta \norm{R^n(t)}_{L^2(\Omega)}^2  \leq \Big| \int_{\Omega}  I^n R^n \, d\mathbf{x} \Big|.
			\end{array}
		\end{equation}
		Observe that, since $\operatorname{div} U_n=0$ in $\Omega_T$ and $U_n=0$ on $\Sigma_T$, we get
		\begin{equation}\label{galc}
\displaystyle \int_\Omega  \boldsymbol{U}^n\cdot\nabla C^n \, C^n \, d\mathbf{x}= \frac{1}{2}\int_\Omega\boldsymbol{U}^n\nabla ((C^n)^2) \, d\mathbf{x}=-\frac{1}{2}\int_\Omega\operatorname{div}(\boldsymbol{U}^n) \, (C^n)^2 \, d\mathbf{x}+\frac{1}{2}\int_{\partial\Omega} \boldsymbol{U}^n\cdot\boldsymbol{n}\,(C^n)^2\,ds=0.
		\end{equation}
		Using this and by virtue of Poincaré inequality, we deduce from \eqref{est-fluid1} that
		\begin{equation}\label{1} 
			\begin{array}{rl}
				\displaystyle  \frac{1}{2} \frac{d}{dt}\norm{C^n(t)}_{L^2(\Omega)}  + \displaystyle  c_1     \norm{C^n}_{H^{1}}^2  +\lambda    \norm{C^n(t)}_{L^2(\Omega)}^2
				 &\displaystyle  \leq  \alpha  \int_{\Omega} I^nC^n  \, d\mathbf{x} \\
				 &\displaystyle  \leq   
				  \alpha \norm{I^n(t)}_{L^2(\Omega)}\norm{C^n(t)}_{L^2(\Omega)}
				 \\ 
				& \displaystyle  \leq c_2  \norm{I^n(t)}_{L^2(\Omega)}^2 + c_3  \norm{C^n(t)}_{L^2(\Omega)}^2.
			\end{array}
		\end{equation}
		Additionally, we derive
		\begin{equation}\label{2} 
			\begin{array}{rl}	
				\displaystyle	\frac{1}{2} \frac{d}{dt}\norm{S^n(t)}_{L^2(\Omega)}^2 +c_4     \norm{C^n}_{H^{1}}^2 
				&	\leq \displaystyle
				\Big| \int_{\Omega} \beta(C^n) \frac{S^n I^n }{N_n} S^n\, d\mathbf{x} \Big|
				\\\\
				&	\leq \displaystyle\beta_2  \int_{\Omega} \min(S^n,I^n) S^n\, d\mathbf{x}
				\\\\
				&	\leq \displaystyle\beta_2  \int_{\Omega} (S^n)^2\, d\mathbf{x}=\beta_2 \norm{S^n}_{L^2(\Omega)}^2,
			\end{array}
		\end{equation}
		
		\begin{equation}\label{3} 
			\begin{array}{rl}	
				\displaystyle	\frac{1}{2} \frac{d}{dt}\norm{I^n(t)}_{L^2(\Omega)}^2 +c_5     \norm{C^n}_{H^{1}}^2 
				&	\leq \displaystyle
				\Big| \int_{\Omega} \beta(C^n) \frac{S^n I^n }{N_n} I^n\, d\mathbf{x} \Big|
				\\\\
				&	\leq \displaystyle\beta_2  \int_{\Omega} \min(S^n,I^n) I^n\, d\mathbf{x}
				\\\\
				&	\leq \displaystyle\beta_2  \int_{\Omega} (I^n)^2\, d\mathbf{x}=\beta_2 \norm{I^n}_{L^2(\Omega)}^2,
			\end{array}
		\end{equation}
		and
		\begin{equation}\label{4} 
			\begin{array}{rl}	
				\displaystyle	\frac{1}{2} \frac{d}{dt}\norm{R^n(t)}_{L^2(\Omega)}^2 +c_6     \norm{R^n}_{H^{1}}^2 
				&	\leq \displaystyle
				\Big| \int_{\Omega} R^n I^n\, d\mathbf{x} \Big|
				\\\\
				&	\leq \displaystyle \norm{R^n}_{L^2(\Omega)}\norm{I^n}_{L^2(\Omega)}
				\\\\
				&	\leq \displaystyle\frac{1}{2} \norm{I^n}_{L^2(\Omega)}^2 +\frac{1}{2} \norm{R^n}_{L^2(\Omega)}^2.
			\end{array}
		\end{equation}
		\noindent
		Adding up the equations (\ref{1}), (\ref{2}), (\ref{3}) and (\ref{4}) gives
		\begin{equation}\label{5} 
			\begin{array}{rl}	
				\displaystyle \frac{d}{dt}\Big(\norm{C^n(t)}_{L^2(\Omega)}^2+\norm{S^n(t)}_{L^2(\Omega)}^2&+\norm{I^n(t)}_{L^2(\Omega)}^2+\norm{R^n(t)}_{L^2(\Omega)}^2\Big) +c_7     \Big(\norm{C^n}_{H^{1}}^2 +\norm{S^n}_{H^{1}}^2 \\
				&+\norm{I^n}_{H^{1}}^2 +\norm{R^n}_{H^{1}}^2  \Big)	
				\\
				&\leq \displaystyle c_8 \Big(\norm{C^n}_{L^2(\Omega)}^2+\norm{S^n}_{L^2(\Omega)}^2+\norm{I^n}_{L^2(\Omega)}^2 + \norm{R^n}_{L^2(\Omega)}^2\Big).
			\end{array}
		\end{equation}
		An application of Gronwall's inequality, we get that there exist a constant $c_9>0$ independent of $n$ such that
		\begin{equation}\label{6}
			\max_{0<\tau\leq T}    \int_{\Omega} |C^n(\tau,x)|^2 \, d\mathbf{x}+ \max_{0<\tau\leq T}    \int_{\Omega} |S^n(\tau,x)|^2 \, d\mathbf{x}+ \max_{0<\tau\leq T}    \int_{\Omega} |I^n(\tau,x)|^2 \, d\mathbf{x}+ \max_{0<\tau\leq T}    \int_{\Omega} |R^n(\tau,x)|^2 \, d\mathbf{x} \leq c_9 
			,
		\end{equation}
		Using this and integrating (\ref{5}) over $(0,T)$ we get that
		\begin{equation}\label{7}
			\int_{0}^T \norm{C^n}_{H^{1}}^2\,dt +\int_{0}^T\norm{S^n}_{H^{1}}^2\,dt +\int_{0}^T\norm{I^n}_{H^{1}}^2\,dt +\int_{0}^T\norm{R^n}_{H^{1}}^2\,dt \leq c_{10}
			,
		\end{equation}
		for some constant $c_{10}>0$ independent of $n$.

		\noindent
		Thus,  we deduce that $(C^n,S^n, I^n,R^n)$ is uniformly bounded in $L^\infty\big(0,T; L^2(\Omega) \big)  \cap L^2\big(0,T; H^1(\Omega) \big)$.
		
		\textit{Proof of $(iii)$.} First, to prove $(iii)$, we first show that $\partial_tC^n$, $\partial_tS^n$, $\partial_tI^n$ and  $\partial_tR^n$ are bounded in $L^2(0,T; W^{-1,2}(\Omega))$, i.e., there exist a constant positive $c_{11}$ such that:
	\begin{equation}
		{\norm{\partial_tC^n}}_{W^{-1,2}(\Omega)}^2+{\norm{\partial_tS^n}}_{W^{-1,2}(\Omega)}^2+{\norm{\partial_tI^n}}_{W^{-1,2}(\Omega)}^2+{\norm{\partial_tR^n}}_{W^{-1,2}(\Omega)}^2\leq c_{11}. 
	\end{equation}

\noindent To this end, let $(\psi_C, \psi_S, \psi_I, \psi_R)\in H^1(\Omega)$ then, we obtain
	\begin{eqnarray}\label{eq:1}
	\begin{split}
		\left|\langle\partial_tC^n, \psi_C \rangle\right|&\leq  \left|\int_{\Omega}  \boldsymbol{U}^n\cdot\nabla C^n\, \psi_C\, d\mathbf{x}\right|+ D_C\left|\int_{\Omega}\nabla C^n \cdot\nabla \psi_C\, d\mathbf{x}\right| + 
		 \alpha\left|\int_{\Omega} I^n\, \psi_C\, d\mathbf{x} \right|+ \lambda \left|\int_{\Omega} C^n \,\psi_C \, d\mathbf{x}\right|
		 \\
		&\leq {\norm{\boldsymbol{U}^n}}_{L^4(\Omega)} {\norm{\nabla C^n}}_{L^2(\Omega)} {\norm{\psi_C}}_{L^4(\Omega)}+D_C {\norm{\nabla C^n}}_{L^2(\Omega)} {\norm{\nabla \psi_C}}_{L^2(\Omega)}\\&+\alpha {\norm{I^n}}_{L^2(\Omega)} {\norm{\psi_C}}_{L^2(\Omega)}+ \lambda{\norm{ C^n}}_{L^2(\Omega)} {\norm{\psi_C}}_{L^2(\Omega)}
		 \\
		&\leq c_1^* \Big({\norm{\boldsymbol{U}^n}}_{L^4(\Omega)} {\norm{\nabla C^n}}_{L^2(\Omega)}+ {\norm{\nabla C^n}}_{L^2(\Omega)} 
		+ {\norm{I^n}}_{L^2(\Omega)} + {\norm{ C^n}}_{L^2(\Omega)}\Big) {\norm{\psi_C}}_{H^1(\Omega)}.
		\end{split} 
\end{eqnarray}
Similarly,
	\begin{eqnarray}\label{eq:2}
		\begin{split} 	  		
		\left|\langle\partial_t S^n, \psi_S \rangle\right|&\leq D_S\left|\int_{\Omega}\nabla S^n \cdot\nabla \psi_S\, d\mathbf{x}\right|+\left|\int_{\Omega}\beta(C^n) \frac{S^n\,I^n}{N_n} \, \psi_S\, d\mathbf{x}\right|
		+\left|\int_{\Omega} \Lambda \,\psi_S \, d\mathbf{x} \right|+ \eta \left|\int_{\Omega} S^n \,\psi_S \, d\mathbf{x}\right|
		\\
		&\leq  D_S\int_{\Omega}\left|\nabla S^n \cdot\nabla \psi_S\right|\, d\mathbf{x}+\beta_2 \int_{\Omega}\left| \frac{S^n\,I^n}{N_n}\right| \, \left|\psi_S\right|\, d\mathbf{x}
		+\left|\int_{\Omega} \Lambda \,\psi_S \, d\mathbf{x} \right|+ \eta \left|\int_{\Omega} S^n \,\psi_S \, d\mathbf{x}\right|
	 \\
	&\leq D_S {\norm{\nabla S^n}}_{L^2(\Omega)} {\norm{\nabla \psi_S}}_{L^2(\Omega)}+ \beta_2 {\norm{S^n}}_{L^2(\Omega)} {\norm{\psi_S}}_{L^2(\Omega)} +  {\norm{\Lambda}}_{L^2(\Omega)} {\norm{\psi_S}}_{L^2(\Omega)}\\
	&+\eta {\norm{S^n}}_{L^2(\Omega)}{\norm{\psi_S}}_{L^2(\Omega)} \\
	&\leq c_2^* \Big({\norm{\nabla S^n}}_{L^2(\Omega)} + {\norm{S^n}}_{L^2(\Omega)} +  {\norm{\Lambda}}_{L^2(\Omega)}  +{\norm{S^n}}_{L^2(\Omega)} \Big){\norm{\psi_S}}_{H^1(\Omega)}.
	\end{split} 
\end{eqnarray}
Additionally, we acquire that
\begin{eqnarray}\label{eq:3}
	\begin{split} 	  		
		\left|\langle\partial_t I^n, \psi_I\rangle\right|&\leq  D_I\left|\int_{\Omega}\nabla I^n \cdot\nabla \psi_I\, d\mathbf{x}\right|+\left|\int_{\Omega}\beta(C^n) \frac{S^n\,I^n}{N_n} \, \psi_I\, d\mathbf{x} \right|+ (\eta+\gamma) \left|\int_{\Omega} I^n \,\psi_I \, d\mathbf{x}\right|
	 \\
	&\leq c_3^* \Big({\norm{\nabla I^n}}_{L^2(\Omega)} + {\norm{I^n}}_{L^2(\Omega)} +{\norm{S^n}}_{L^2(\Omega)} \Big){\norm{\psi_I}}_{H^1(\Omega)}
		\end{split} 
\end{eqnarray}
and
\begin{eqnarray}\label{eq:4}
	\begin{split} 	  		
	\left|\langle\partial_t R^n, \psi_R\rangle\right|&\leq  D_R\left|\int_{\Omega}\nabla R^n \cdot\nabla \psi_R\, d\mathbf{x}\right|+\gamma \left|\int_{\Omega} I^n \,\psi_R \, d\mathbf{x}\right| +\eta \left|\int_{\Omega} R^n \,\psi_R \, d\mathbf{x}\right|
\\
&\leq c_4^* \Big({\norm{\nabla R^n}}_{L^2(\Omega)} + {\norm{I^n}}_{L^2(\Omega)} +{\norm{R^n}}_{L^2(\Omega)} \Big){\norm{\psi_R}}_{H^1(\Omega)}.
	\end{split} 
	\end{eqnarray}
	Since $\boldsymbol{U}^n \in L^2(0,T;\boldsymbol{H}^1(\Omega))$,  $C^n\in L^2(0,T;H^1(\Omega))\cap L^{\infty}(0,T;L^2(\Omega))$,  $S^n\in L^2(0,T;H^1(\Omega))\cap L^{\infty}(0,T;L^2(\Omega))$,  $I^n\in L^2(0,T;H^1(\Omega))\cap L^{\infty}(0,T;L^2(\Omega))$ and $R^n\in L^2(0,T;H^1(\Omega))\cap L^{\infty}(0,T;L^2(\Omega))$, then we  deduce there exists a positive constant $c_5$ such that
\begin{eqnarray}\label{eq:5}
	\begin{split}
	&\sup_{\norm{\psi_C}_{H^1(\Omega)}=1}\left|\langle\partial_t S^n, \psi_C\rangle\right|^2+	\sup_{\norm{\psi_S}_{H^1(\Omega)}=1}\left|\langle\partial_t S^n, \psi_S\rangle\right|^2\\
	&\qquad +	\sup_{\norm{\psi_I}_{H^1(\Omega)}=1}\left|\langle\partial_t I^n, \psi_I\rangle\right|^2+	\sup_{\norm{\psi_R}_{H^1(\Omega)}=1}\left|\langle\partial_t R^n, \psi_R\rangle\right|^2\leq c_5.
\end{split} 
\end{eqnarray}
 From (\ref{eq:5}) and using Aubin-Lions compactness lemma, we conclude the proof of Lemma \ref{lem-2}.
	\end{proof}
	
	\subsubsection{Passage to the limit in Faedo-Galerkin solutions}
	From Lemma \ref{lem-1} and Lemma \ref{lem-2}, taking subsequences if necessary,  there exist functions $C, S, I,R \in L^2(0,T; H^1(\Omega))$ and $\boldsymbol{U} \in L^2(0,T; \boldsymbol{H}^1(\Omega))$ such that the following convergence results:
	\begin{equation}\label{limit1}
	\begin{split} 
		\boldsymbol{U}^n &\to \boldsymbol{U}\text{ weakly in }L^2(0,T; \boldsymbol{H}^1(\Omega)), \\
		\boldsymbol{U}^n &\to \boldsymbol{U}\text{ strongly in }L^2(0,T; \boldsymbol{L}^2(\Omega)),\\
		 C^n &\to C \text{ weakly in }L^2(0,T; H^1(\Omega)), \\
		 C^n &\to C\text{ strongly in }L^2(0,T; L^2(\Omega)), \\
		 S^n &\to S\text{ weakly in }L^2(0,T; H^1(\Omega)), \\
	    S^n &\to S\text{ strongly in }L^2(0,T; L^2(\Omega)),\\
	    I^n &\to I\text{ weakly in }L^2(0,T; H^1(\Omega)), \\
	    I^n &\to I\text{ strongly in }L^2(0,T; L^2(\Omega)),\\
	    R^n &\to R\text{ weakly in }L^2(0,T; H^1(\Omega)), \\
	    R^n &\to R\text{ strongly in }L^2(0,T; L^2(\Omega)),
	    .
	\end{split}
\end{equation}
Thus, we can pass immediately to the limit in the weak approximate formulation (\ref{eq:var-approched}) as $n\longrightarrow \infty$ to obtain the result of Theorem \ref{theo-existence}.
\subsection{Uniqueness of weak solutions in the two-dimensional case $(d=2)$}
In order to achieve the uniqueness of weak solution in the two dimensional case, we additionally impose the following condition on the nonlinear coupling coefficients:
\begin{enumerate}
	\item[\textbf{(A4).}]  The functions $\beta$ and $\nu$ are globally Lipschitz.
\end{enumerate}
Consequently, we derive the following uniqueness result:
\begin{thm}
		The weak solution of (\ref{SIRCNS}) is unique.
\end{thm}
	\begin{proof}
		Suppose that there are two solutions $$(\boldsymbol{U}_1,  C_1, S_1, I_1, R_1)$$ and 
		$$(\boldsymbol{U}_2, C_2, S_2, I_2, R_2)$$ 
		to the problem (\ref{eqvar1})-(\ref{eqvar2}). 
		
		\noindent
{Denote $\boldsymbol{U}:= \boldsymbol{U}_1 - \boldsymbol{U}_2$,  $C:= C_1 - C_2$, $S:= S_1 - S_2$, $I:= I_1 - I_2$ and $R:= R_1 - R_2$.   
		 Then, $\boldsymbol{U}$, $C$, $S$, $I$ and  $R$ satisfy the following equations}
		\begin{eqnarray}
			\left\langle\partial_{t}\boldsymbol{U}, \boldsymbol{Z}\right\rangle+ \int_{\Omega} \nu(C_1) \nabla \boldsymbol{U}: \nabla \boldsymbol{Z}~d{\mathbf{x}}+\int_{\Omega}\left[\nu(C_1)-\nu\left(C_{2}\right)\right] \mathbb{D}(\boldsymbol{U}_{2}): \mathbb{D}(\boldsymbol{Z}) ~d{\mathbf{x}}&
			\nonumber\\
			+\int_{\Omega}( \boldsymbol{U}\cdot \nabla) \boldsymbol{U}_{2}\cdot \boldsymbol{Z}~d{\mathbf{x}}+\int_{\Omega}( \boldsymbol{U}_1\cdot \nabla) \boldsymbol{U}\cdot \boldsymbol{Z}~d{\mathbf{x}}
			&=0, \;\;\label{uni1} \\  
			\left\langle\partial_tC, \psi_C\right\rangle+D_C\int_{\Omega}\nabla C\cdot\nabla\, \psi_C~d{\mathbf{x}}+\int_{\Omega}\boldsymbol{U}\cdot\nabla C_1 \psi_C~d{\mathbf{x}}+\int_{\Omega} \boldsymbol{U}_{2} \cdot\nabla C\, \psi_C~d{\mathbf{x}}
			&
			\nonumber\\
			-\alpha\int_{\Omega} I\, \psi_C~d{\mathbf{x}}+\lambda\int_{\Omega} C\, \psi_C~d{\mathbf{x}}
			&=0,\label{uni2} \\ 
			\left\langle\partial_tS, \psi_S\right\rangle+D_S\int_{\Omega}\nabla S\cdot\nabla\, \psi_S ~d{\mathbf{x}}+\int_{\Omega}\Big[\beta(C_1)\frac{S_1 I_1}{N_1}-\beta(C_2)\frac{S_2 I_2}{N_2}\Big]\psi_S~d{\mathbf{x}}
			+\eta\int_{\Omega} S\, \psi_S~d{\mathbf{x}}&= 0, \;\;\label{uni3}
			\\ 
			\left\langle\partial_t I, \psi_I\right\rangle+D_I\int_{\Omega}\nabla I\cdot\nabla\, \psi_I ~d{\mathbf{x}}-\int_{\Omega}\Big[\beta(C_1)\frac{S_1 I_1}{N_1}-\beta(C_2)\frac{S_2 I_2}{N_2}\Big]\psi_I~d{\mathbf{x}}
			+(\eta+\gamma)\int_{\Omega} I\, \psi_I~d{\mathbf{x}}&= 0, \;\;\label{uni4}
			\\ 
			\left\langle\partial_tR, \psi_R\right\rangle+D_R\int_{\Omega}\nabla R\cdot\nabla\, \psi_R ~d{\mathbf{x}}
			+\eta\int_{\Omega} R\, \psi_R~d{\mathbf{x}}-\gamma\int_{\Omega} I\, \psi_R~d{\mathbf{x}}&= 0, \;\;\label{uni5}
		\end{eqnarray} 
		for all test functions  
		$\boldsymbol{Z} \in D([0,T); \mathbf{H})$, $\psi_C \in D([0,T);  H^1_0)$, $\psi_S \in D([0,T);  H^1_0)$, $\psi_I \in D([0,T);  H^1_0)$ 
		and $\psi_R \in D([0,T);  H^1_0)$,  
		with the initial condition
		$$\begin{aligned}
			\boldsymbol{U}(0,\mathbf{.})=C(0,\mathbf{.})=S(0,\mathbf{.})=I(0,\mathbf{.})=R(0,\mathbf{.})&=0& \text{ in } & \Omega.
			\end{aligned}
		$$
		Now, by using $\boldsymbol{Z} = \boldsymbol{U}(t)$ as a test function in $(\ref{uni1})$, we acquire the following estimate
		\begin{equation}\label{testU}
			\begin{aligned}
				\frac{1}{2} \frac{\mathrm{d}}{\mathrm{d} t}\|\boldsymbol{U}(t)\|_{\boldsymbol{L}^2(\Omega)}^{2}+\nu_{1}\|\boldsymbol{U}(t)\|_{\boldsymbol{H}^1(\Omega)}^{2} & \leq \left|\int_{\Omega}( \boldsymbol{U}(t)\cdot \nabla) \boldsymbol{U}_{2}(t)\cdot \boldsymbol{U}(t)~d{\mathbf{x}}\right|+\left|\int_{\Omega}( \boldsymbol{U}_1(t)\cdot \nabla) \boldsymbol{U}(t)\cdot \boldsymbol{U}(t)~d{\mathbf{x}}\right|
				\\
				& \hspace{1  cm}
				+\left|\int_{\Omega}(\nu(C_1)-\nu(C_{2})) \mathbb{D}(\boldsymbol{U}_{2}(t)): \mathbb{D}(\boldsymbol{U}(t))~d{\mathbf{x}}\right|.
			\end{aligned} 
		\end{equation}
		Based on (\ref{Gagliardo-Nirenberg-U}) and (\ref{inq-Young}), the first term in (\ref{testU}) can be estimate as follows
		\begin{equation}\label{estem-b}
			\begin{split} 
				\hspace*{-0.2cm}
				\left|\int_{\Omega}( \boldsymbol{U}(t)\cdot \nabla) \boldsymbol{U}_{2}(t)\cdot \boldsymbol{U}(t)~d{\mathbf{x}}\right| 
				&\leq {\|\boldsymbol{U}(t)\|}_{\boldsymbol{L}^4(\Omega)} {\|\nabla\boldsymbol{U}_2(t)\|}_{\boldsymbol{L}^2(\Omega)} {\|\boldsymbol{U}(t)\|}_{\boldsymbol{L}^4(\Omega)}
				\\
				&\leq c{\|\boldsymbol{U}(t)\|}_{\boldsymbol{H}^1(\Omega)}^{1/2}
				{\|\boldsymbol{U}(t)\|}_{\boldsymbol{L}^2(\Omega)}^{1/2} \left\|\boldsymbol{U}_{2}(t)\right\|_{\mathbf{H}^1(\Omega)}{\|\boldsymbol{U}(t)\|}_{\boldsymbol{H}^1(\Omega)}^{1/2}
				{\|\boldsymbol{U}(t)\|}_{\boldsymbol{L}^2(\Omega)}^{1/2}
				\\
				&\leq \delta\|\boldsymbol{U}(t)\|_{\boldsymbol{H}^1(\Omega)}^{2}+c_{\delta}\left\|\boldsymbol{U}_{2}(t)\right\|_{\mathbf{H}^{1}(\Omega)}^2\|\boldsymbol{U}(t)\|_{\mathbf{L}^{2}(\Omega)}^{2}.
			\end{split}
		\end{equation}
		For the second term, by Green's formula, we have that
		\begin{equation}\label{estm-b-null} 
			\int_{\Omega}( \boldsymbol{U}_1(t)\cdot \nabla) \boldsymbol{U}(t)\cdot \boldsymbol{U}(t)~d{\mathbf{x}}=0.
		\end{equation}
		In addition, the last term in (\ref{testU})  can be estimates by means of  Young inequality as well as H{\"o}lder's inequality to obtain
		\begin{equation}\label{estmD}
			\begin{aligned}
				\left|\int_{\Omega}\left(\nu(C_1)-\nu(C_{2})\right) \mathbb{D}(\boldsymbol{U}_{2}): \mathbb{D}(\boldsymbol{U})~d{\mathbf{x}}\right|
				& \leq \|(\nu(C_1)-\nu(C_{2}))\|_{L^{\infty}(\Omega)}
				\| \mathbb{D}(\boldsymbol{U}_{2})\|_{\mathbf{L}^{2}(\Omega)} \|\mathbb{D}(\boldsymbol{U})\|_{\mathbf{L}^{2}(\Omega)} \\
				& \leq c\|C(t)\|_{L^2(\Omega)}
				\| \mathbb{D}(\boldsymbol{U}_{2})\|_{\boldsymbol{L}^{2}(\Omega)} \|\boldsymbol{U}(t)\|_{\mathbf{H}^{1}(\Omega)}\\
				& \leq \delta \|\boldsymbol{U}(t)\|_{\boldsymbol{H}^{1}(\Omega)}^2+C_\delta
				\|\boldsymbol{U}_{2}(t)\|_{\mathbf{H}^{1}(\Omega)}^2\|C(t)\|_{L^2(\Omega)}^2.
			\end{aligned}
		\end{equation}
		Consequently, the estimates $(\ref{estem-b})-(\ref{estmD})$ imply that
		\begin{equation}\label{estmU}
			\begin{aligned}
				\frac{\mathrm{d}}{\mathrm{d} t}\|\boldsymbol{U}(t)\|_{\boldsymbol{L}^2(\Omega)}^{2}+c\|\boldsymbol{U}(t)\|_{\boldsymbol{H}^1(\Omega)}^{2}  & \leq \delta \|\boldsymbol{U}(t)\|_{\boldsymbol{H}^{1}(\Omega)}^2+
				c_{\delta}\|\boldsymbol{U}_{2}(t)\|_{\mathbf{H}^{1}(\Omega)}^2\left(\|\boldsymbol{U}(t)\|_{\mathbf{L}^{2}(\Omega)}^{2}+\|C(t)\|_{L^2(\Omega)}^2\right), 
			\end{aligned} 
		\end{equation} 
		Now, we substitute $\psi_C=C$ in (\ref{uni2}), to obtain the
		following estimate   
		$$
		\begin{array}{rl}
			\displaystyle \frac{1}{2}\frac{\mathrm{d}}{\mathrm{d} t}\|C(t)\|_{L^2(\Omega)}^{2}+D_C\left\|C\right\|_{H^{1}(\Omega)}^{2}&\leq \displaystyle \Big|\int_{\Omega}\boldsymbol{U}\cdot\nabla C_1 \,C~d{\mathbf{x}}\Big| +\Big|\int_{\Omega} \boldsymbol{U}_{2} \cdot\nabla C\, C~d{\mathbf{x}}\Big| \\&\displaystyle\qquad+\alpha\left| \int_{\Omega} I\, C~d{\mathbf{x}} \right|+\lambda\left|\int_{\Omega} C^2~d{\mathbf{x}} \right|.
		\end{array}
		$$
		By virtue of  H{\"o}lder and Young inequalities, we estimate the terms on the right-hand side of the previous equation as follows
		
		\begin{equation}\label{C1-C2}
			\Big|\int_{\Omega} I C ~d{\mathbf{x}} \Big| \leq \frac{1}{2}\Big(\left\|I\right\|_{L^{2}(\Omega)}^{2}+\left\|C\right\|_{L^{2}(\Omega)}^{2}\Big).
		\end{equation}
		By the Gagliardo-Nirenberg inequality with $d=2$, we deduce that
		\begin{equation}
			\begin{aligned}
				&\Big|\int_{\Omega}\boldsymbol{U}(t)\cdot\nabla C_1(t) \,C(t)~d{\mathbf{x}}\Big|
				\\ &\qquad \leq
				\|\boldsymbol{U}(t)\|_{\mathbf{L}^{4}(\Omega)}\left\|\nabla C_1(t)\right\|_{\mathbf{L}^{2}(\Omega)}\left\|C(t)\right\|_{L^{4}(\Omega)} 
				\\ &\qquad  \leq
				c\|\boldsymbol{U}(t)\|_{\mathbf{H}^{1}(\Omega)}^{1/2}\|\boldsymbol{U}(t)\|_{\mathbf{L}^{2}(\Omega)}^{1/2}\left\|C_1(t)\right\|_{H^{1}(\Omega)}^{}\left\|C(t)\right\|_{H^{1}(\Omega)}^{1/2}\left\|C\right\|_{L^{2}(\Omega)}^{1/2}
				\\ &\qquad  \leq
				\delta(\left\|C(t)\right\|_{H^{1}(\Omega)}^{}\|\boldsymbol{U}(t)\|_{\mathbf{H}^{1}(\Omega)})+c_{\delta}\left\|C_1\right\|_{H^{1}(\Omega)}^{2}\|\boldsymbol{U}(t)\|_{\mathbf{L}^{2}(\Omega)}^{}\left\|C(t)\right\|_{L^2(\Omega)}
				\\ &\qquad \leq
				\delta/2\Big(\left\|C(t)\right\|_{H^{1}(\Omega)}^{2}\hspace*{-0.1 cm}+\|\boldsymbol{U}(t)\|_{\mathbf{H}^{1}(\Omega)}^{2}\Big)\hspace*{-0.1 cm}+c_{\delta}\left\|C_1\right\|_{H^{1}(\Omega)}^{2}\hspace*{-0.1 cm}\Big(\|\boldsymbol{U}(t)\|_{\mathbf{L}^{2}(\Omega)}^{2}\hspace*{-0.1 cm}+\left\|C(t)\right\|_{L^2(\Omega)}^2\Big),
			\end{aligned}
		\end{equation}
	for a given positive function $c_{\delta}$ depending on $\delta$.

	\noindent
		Moreover, using the same technique in (\ref{galc}), we obtain
		\begin{equation}\label{E2}
			\begin{aligned}
				\int_{\Omega} \boldsymbol{U}_{2} \cdot\nabla C\, C~d{\mathbf{x}}=0.
			\end{aligned}
		\end{equation} 
		Collecting the previous results \eqref{C1-C2}-\eqref{E2}, we deduce that
		\begin{equation}\label{estmtheta}
			\begin{split}
				\frac{d}{d t}\left\|C\right\|_{L^{2}(\Omega)}^{2}&+c\left\|C(t)\right\|_{H^{1}(\Omega)}^{2}
				\\&\leq \delta\left(\left\|C(t)\right\|_{H^{1}(\Omega)}^{2}+\|\boldsymbol{U}\|_{\mathbf{H}^{1}(\Omega)}^{2}\right)+c_{\delta}\left\|C_1\right\|_{H^{1}(\Omega)}^{2} \left(\|C\|_{L^{2}(\Omega)}^{2}+\|\boldsymbol{U}\|_{\mathbf{L}^{2}(\Omega)}^{2}\right)\\
				&\qquad+\frac{1}{2}\Big(\left\|I\right\|_{L^{2}(\Omega)}^{2}+\left\|C\right\|_{L^{2}(\Omega)}^{2}\Big)+\lambda\left\|C\right\|_{L^{2}(\Omega)}^{2}.
			\end{split}
		\end{equation}
		
		\noindent For $\psi_S=S$ in (\ref{uni3}), we obtain that 
		\begin{equation}\label{estmS0}
			\displaystyle \frac{1}{2}\frac{\mathrm{d}}{\mathrm{d} t}\|S(t)\|_{L^2(\Omega)}^{2}+D_C\left\|S(t)\right\|_{H^{1}(\Omega)}^{2}\leq  \Big|\int_{\Omega}\Big[\beta(C_1)\frac{S_1 I_1}{N_1}-\beta(C_2)\frac{S_2 I_2}{N_2}\Big] S~d{\mathbf{x}}\Big| +\eta\left|\int_{\Omega} S^2~d{\mathbf{x}} \right|.
		\end{equation}
		We get that
		\begin{equation*}
			\begin{split}
				\Big|\int_{\Omega}\Big[\beta(C_1)\frac{S_1 I_1}{N_1}&-\beta(C_2)\frac{S_2 I_2}{N_2}\Big] S~d{\mathbf{x}}\Big|
				\\&=\Big|\int_{\Omega}\beta(C_1)\Big[\frac{S_1 I_1}{N_1}-\frac{S_2 I_2}{N_2}\Big] S~d{\mathbf{x}}+\int_{\Omega}\Big[\beta(C_1)-\beta(C_2)\Big]\frac{S_2 I_2}{N_2} S~d{\mathbf{x}}\Big|
				\\&\leq \beta_2\int_{\Omega}\Big|\frac{S_1 I_1}{N_1}-\frac{S_2 I_2}{N_2}\Big| S~d{\mathbf{x}}+\int_{\Omega}\Big|\beta(C_1)-\beta(C_2)\Big|\frac{S_2 I_2}{N_2} S~d{\mathbf{x}}\Big|
				\\&\leq\beta_2\left\|\frac{S_1 I_1}{N_1}-\frac{S_2 I_2}{N_2}\right\|_{L^{2}(\Omega)}\left\|S\right\|_{L^{2}(\Omega)}+ \left\|\beta(C_1)-\beta(C_2)\right\|_{L^{4}(\Omega)} \left\|\frac{S_2 I_2}{N_2}\right\|_{L^{2}(\Omega)} \|S\|_{L^{4}(\Omega)}
				\\&\leq \beta_2 L_g(\left\|S\right\|_{L^{2}(\Omega)}+\left\|I\right\|_{L^{2}(\Omega)})\|S\|_{L^{2}(\Omega)}+ L_{\beta}\|C\|_{L^{4}(\Omega)}\left\|S_2\right\|_{L^{2}(\Omega)} \|S\|_{L^{4}(\Omega)}
				\\&\leq \beta_2 L_g(\left\|S\right\|_{L^{2}(\Omega)}+\left\|I\right\|_{L^{2}(\Omega)})\|C\|_{L^{2}(\Omega)}
				\\
			    &+ c\left\|C(t)\right\|_{H^{1}(\Omega)}^{1/2}\left\|C(t)\right\|_{L^{2}(\Omega)}^{1/2}\left\|S_2\right\|_{L^{2}(\Omega)} \left\|S(t)\right\|_{H^{1}(\Omega)}^{1/2}\left\|S(t)\right\|_{L^{2}(\Omega)}^{1/2}
				\\&\leq \delta \left(\left\|S\right\|_{H^{1}(\Omega)}^2+\left\|C\right\|_{H^{1}(\Omega)}^2\right)+c\left( \left\|C\right\|_{L^{2}(\Omega)}^2+\left\|I\right\|_{L^{2}(\Omega)}^2+\left\|S\right\|_{L^{2}(\Omega)}^2\right)
				\\&+c_{\delta}\left\|S_2\right\|_{H^{1}(\Omega)}^{2}\left( \left\|C\right\|_{L^{2}(\Omega)}^2+\left\|S\right\|_{L^{2}(\Omega)}^2\right). 
			\end{split}
		\end{equation*}
		Thereby
		\begin{equation}\label{estmS}
			\begin{split}
				&\frac{\mathrm{d}}{\mathrm{d} t}\|S(t)\|_{L^2(\Omega)}^{2}+c_1\left\|S\right\|_{H^{1}(\Omega)}^{2}\\
				&\quad\leq \delta \left(\left\|S\right\|_{H^{1}(\Omega)}^2+\left\|C\right\|_{H^{1}(\Omega)}^2\right)+\left(c_{\delta}\left\|S_2\right\|_{H^{1}(\Omega)}^{\frac{1}{1-\zeta}}+c_2\right)\left( \left\|C\right\|_{L^{2}(\Omega)}^2+\left\|I\right\|_{L^{2}(\Omega)}^2+\left\|S\right\|_{L^{2}(\Omega)}^2\right). 
			\end{split}
		\end{equation}

		\noindent Similarly to (\ref{estmS}),  we use $\psi_I=I$ as a test function in (\ref{uni4}) to deduce that   
		\begin{equation}\label{estmI}
			\begin{split}
				&\displaystyle \frac{\mathrm{d}}{\mathrm{d} t}\|I(t)\|_{L^2(\Omega)}^{2}+c_3\left\|I\right\|_{H^{1}(\Omega)}^{2}
				\\
				&\quad\leq  \Big|\int_{\Omega}\Big[\beta(C_1)\frac{S_1 I_1}{N_1}-\beta(C_2)\frac{S_2 I_2}{N_2}\Big] I~d{\mathbf{x}}\Big| +(\eta+\gamma)\left|\int_{\Omega} I^2~d{\mathbf{x}} \right|
				\\
				&\quad\leq \delta \left(\left\|I\right\|_{H^{1}(\Omega)}^2+\left\|C\right\|_{H^{1}(\Omega)}^2\right)+\left(c_{\delta}\left\|S_2\right\|_{H^{1}(\Omega)}^{\frac{1}{1-\zeta}}+c_4\right)
				\left( \left\|C\right\|_{L^{2}(\Omega)}^2+\left\|I\right\|_{L^{2}(\Omega)}^2+\left\|S\right\|_{L^{2}(\Omega)}^2\right) 
				.
			\end{split}
		\end{equation}

		\noindent Further, we use substitute $\psi_R=R$ in (\ref{uni5}) to get   
		\begin{equation}\label{estmR}
			\begin{split}
				\displaystyle \frac{1}{2}\frac{\mathrm{d}}{\mathrm{d} t}\|R(t)\|_{L^2(\Omega)}^{2}+D_R\left\|R\right\|_{H^{1}(\Omega)}^{2}&\leq  \gamma \Big|\int_{\Omega}I\, R~d{\mathbf{x}}\Big| +\eta\left|\int_{\Omega} R^2~d{\mathbf{x}} \right|
				\\&\leq\gamma\left\|I\right\|_{L^{2}(\Omega)}\left\|C\right\|_{L^{2}(\Omega)}+\eta \left\|C\right\|_{L^{2}(\Omega)}^{2}
				\\&\leq\max(\gamma/2,\eta+\gamma/2)\Big(\left\|I\right\|_{L^{2}(\Omega)}^2+ \left\|C\right\|_{L^{2}(\Omega)}^{2}\Big)
				.
			\end{split}
		\end{equation}

		\noindent	We sum up the estimates (\ref{estmU}), (\ref{estmtheta}),  (\ref{estmS}), (\ref{estmI}) and (\ref{estmR}) to deduce, for a sufficiently small $\delta>0$
		\begin{equation}\label{estm-U-Theta}
			\begin{aligned}
				\frac{d}{d t}\Big(\|\boldsymbol{U}\|_{\mathbf{L}^{2}}^{2}&+ \left\|C\right\|_{L^{2}(\Omega)}^{2}+\left\|S\right\|_{L^{2}(\Omega)}^{2}+\left\|I\right\|_{L^{2}(\Omega)}^{2}+\left\|R\right\|_{L^{2}(\Omega)}^{2} \Big)\\ & \leq K(t)\left(\|\boldsymbol{U}\|_{\mathbf{L}^{2}}^{2}+ \left\|C\right\|_{L^{2}(\Omega)}^{2}+\left\|S\right\|_{L^{2}(\Omega)}^{2}+\left\|I\right\|_{L^{2}(\Omega)}^{2}+\left\|R\right\|_{L^{2}(\Omega)}^{2}  \right),
			\end{aligned} 
		\end{equation}
		where $$K(t)=c_{\delta}\left(\left\|S_2(t)\right\|_{H^{1}(\Omega)}^{2}+\left\|C_1(t)\right\|_{H^{1}(\Omega)}^{2}+\|\boldsymbol{U}_{2}(t)\|_{\mathbf{H}^{1}(\Omega)}^2+1\right).$$
		
		\noindent
		Applying Gronwall's inequality to $(\ref{estm-U-Theta})$ along with the fact that $$\boldsymbol{U}(0,\mathbf{.})=C(0,\mathbf{.})=S(0,\mathbf{.})=I(0,\mathbf{.})=R(0,\mathbf{.})=0$$  
		we arrive at
		$$C=S=I=R=\boldsymbol{U}=0.$$
		\noindent
		This yields the uniqueness of the weak solution. 
	\end{proof}
\begin{rem}
We mention that the assumption of global Lipschitzity on the coupling coefficients $\nu$ and $\beta$ can be relaxed to local Lipschitzity. Indeed, we can proceed by truncation techniques \cite{mehdaoui2024optimal} and establish uniqueness as long as the initial condition is bounded.
\end{rem}

\section{Numerical simulations of the SIRPNS model}\label{numerics}

\subsection{Numerical scheme}
	For the numerical resolution of the fully coupled system (\ref{SIRCNS}), we use a combination of standard techniques for time and space discretization.
	On the one hand, in order to discretize time, we fix an integer $M$ and define a time subdivision
	$
	t_0 = 0 < t_1 < \cdots < t_M = T
	$, 
	where $T$ represents the final simulation time. The time steps are defined by $\Delta t = t_{n+1} - t_n$ for $n = 0, \ldots, M-1$. Then, 
	we build a Backward Euler scheme for the time semi-discretization of the problem~(\ref{eqvar1})-(\ref{eqvar2}). For any sequence of functions $\{f^n\}_{n=0}^{M}$ defined on the domain $\Omega$, the discrete time derivative operator is defined as follows:
	$$
	D_{\Delta t} f^{n+1} = \frac{f^{n+1} - f^n}{\Delta t}.
	$$
	
	Following the discretization in time, we obtain an elliptic problem varying in the spatial domain $\Omega$. In order to approximate this problem, we use the Finite Element Method (FEM) as follows.
	
	\begin{itemize}
		\item[$\bullet$] 	$P1$-Bubble Element for the velocity variable ($\boldsymbol{U}$).
		\item[$\bullet$]  $P1$ Element for the pressure ($P$), concentration of the pathogen ($C$), density of the susceptible ($S$), density of the infected ($I$), and density of the recovered ($R$).
	\end{itemize}
	For the discrete versions of these variables, we keep the same notations $\boldsymbol{U}$, $C$, $S$, $I$, and $R$ . To handle the nonlinearities present in the problem, we implement an iterative scheme based on the fixed-point iteration method.
	In each iteration of this scheme, the velocity ($\boldsymbol{U}$), susceptibility ($S$), infection ($I$), and recovery ($R$) variables are computed using the value of the other concentration variable from the previous iteration. The proposed iterative scheme is implemented as follows:

\begin{algorithm}[H]
	\caption{Iterative scheme for the numerical solution}
	\footnotesize
	\begin{description}
		\item[Step 1. ] Initialize with an initial guess $(C^0, \boldsymbol{U}^0, S^0 , I^0 , R^0)$.
		
		\item[Step 2. ] For $n \geq 0$, given $\left(C^n, \boldsymbol{U}^n, S^n , I^n , R^n\right)$, compute the values of $\left(C^{n+1}, \boldsymbol{U}^{n+1}, S^{n+1} , I^{n+1} , R^{n+1}\right)$ as follows:
		\begin{enumerate}
			\item Given $C^n$, $S^n$, $I^n$, solve for $(\boldsymbol{U}^{n+1},  p^{n+1}, S^{n+1}, I^{n+1}, R^{n+1})$ from:
			\begin{eqnarray*}
				\left\langle \frac{\boldsymbol{U}^{n+1}-\boldsymbol{U}^{n}}{\Delta t}, \boldsymbol{Z} \right\rangle 
				+ \int_{\Omega}\nu(C^n) \nabla\boldsymbol{U}^{n+1}: \nabla \boldsymbol{Z} \, d\mathbf{x}
				+ \int_{\Omega}(\boldsymbol{U}^{n+1}\cdot \nabla) \boldsymbol{U}^{n}\cdot \boldsymbol{Z} \, d\mathbf{x} 
				&- \displaystyle\int_{\Omega} p^{n+1} \nabla \cdot \boldsymbol{Z} \, d\mathbf{x} 
				&= \int_{\Omega}\boldsymbol{f}\cdot \boldsymbol{Z}\, d\mathbf{x} \\
				&     \displaystyle  \int_{\Omega} q \nabla \cdot \boldsymbol{U}^{n+1} \, d\mathbf{x} 
				&= 0.
			\end{eqnarray*}
			and			
			\begin{eqnarray*}
				\begin{split}  	  		
					\left\langle \frac{S^{n+1}-S^{n}}{\Delta t}, \psi_S \right\rangle+ D_S\int_{\Omega}\nabla S^{n+1} \cdot\nabla \psi_S\, d\mathbf{x}+\int_{\Omega}\beta(C^n) \frac{S^{n+1}\,I^n}{N_n} \, \psi_S\, d\mathbf{x}
					-\int_{\Omega} \Lambda \,\psi_S \, d\mathbf{x} + \eta \int_{\Omega} S^{n+1} \,\psi_S \, d\mathbf{x}
					&=0,
					\\ 	  		
					\left\langle \frac{I^{n+1}-I^{n}}{\Delta t}, \psi_I \right\rangle+ D_I\int_{\Omega}\nabla I^{n+1} \cdot\nabla \psi_I\, d\mathbf{x}-\int_{\Omega}\beta(C^n) \frac{S^n\,I^{n+1}}{N_n} \, \psi_I\, d\mathbf{x} + (\eta+\gamma) \int_{\Omega} I^{n+1} \,\psi_I \, d\mathbf{x}
					&=0,
					\\ 	  		
					\left\langle \frac{R^{n+1}-R^{n}}{\Delta t}, \psi_R \right\rangle+ D_R\int_{\Omega}\nabla R^{n+1} \cdot\nabla \psi_R\, d\mathbf{x}-\gamma\int_{\Omega} I^n \,\psi_R \, d\mathbf{x}+ \eta\int_{\Omega} R^{n+1} \,\psi_R \, d\mathbf{x}
					&=0,
				\end{split} 
			\end{eqnarray*} 
			
			\item Given $\boldsymbol{u}^{n+1}$ and $I^{n+1}$, solve for $C^{n+1}$ from:
			\begin{eqnarray*} 
				\left\langle \frac{C^{n+1}-C^{n}}{\Delta t}, \psi_C \right\rangle 
				+ \int_{\Omega}  \boldsymbol{U}^{n+1}\cdot\nabla C^n\, \psi_C\, d\mathbf{x}+ D_C\int_{\Omega}\nabla C^{n+1} \cdot\nabla \psi_C\, d\mathbf{x}
				- \alpha\int_{\Omega} I^{n+1}\, \psi_C\, d\mathbf{x} + \lambda \int_{\Omega} C^{n+1} \,\psi_C \, d\mathbf{x}  
				= 0.
			\end{eqnarray*}
			
			\item If the difference between $\left(C^{n+1}, \boldsymbol{U}^{n+1}, S^{n+1} , I^{n+1} , R^{n+1}\right)$ and $\left(C^{n}, \boldsymbol{U}^{n}, S^{n} , I^{n} , R^{n}\right)$ exceeds a predefined tolerance, set $n = n + 1$ and return to \textbf{Step 2}. Otherwise, proceed to \textbf{Step 3}.
		\end{enumerate}
		
		\item[Step 3. ] Stop the iterative process once the solution converges (within the tolerance).
	\end{description}
	\label{algorithm}
\end{algorithm}

	\subsection{Numerical experiments and interpretations}
	
	In this section, to validate our proposed model, we conduct several numerical studies that will confirm the interaction between the variables in our SIRPNS model. Specifically, we demonstrate the interaction between C, S, I, and R, as well as the effect of fluid on disease spread in a fluid medium. For this, we consider the fixed spatial domain
	$$\Omega = (0, 1)^2$$
	and the fixed time horizon
	$$(0, T) = (0, 40).$$
		We start from an initial of $S$, $I$, $R$, $C$ and $\boldsymbol{U}$ given by
	$$
	\begin{aligned}
		\boldsymbol{U}_0(x,y)&=\begin{pmatrix}
			\sin^2(\pi (x-1))\sin(\pi (y-1))\cos(\pi (y-1)) \\
			-\sin^2(\pi (y-1))\sin(\pi (x-1))\cos(\pi (x-1))
		\end{pmatrix},\\
		C_0(x, y)&=(0.5+100((x-0.5)^2+(y-0.5)^2))\exp(-(x-0.5)^2-(y-0.5)^2),\\
		S_0(x, y)&=
		\begin{cases}0.9 \exp \left(-(x-0.5)^2-(y-0.5)^2\right), & \text { if }(x-0.5)^2 + (y-0.5)^2  \leq 1, \\ 0, & \text { otherwise.}\end{cases}
		\\
		I_0(x, y)&=
		\begin{cases}0.1 \exp \left(-(x-0.5)^2-(y-0.5)^2\right), & \text { if }(x-0.5)^2 + (y-0.5)^2  \leq 0.1, \\ 0, & \text { otherwise, }\end{cases}\\
		R_0(x, y)&=0.01(0.5+5((x-0.7)^2+(y-0.3)^2))\exp(-100(x-0.7)^2-100(y-0.3)^2)
	\end{aligned}
	$$
	\begin{table}[!ht]
		\centering
		\begin{tabular}{|c|c|c|c|c|c|c|c|c|c|c|c|}
			\hline
			Parameter & $D_S$ & $D_I$ & $D_R$ & $D_C$ & $\alpha$ & $\gamma$ & $\lambda$ & $\eta$ & $\Lambda$ & $\Delta t$ & $T$ \\
			\hline
			Valuer  & 0.2 & 0.3 & 0.4 & 0.1 & 0.6 & 0.4 & 0.4  & 0.05 &0.4 & 0.01 & 40 \\
			\hline
		\end{tabular}
		\caption{Assigned numerical values to the biological parameters (assumed).}
	\end{table}

	
\subsubsection{SIR Interaction}

The first numerical experiment examines how an infection diffuses through space when neither environmental contamination nor fluid motion is present. Figure~\ref{fig:SIR-evolution} shows the time evolution starting from a localized infected zone. The infection gradually extends outward by diffusion, influencing the nearby susceptible population. The number of susceptibles decreases around the initial source, while the recovered density increases as immunity develops. Spatially, the epidemic forms a smooth and symmetric gradient among $S$, $I$, and $R$. This represents a closed population where disease transmission occurs only through local contact. After reaching its maximal intensity, the infection front weakens as recovery dominates, reproducing the classical SIR diffusion pattern: an early expansion followed by attenuation once immune individuals accumulate.


\subsubsection{The impact of environmental pathogens when $\beta(C)=\beta_0+C$ and without fluid}

Allowing the transmission rate to depend on pathogen concentration, $\beta(C)=\beta_0+C$, introduces the effect of environmental contamination. Indeed, an observation of Fig. \ref{fig:SIRC-evolution} asserts that the coupling between the pathogen (through the density $C$) and the infection dynamics ($SIR$) leads to and intensification in disease spread, especially near the contamination center. Additionally, we observe that the infected population $I$ rises more sharply, while the susceptible component $S$ decreases faster than in the diffusion-only case. Although the contaminant $C$ diffuses outward and decays over time, it continues to act as a local reservoir that sustains transmission. Finally, the recovered variable $R$ grows steadily as reinfection persists around contaminated areas. Thus, through this simulation, we conclude  how indirect transmission through the surrounding medium can maintain infection even when direct contact is limited.


\subsubsection{Impact of Fluid Flow in the SIRP Model (Constant versus Pathogen-Dependent Viscosity)}

Incorporating the Navier–Stokes flow reveals the role of fluid advection in pathogen transport (Figures~\ref{fig:CIRCU-const}–\ref{fig:CIRCU-var}). For constant viscosity, advection conveys pathogens along streamlines, stretching and enlarging the infected region downstream. When viscosity depends on pathogen concentration ($\nu=\nu_0+C$), the flow becomes slower and more resistant within contaminated zones. Pathogens are then partially trapped, forming localized clusters of infection. This behaviour reproduces what can occur in natural fluids, where contamination modifies viscosity and consequently alters flow structure, thereby prolonging local persistence of microorganisms.


\subsubsection{Pathogen Transport Before Epidemic Feedback}

To isolate the transport process, pathogen dynamics are first simulated without coupling to the epidemic variables (Figure~\ref{fig:C-evolution}). The contaminant field $C$ expands outward from its initial concentration by diffusion and mild advection, while progressively decaying in amplitude. Over time, the distribution becomes smoother and weaker, defining potential contamination spots that may later trigger infection. This step highlights how environmental persistence alone can shape the initial spatial structure of an outbreak, even before epidemiological feedback is activated.

\newpage
\begin{figure}[H]
	\centering
	\begin{minipage}{\linewidth}
		\includegraphics[width=.3\linewidth]{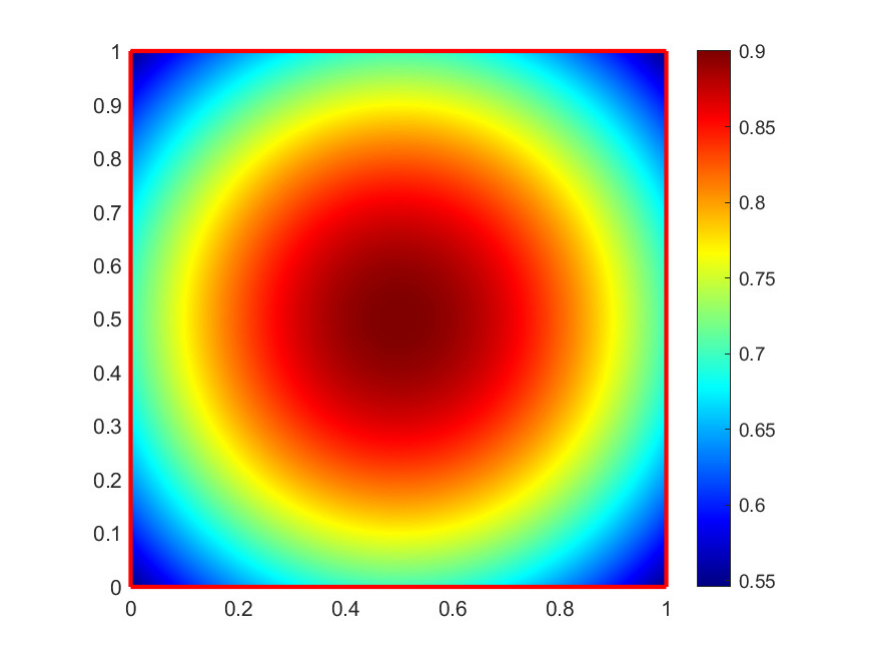}
		\includegraphics[width=.3\linewidth]{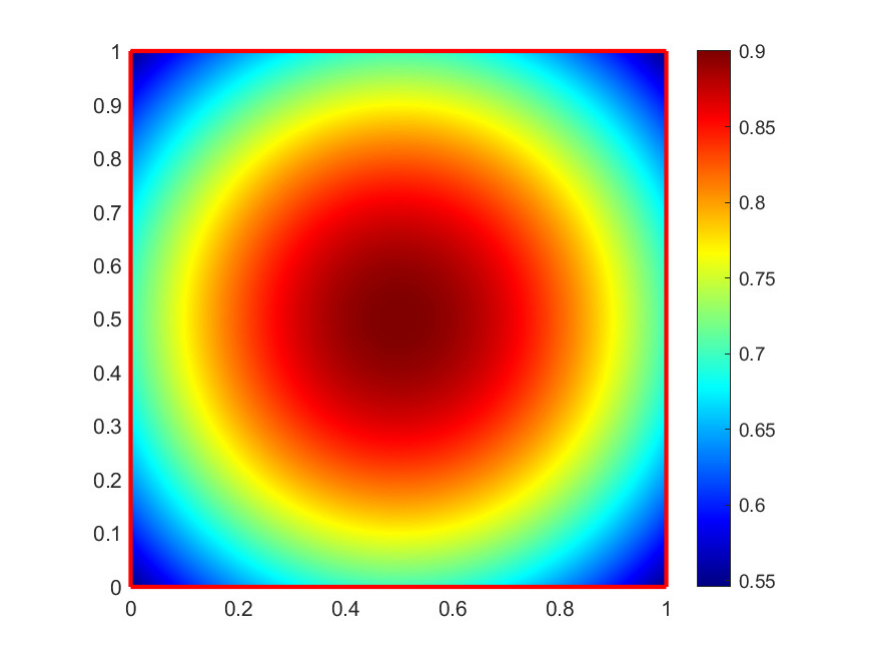}
		\includegraphics[width=.3\linewidth]{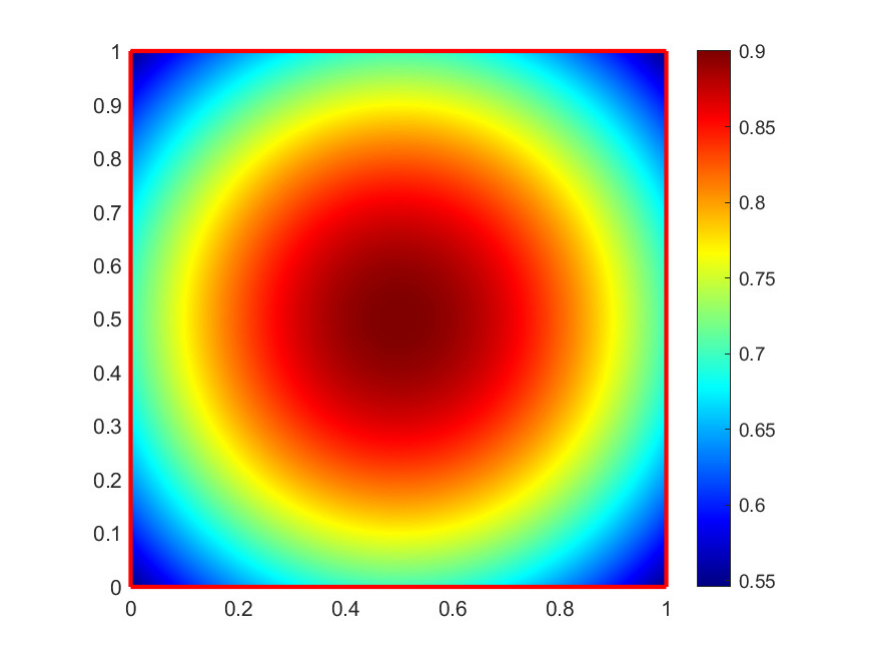}
		\caption*{(a) Spatial distribution of $S$ at $t=0$, $t=20$, and $t=40$.}
	\end{minipage}\\
	\begin{minipage}{\linewidth}
		\includegraphics[width=.3\linewidth]{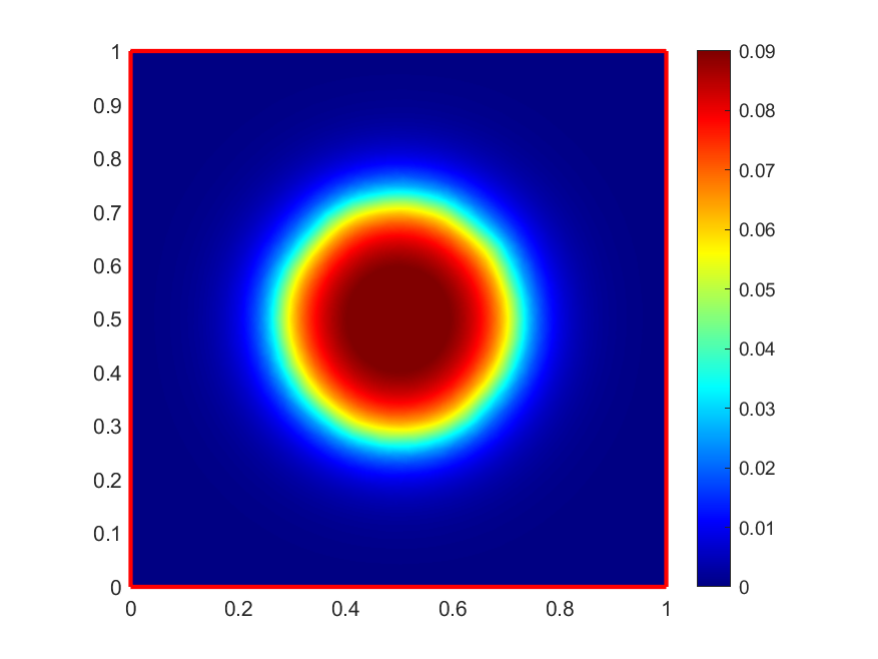}
		\includegraphics[width=.3\linewidth]{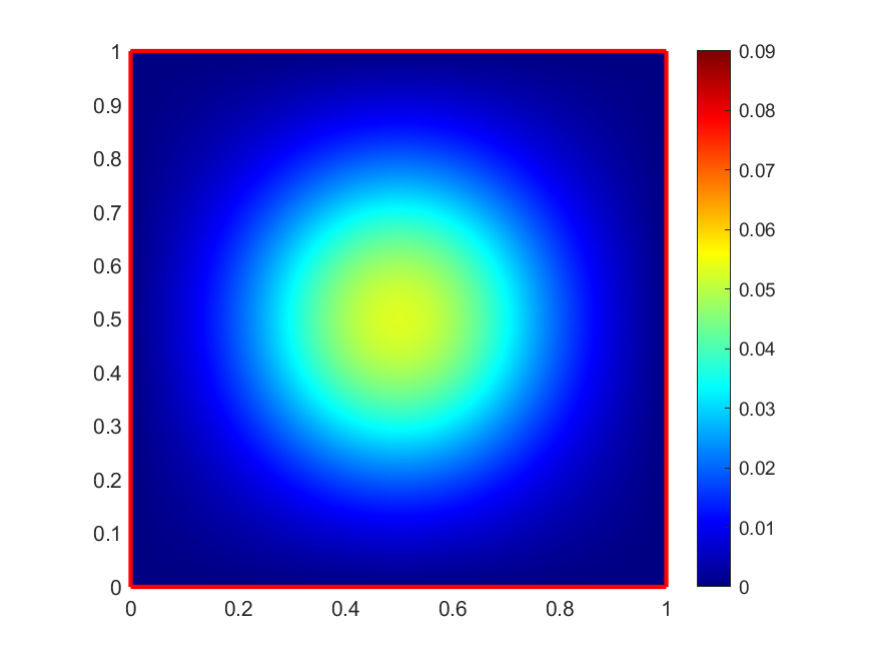}
		\includegraphics[width=.3\linewidth]{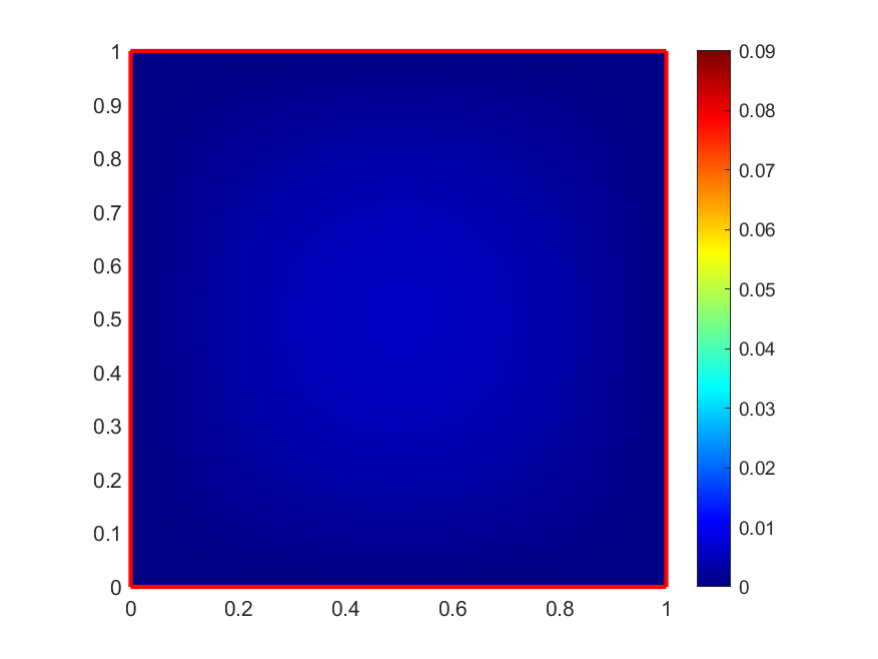}
		\caption*{(b) Spatial distribution of $I$ at $t=0$, $t=20$, and $t=40$.}
	\end{minipage}\\
	\begin{minipage}{\linewidth}
		\includegraphics[width=.3\linewidth]{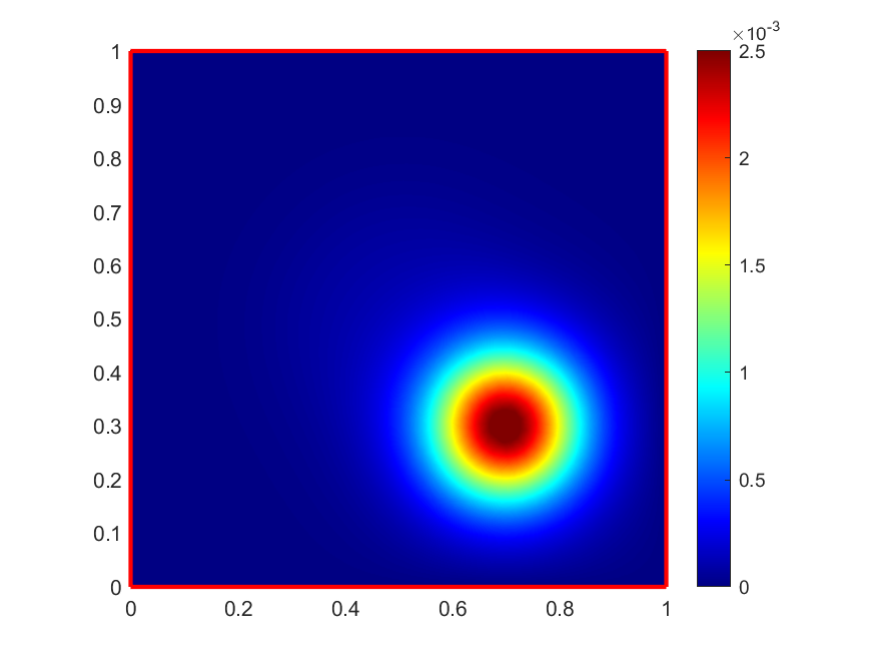}
		\includegraphics[width=.3\linewidth]{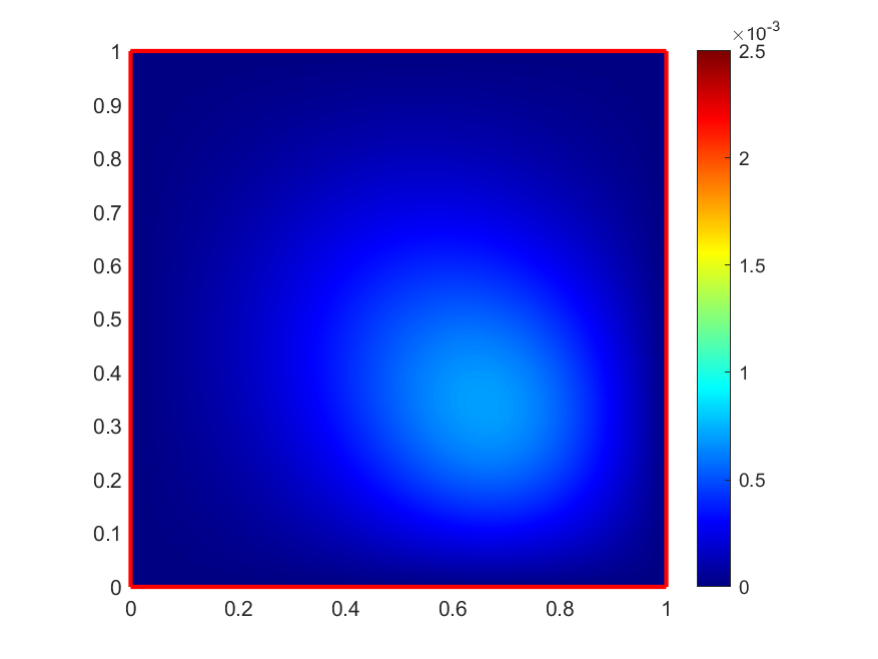}
		\includegraphics[width=.3\linewidth]{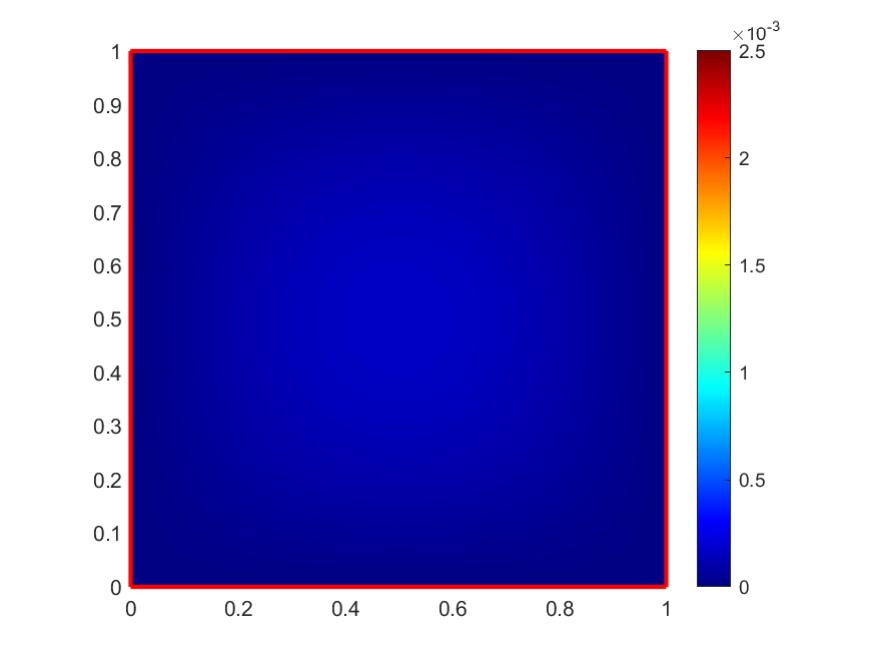}
		\caption*{(c) Spatial distribution of $R$ at $t=0$, $t=20$, and $t=40$.}
	\end{minipage}
	\caption{Example 1 – Diffusion-driven SIR dynamics without pathogen or fluid.  
		Starting from a localized infection, the epidemic front diffuses outward, reducing $S$ around the source and increasing $R$ behind the front.  
		The fields remain smooth and symmetric, reflecting pure diffusion where transmission occurs only through local contact.}
	\label{fig:SIR-evolution}
\end{figure}

\begin{figure}[H]
	\centering
	\includegraphics[width=.3\linewidth]{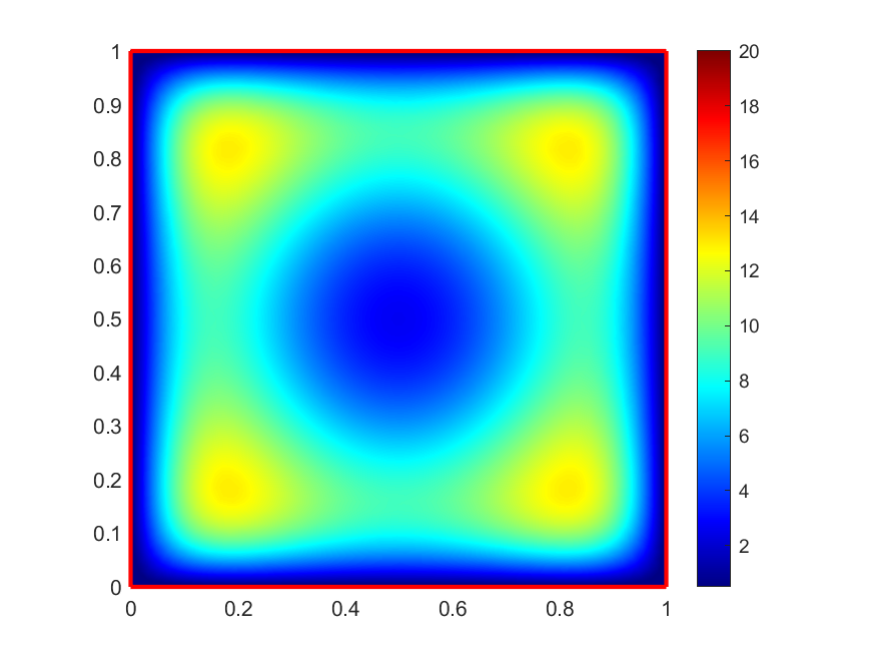}
	\includegraphics[width=.3\linewidth]{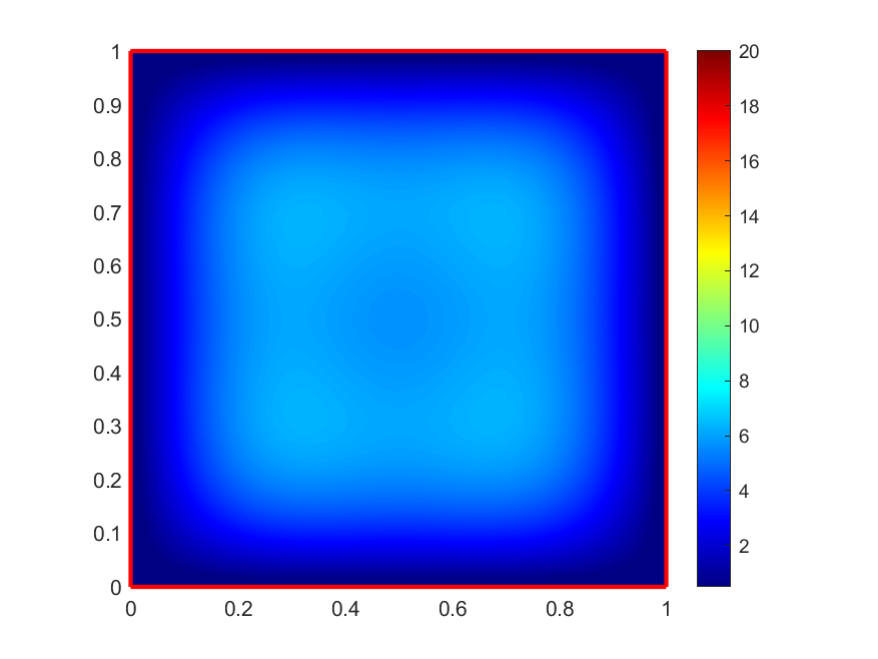}
	\includegraphics[width=.3\linewidth]{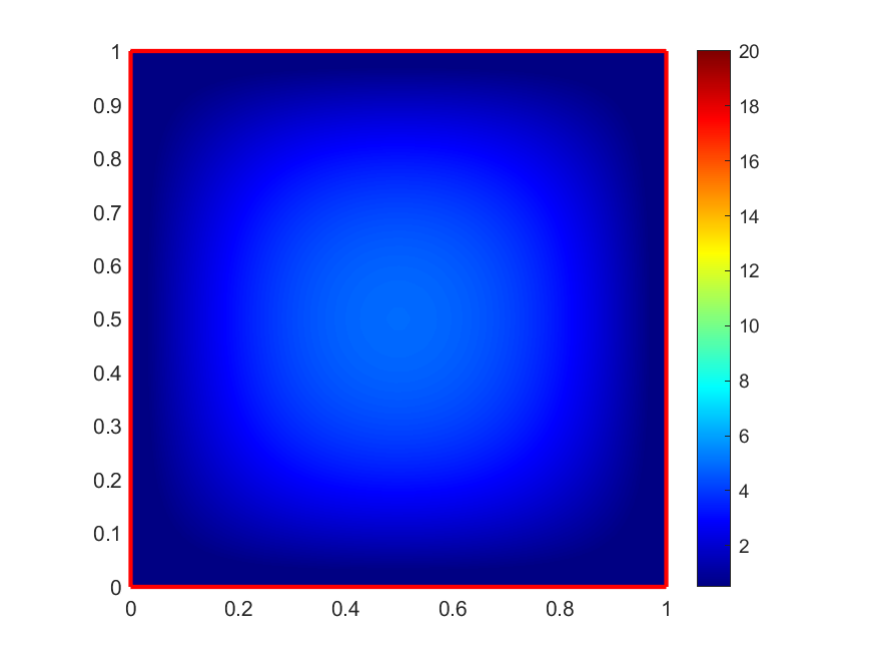}
	\caption{Evolution of the pathogen concentration field $C$ in the absence of epidemic feedback.  
		The contaminant spreads outward and gradually decays in amplitude, producing smooth, symmetric profiles.  
		This preliminary behaviour represents possible contamination zones that may act as sources for subsequent infection.}
	\label{fig:C-evolution}
\end{figure}

\begin{figure}[H]
	\centering
	\begin{minipage}{\linewidth}
		\includegraphics[width=.3\linewidth]{test1-S0-eps-converted-to.pdf}
		\includegraphics[width=.3\linewidth]{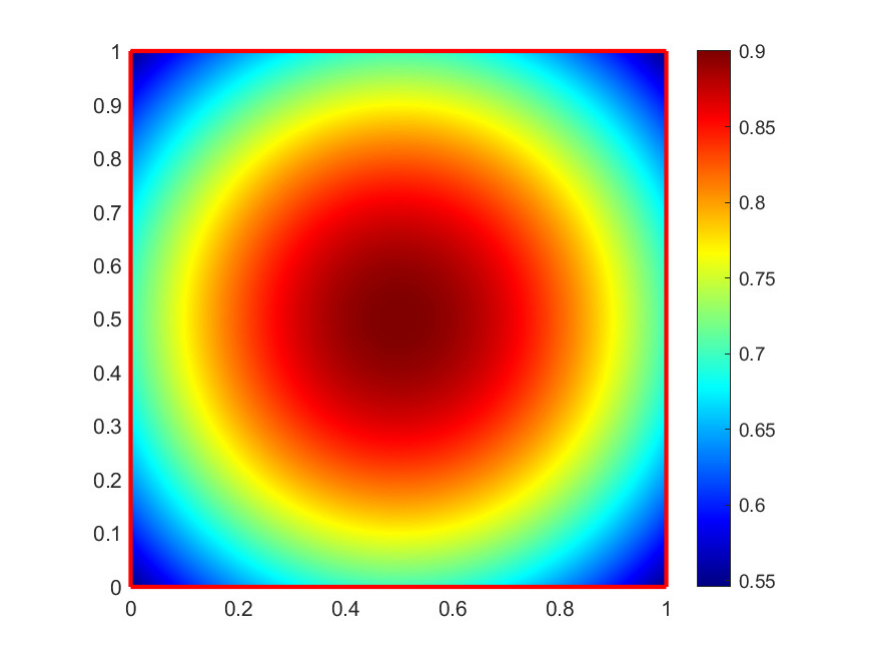}
		\includegraphics[width=.3\linewidth]{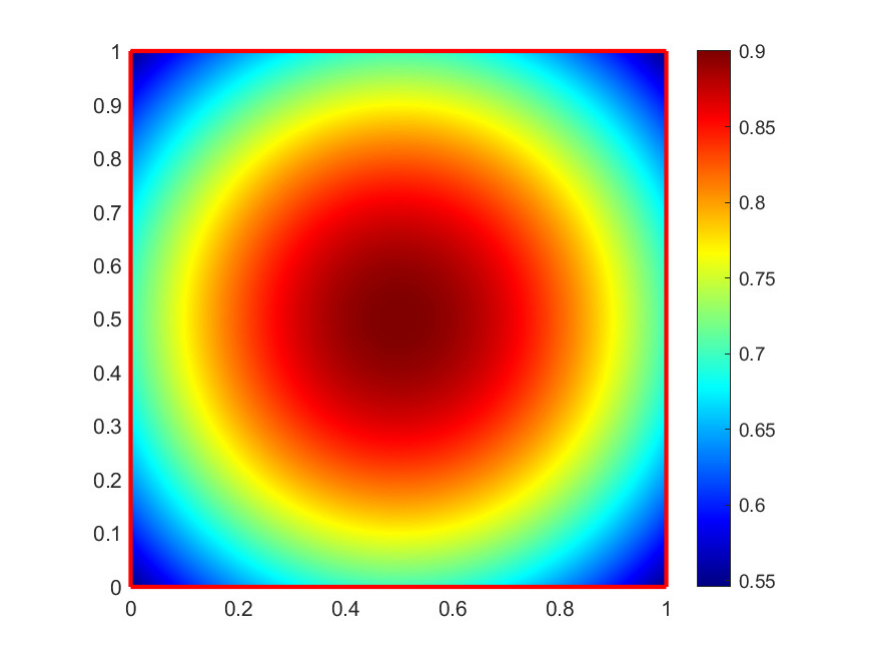}
		\caption*{(a) $S$ at $t=0$, $t=20$, and $t=40$.}
	\end{minipage}\\
	\begin{minipage}{\linewidth}
		\includegraphics[width=.3\linewidth]{test1-I0-eps-converted-to.pdf}
		\includegraphics[width=.3\linewidth]{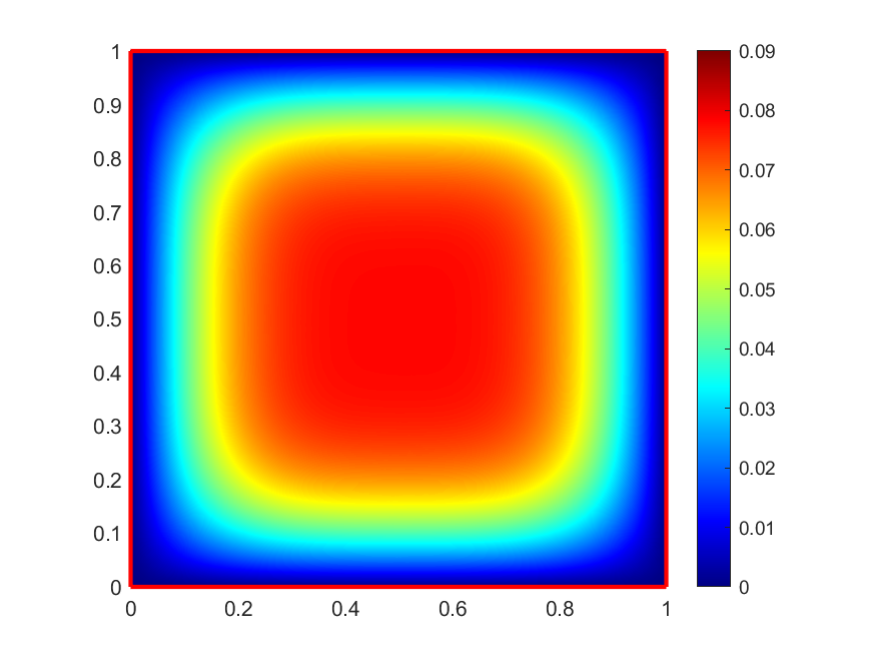}
		\includegraphics[width=.3\linewidth]{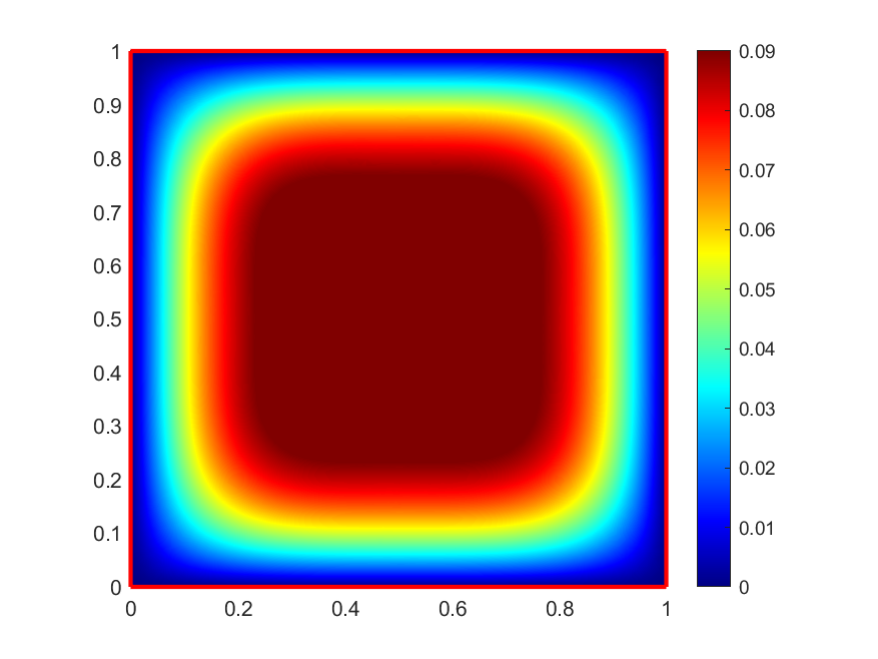}
		\caption*{(b) $I$ at $t=0$, $t=20$, and $t=40$.}
	\end{minipage}\\
	\begin{minipage}{\linewidth}
		\includegraphics[width=.3\linewidth]{test1-R0-eps-converted-to.pdf}
		\includegraphics[width=.3\linewidth]{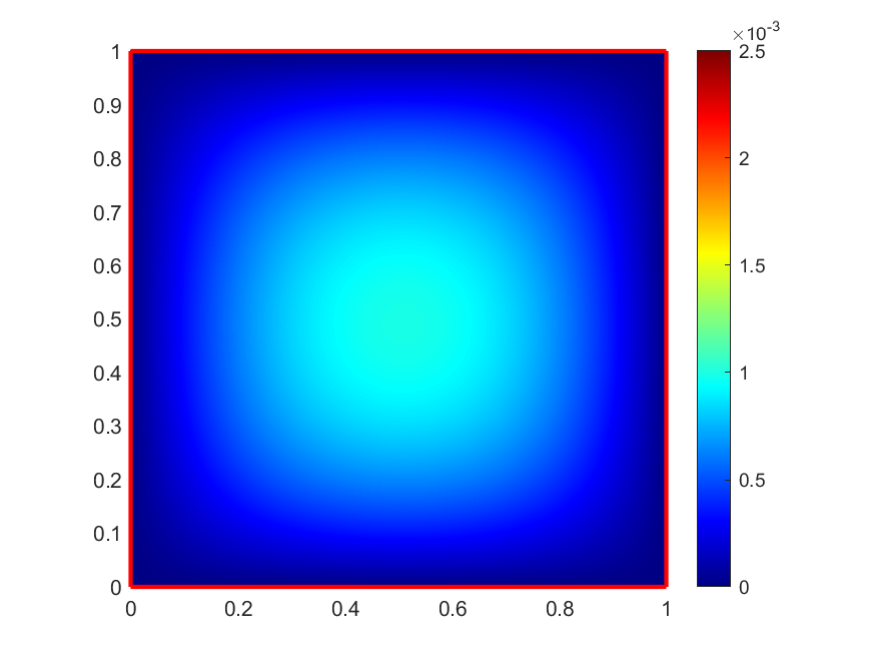}
		\includegraphics[width=.3\linewidth]{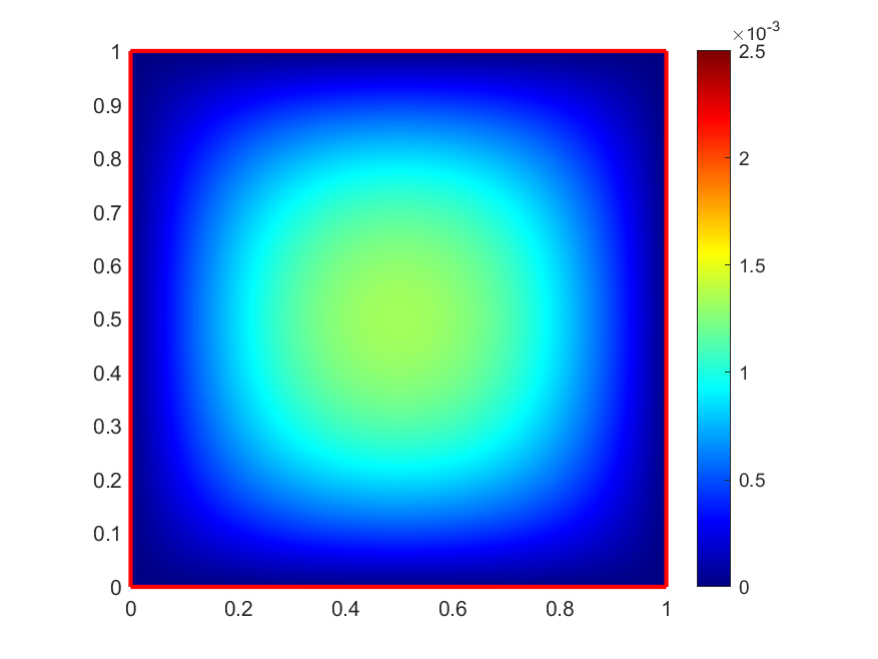}
		\caption*{(c) $R$ at $t=0$, $t=20$, and $t=40$.}
	\end{minipage}\\
	\begin{minipage}{\linewidth}
		\includegraphics[width=.3\linewidth]{test22-C5-eps-converted-to.pdf}
		\includegraphics[width=.3\linewidth]{test22-C20-eps-converted-to.pdf}
		\includegraphics[width=.3\linewidth]{test22-C40-eps-converted-to.pdf}
		\caption*{(d) $C$ at $t=0$, $t=20$, and $t=40$.}
	\end{minipage}
	\caption{Example 2 – Coupled SIR dynamics with environmental contamination, $\beta(C)=\beta_0+C$ (no fluid). Pathogen concentration amplifies infection near the contamination center, leading to faster growth of $I$ and quicker depletion of $S$.  
	The contaminant $C$ diffuses outward and decays slowly, sustaining transmission even after the direct-contact phase declines. $R$ increases progressively, confirming the reinforcing effect of environmental reservoirs on epidemic persistence.}
	\label{fig:SIRC-evolution}
\end{figure}

\begin{figure}[H]
	\centering
	\begin{minipage}{\linewidth}
		\includegraphics[width=.3\linewidth]{test1-S0-eps-converted-to.pdf}
		\includegraphics[width=.3\linewidth]{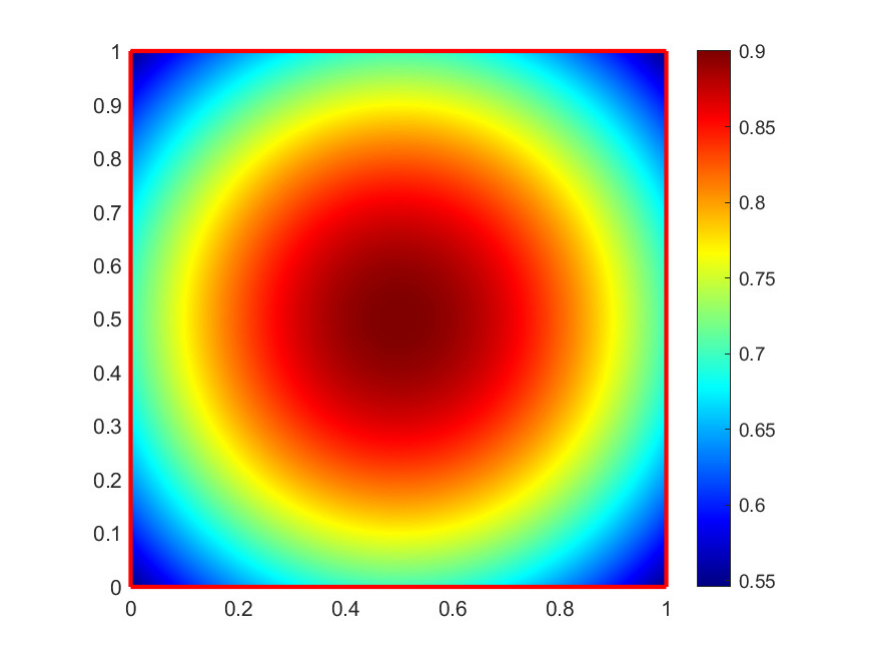}
		\includegraphics[width=.3\linewidth]{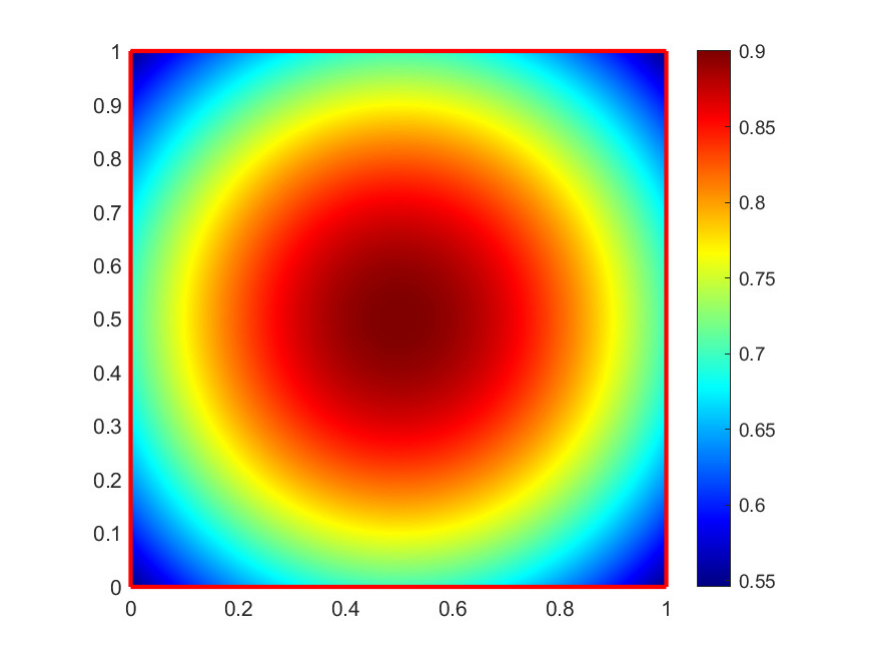}
		\caption*{(a) $S$ at $t=0$, $t=20$, and $t=40$.}
	\end{minipage}\\
	\begin{minipage}{\linewidth}
		\includegraphics[width=.3\linewidth]{test1-I0-eps-converted-to.pdf}
		\includegraphics[width=.3\linewidth]{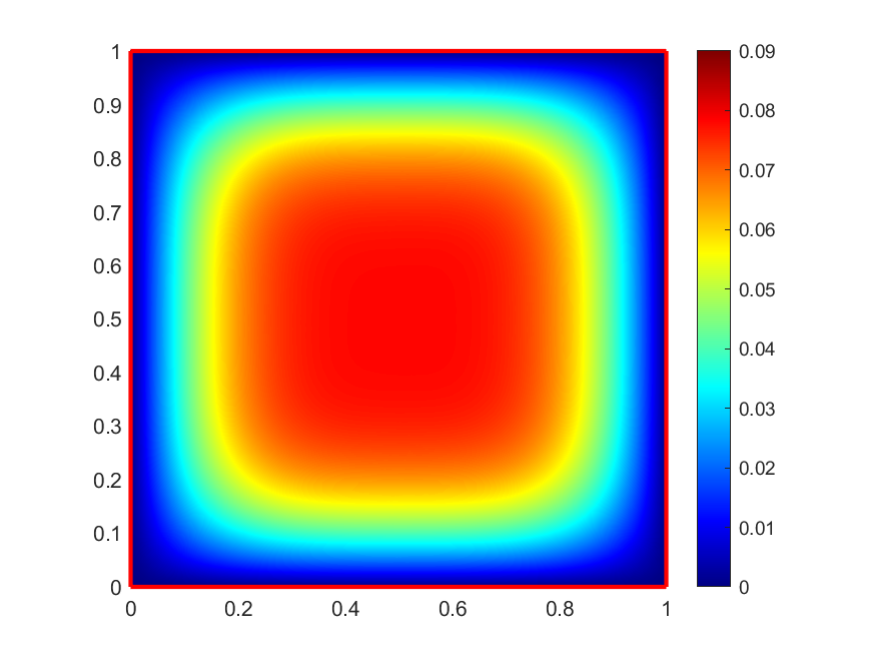}
		\includegraphics[width=.3\linewidth]{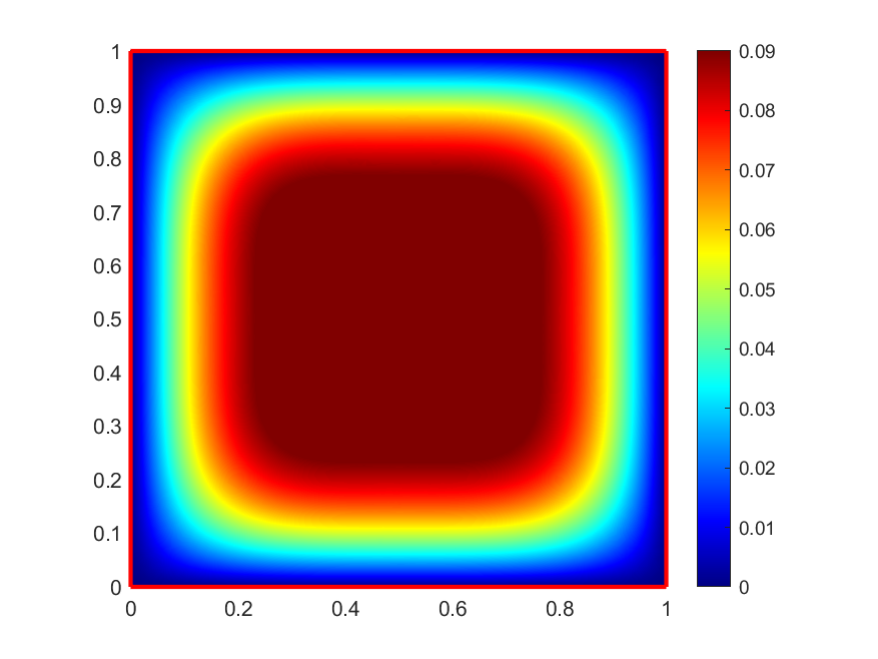}
		\caption*{(b) $I$ at $t=0$, $t=20$, and $t=40$.}
	\end{minipage}\\
	\begin{minipage}{\linewidth}
		\includegraphics[width=.3\linewidth]{test1-R0-eps-converted-to.pdf}
		\includegraphics[width=.3\linewidth]{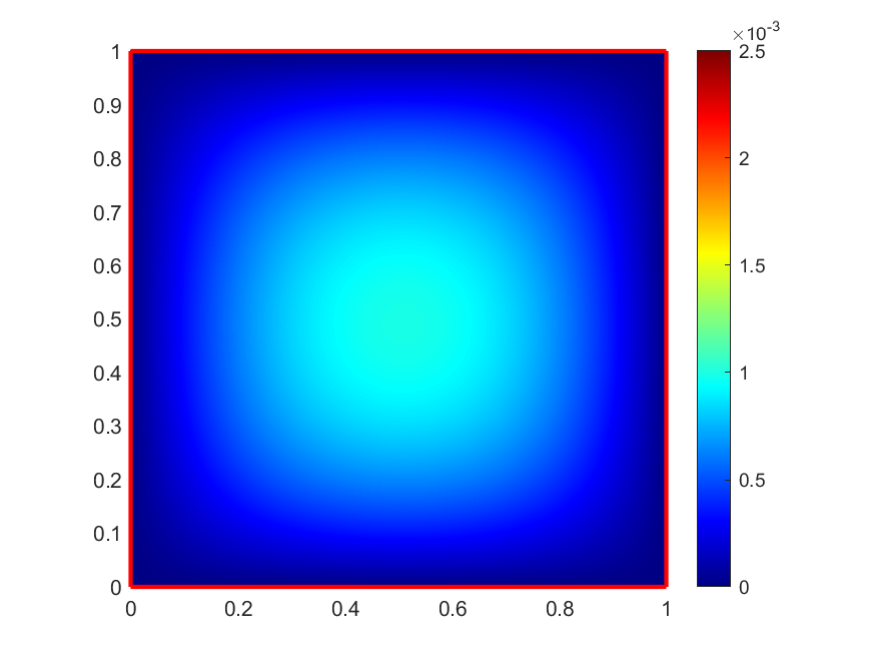}
		\includegraphics[width=.3\linewidth]{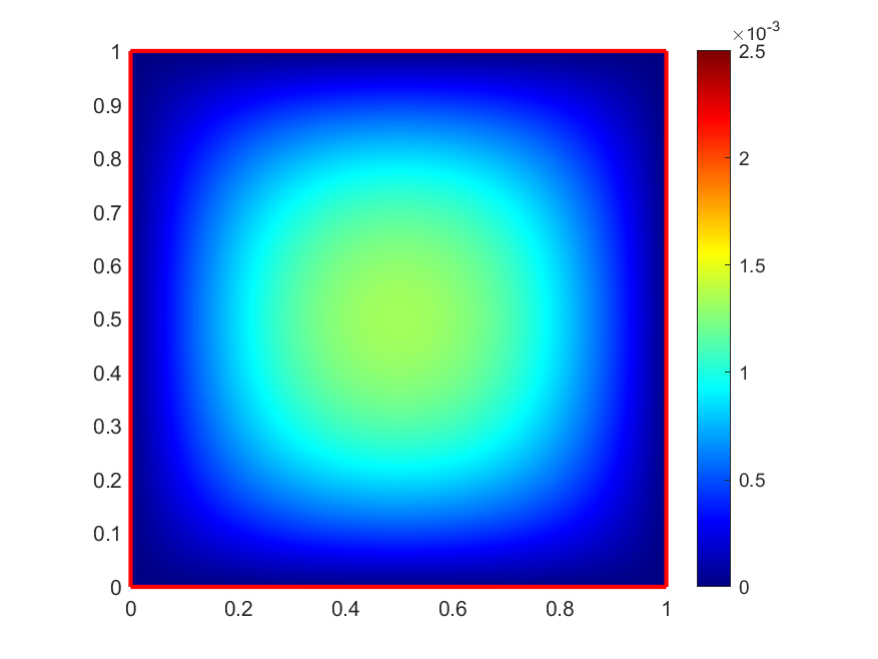}
		\caption*{(c) $R$ at $t=0$, $t=20$, and $t=40$.}
	\end{minipage}\\
	\begin{minipage}{\linewidth}
		\includegraphics[width=.3\linewidth]{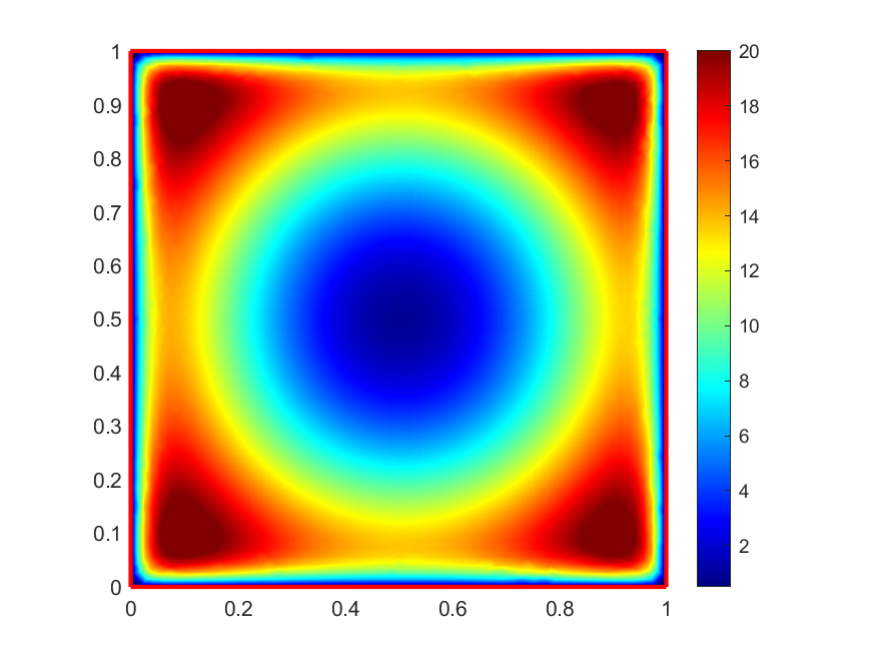}
		\includegraphics[width=.3\linewidth]{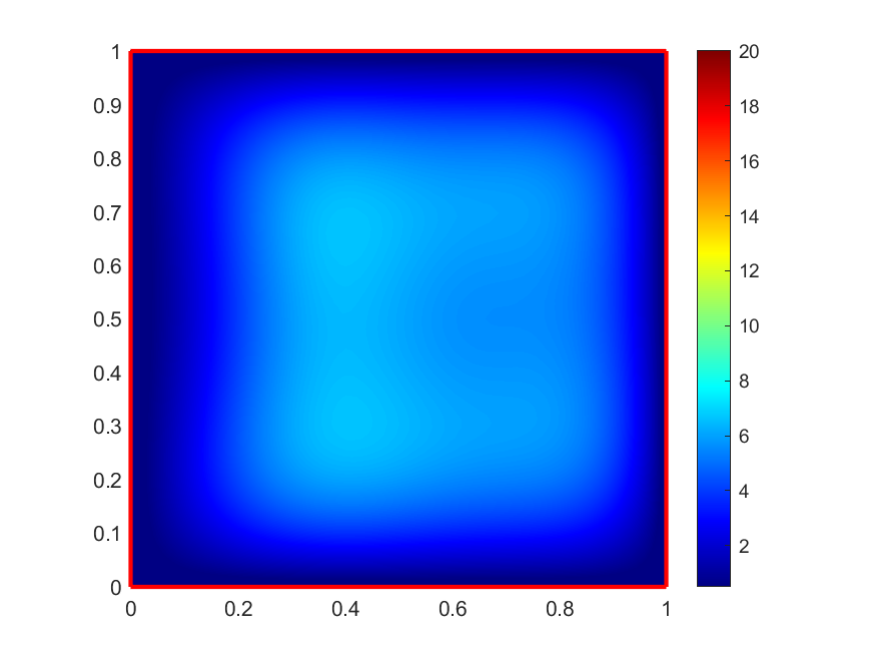}
		\includegraphics[width=.3\linewidth]{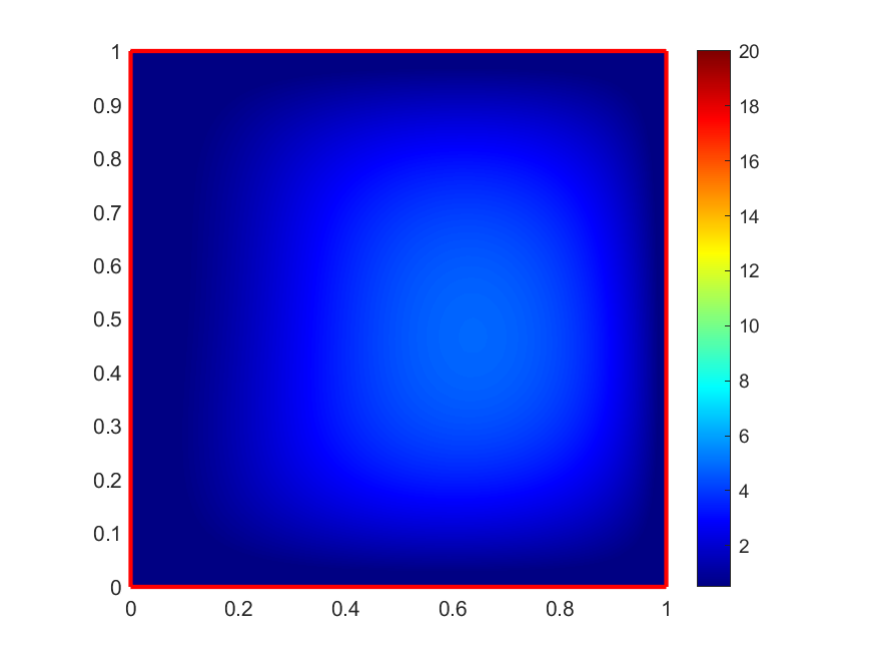}
		\caption*{(d) $C$ at $t=0$, $t=20$, and $t=40$.}
	\end{minipage}\\
	\begin{minipage}{\linewidth}
		\includegraphics[width=.3\linewidth]{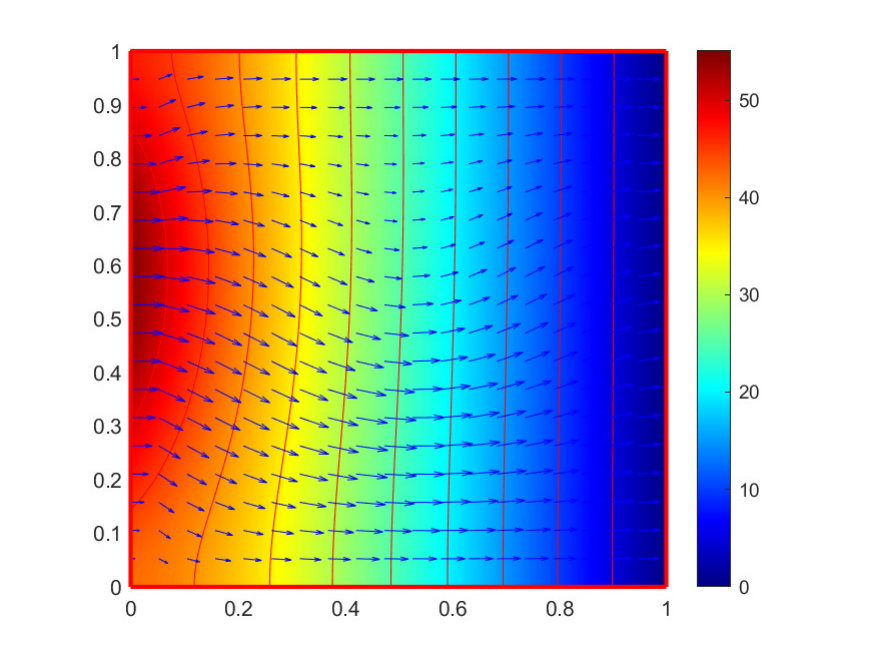}
		\includegraphics[width=.3\linewidth]{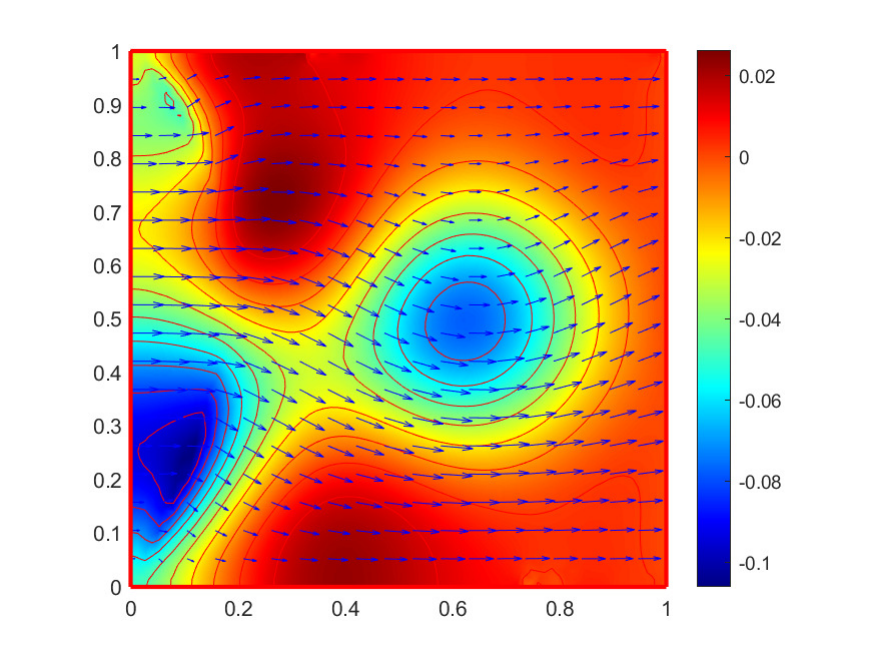}
		\includegraphics[width=.3\linewidth]{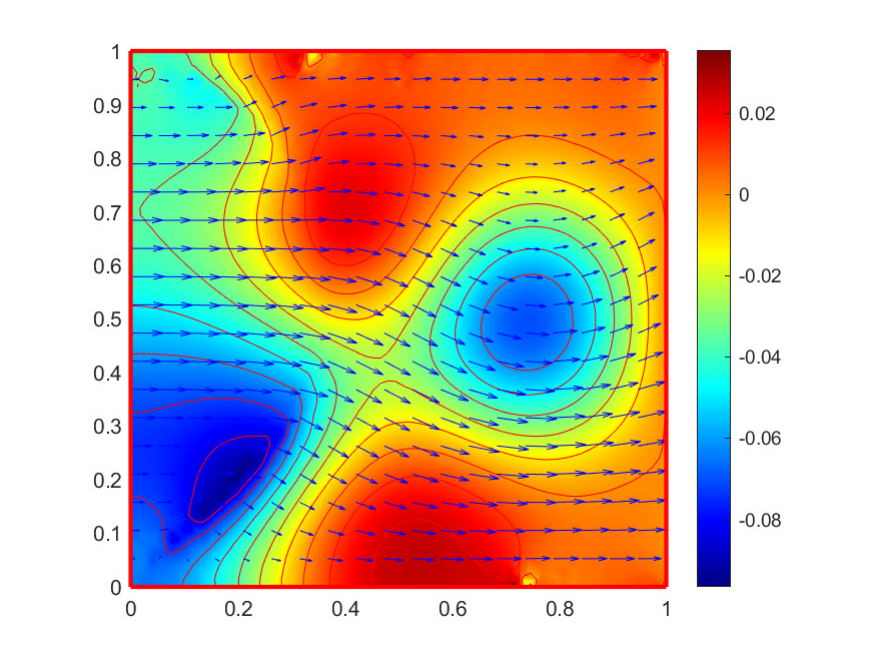}
		\caption*{(e) $\boldsymbol{U}$ at $t=0$, $t=20$, and $t=40$.}
	\end{minipage}
	\caption{Example 3 – Hydrodynamic coupling with constant viscosity ($\nu=\mathrm{const}$). The velocity field $\boldsymbol{U}$ advects the pathogen, stretching the infected area along streamlines and enlarging its spatial reach.  
	The resulting dispersion accelerates epidemic spread compared to the static case, illustrating the amplifying role of fluid transport.}
	\label{fig:CIRCU-const}
\end{figure}

\begin{figure}[H]
	\centering
	\begin{minipage}{\linewidth}
		\includegraphics[width=.3\linewidth]{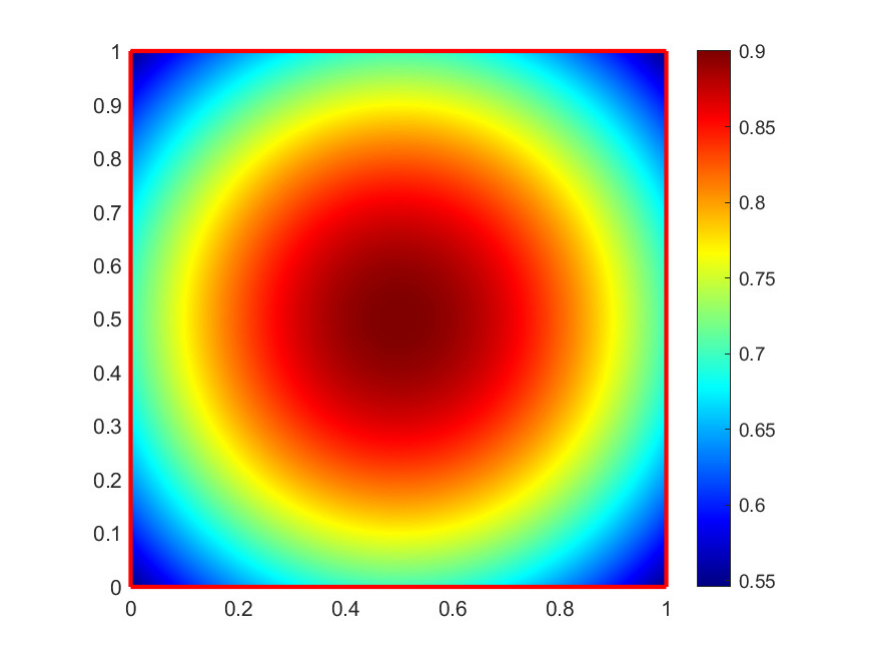}
		\includegraphics[width=.3\linewidth]{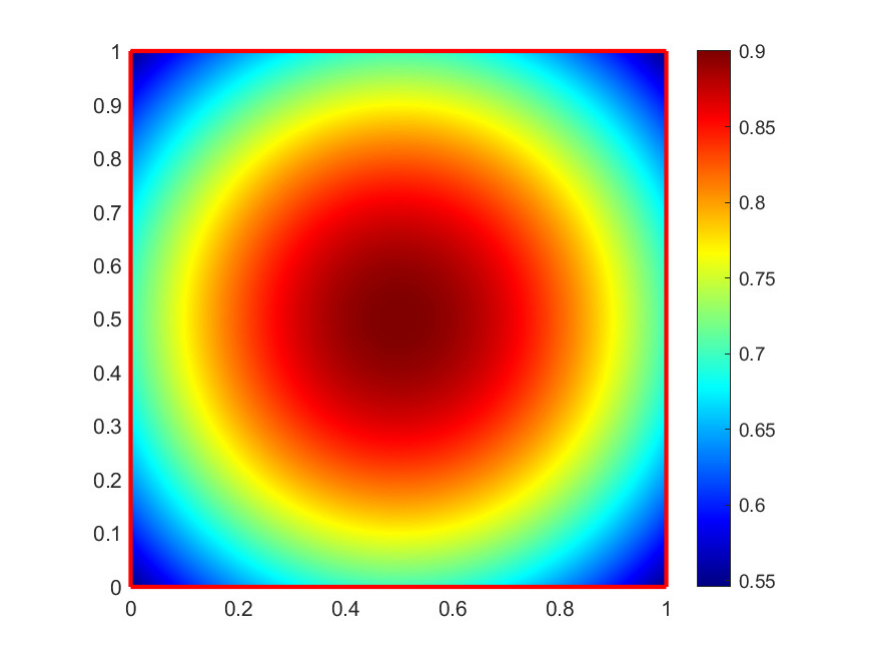}
		\includegraphics[width=.3\linewidth]{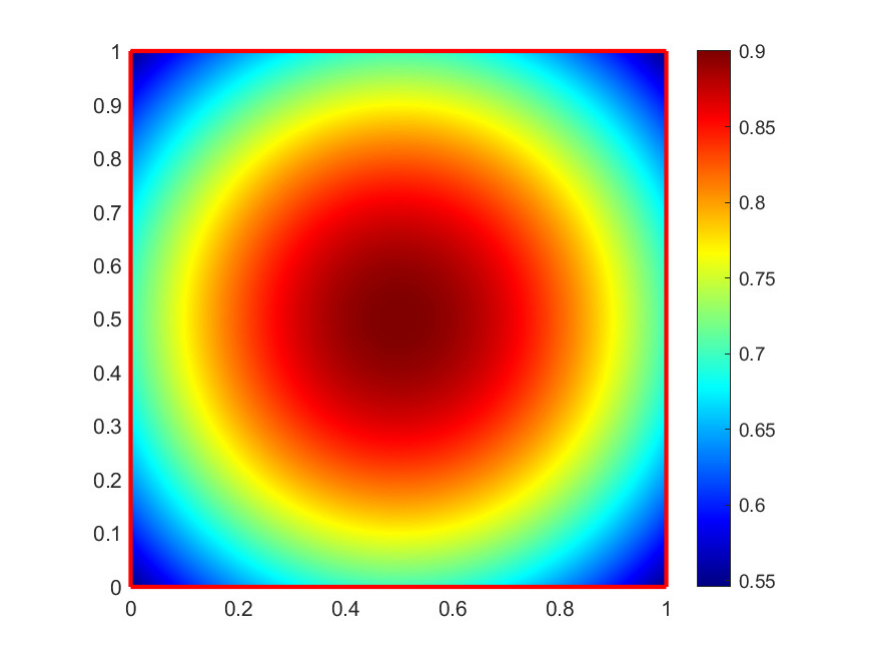}
		\caption*{(a) $S$ at $t=0$, $t=20$, and $t=40$.}
	\end{minipage}\\
	\begin{minipage}{\linewidth}
		\includegraphics[width=.3\linewidth]{test1-I0-eps-converted-to.pdf}
		\includegraphics[width=.3\linewidth]{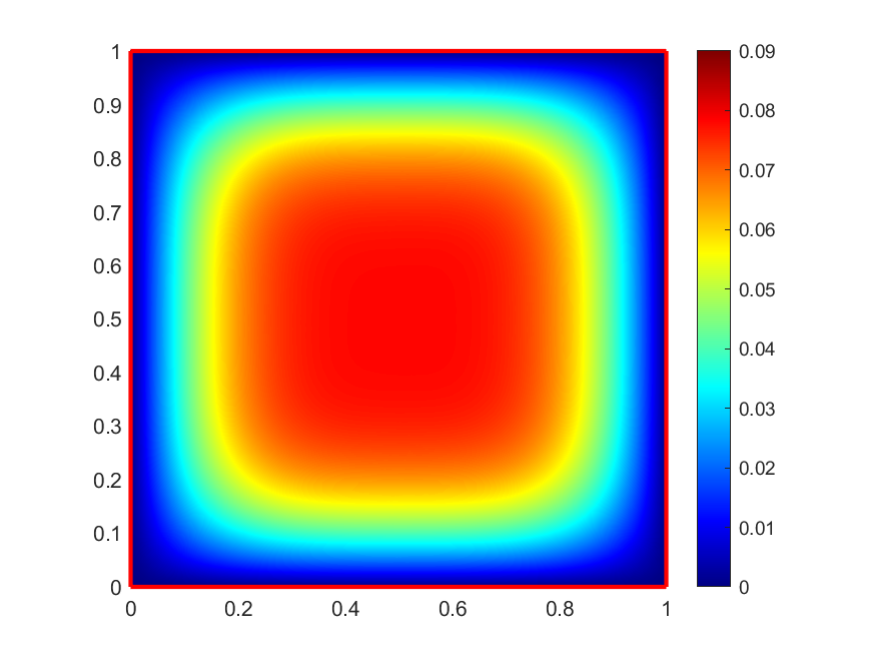}
		\includegraphics[width=.3\linewidth]{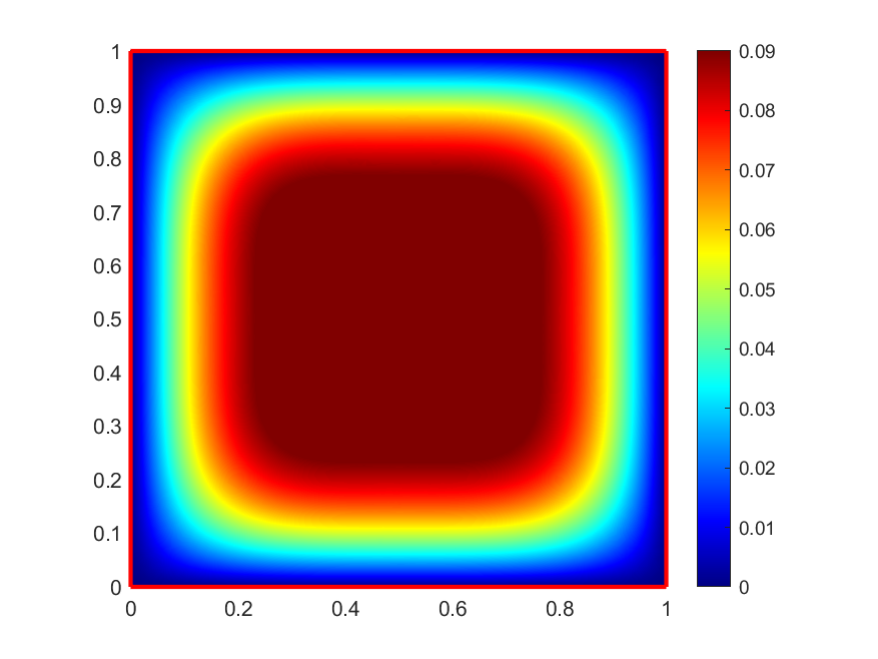}
		\caption*{(b) $I$ at $t=0$, $t=20$, and $t=40$.}
	\end{minipage}\\
	\begin{minipage}{\linewidth}
		\includegraphics[width=.3\linewidth]{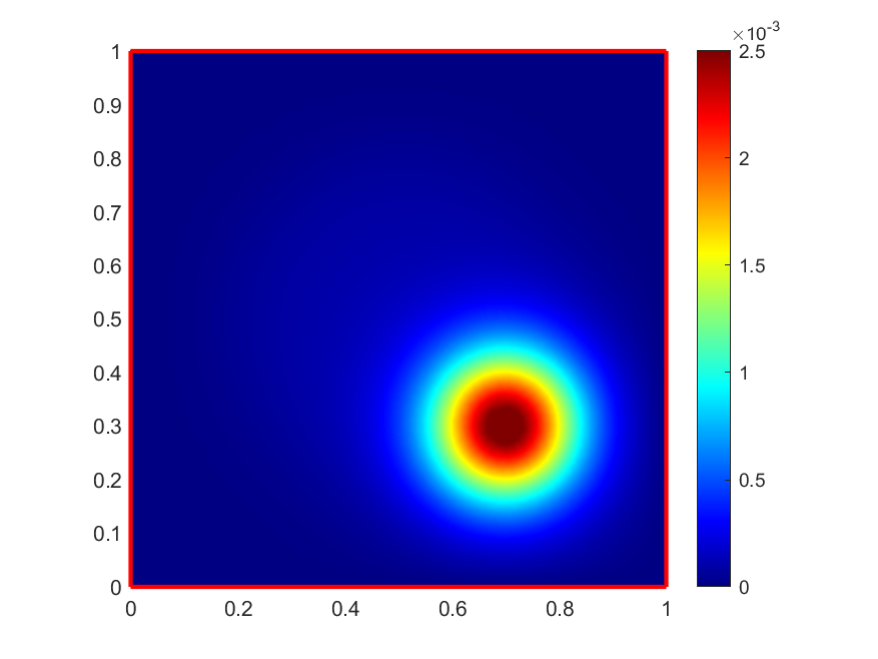}
		\includegraphics[width=.3\linewidth]{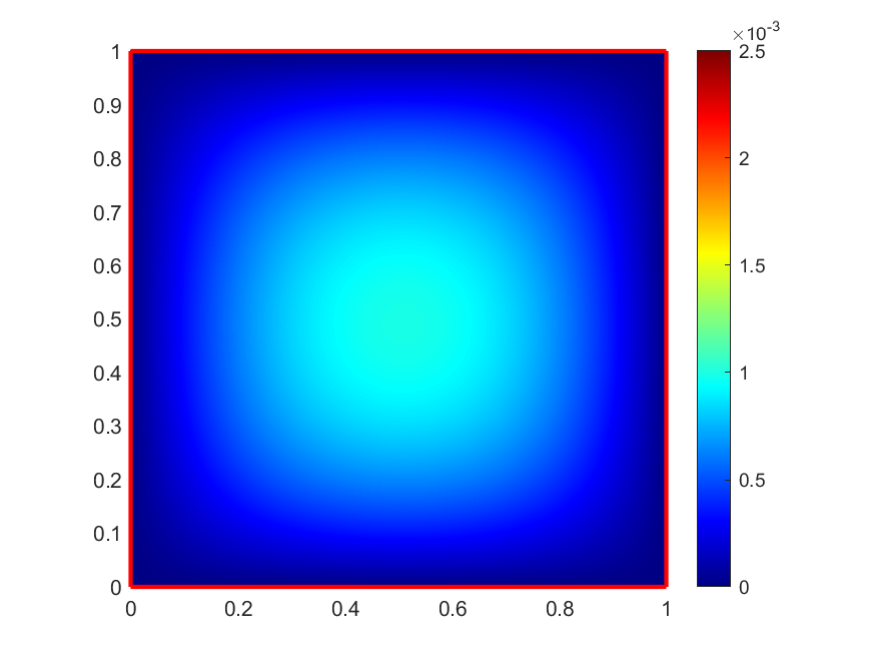}
		\includegraphics[width=.3\linewidth]{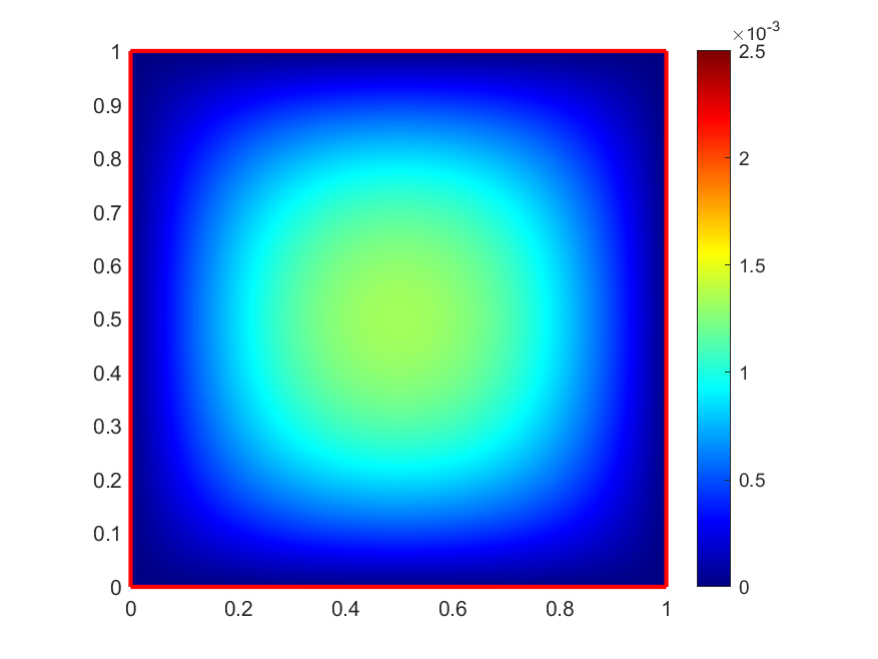}
		\caption*{(c) $R$ at $t=0$, $t=20$, and $t=40$.}
	\end{minipage}\\
	\begin{minipage}{\linewidth}
		\includegraphics[width=.3\linewidth]{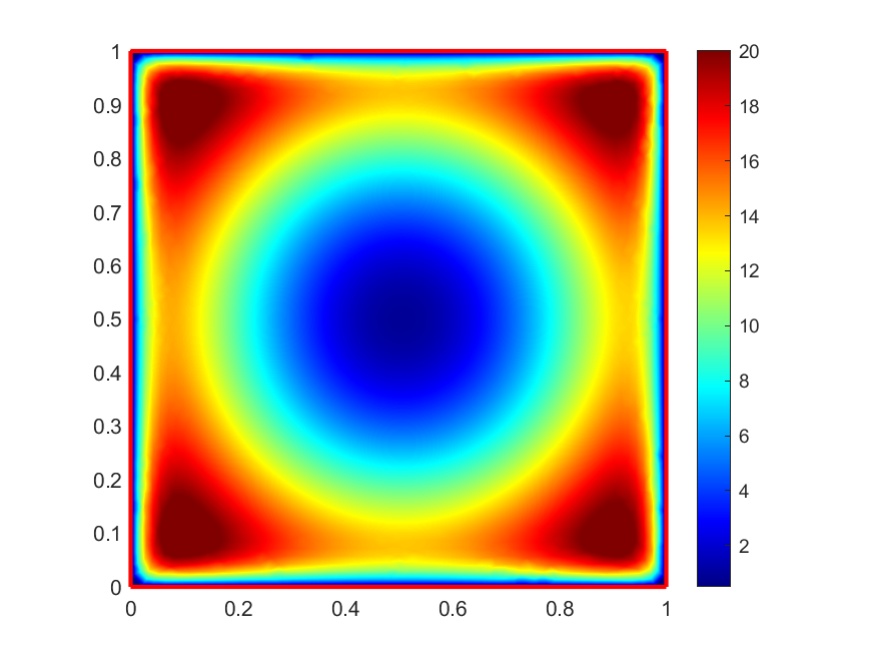}
		\includegraphics[width=.3\linewidth]{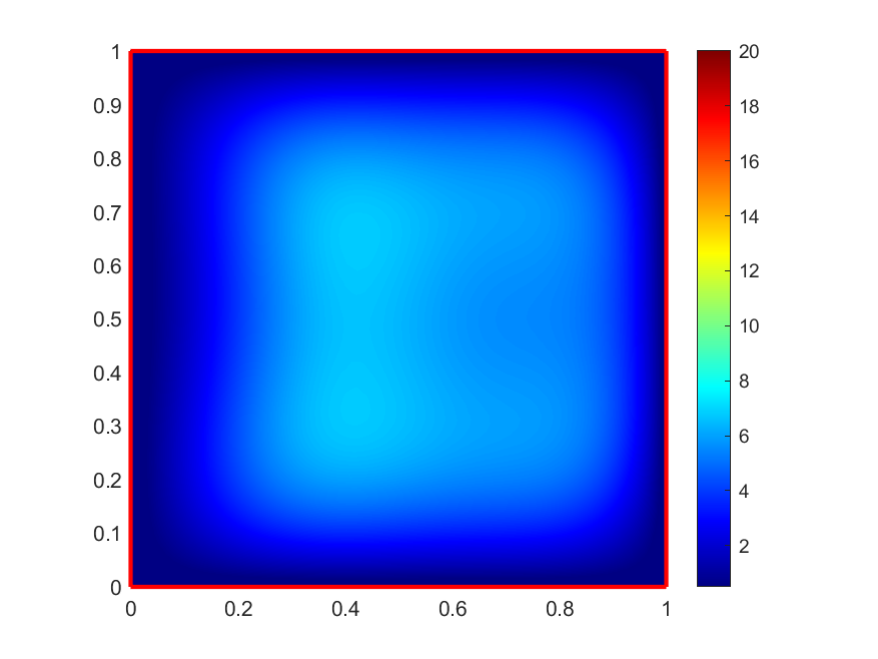}
		\includegraphics[width=.3\linewidth]{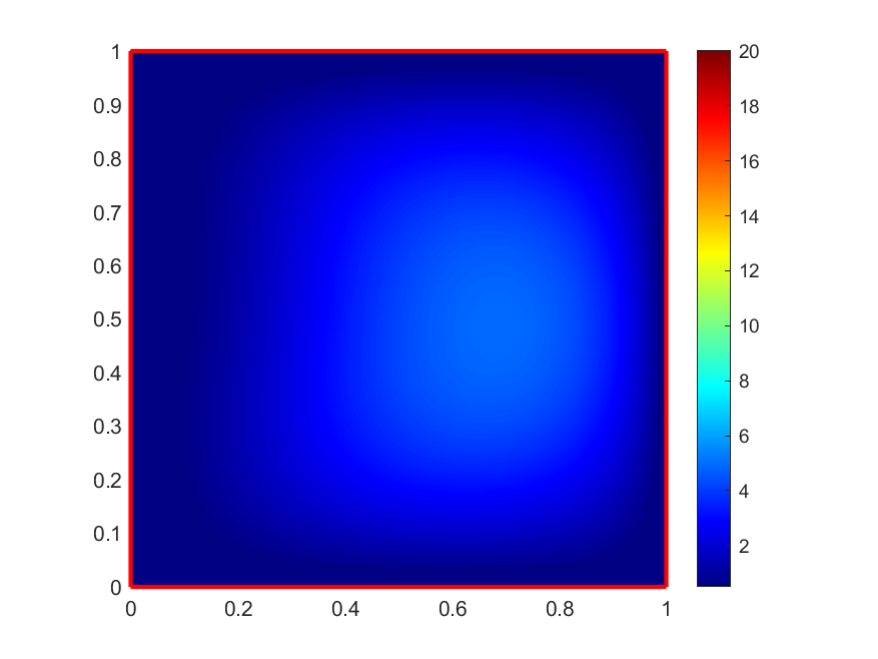}
		\caption*{(d) $C$ at $t=0$, $t=20$, and $t=40$.}
	\end{minipage}\\
	\begin{minipage}{\linewidth}
		\includegraphics[width=.3\linewidth]{test31-V0-eps-converted-to.pdf}
		\includegraphics[width=.3\linewidth]{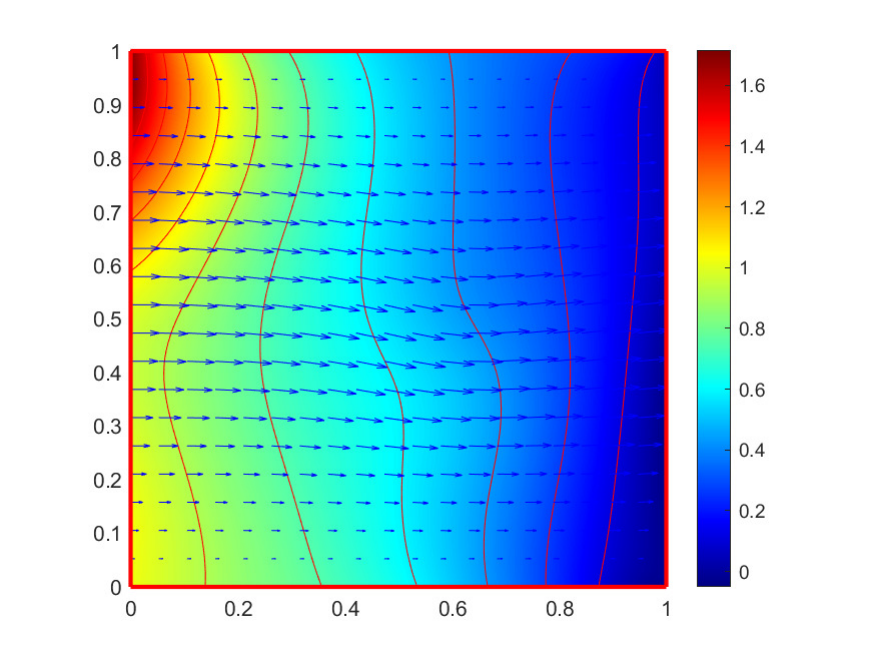}
		\includegraphics[width=.3\linewidth]{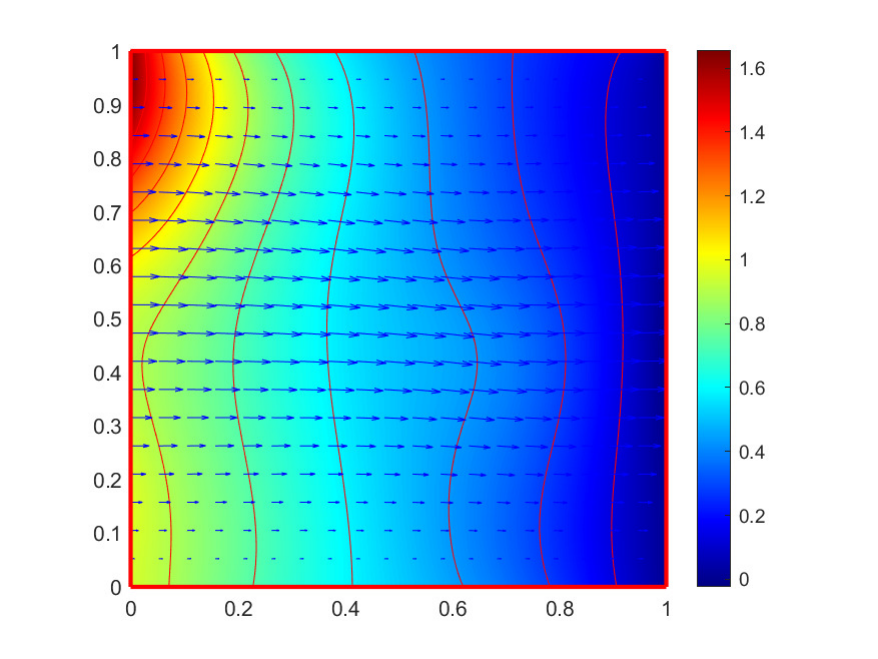}
		\caption*{(e) $\boldsymbol{U}$ at $t=0$, $t=20$, and $t=40$.}
	\end{minipage}
	\caption{Example 4 – Hydrodynamic coupling with variable viscosity ($\nu=\nu_0+C$).  
		Contaminated regions locally increase viscosity, reducing flow velocity and producing stagnation zones where pathogens accumulate.  
		This feedback limits dispersion but enhances long-term persistence, yielding heterogeneous patterns consistent with realistic biofluid behavior.}
	\label{fig:CIRCU-var}
\end{figure}

	\section{Conclusion and future work}\label{conclusion}
In this paper, we introduced and examined a new coupled SIR–Pathogen–Navier–Stokes (SIRPNS) system that explores the interplay between disease dynamics, environmental pollution, and fluid movement. Our model features a two-way connection between pathogen levels and fluid flow, using a viscosity that changes with concentration. This approach expands the traditional SIR model to include hydrodynamic effects, which are crucial for understanding waterborne diseases. Mathematically, we demonstrated the global existence and uniqueness of biologically valid weak solutions by employing the Faedo–Galerkin method along with compactness and energy estimates. This means our system is not only analytically sound but also biologically relevant.

The numerical experiments, carried out using finite element discretization, offered important insights into how the model behaves under different scenarios. Without fluid movement, the model mirrored the classic diffusion-driven epidemic cycle, where infections spread locally and stabilize through recovery. However, when we added environmental contamination, the concentration of pathogens increased the intensity of infections and extended the duration of the epidemic, even after hosts recovered. By incorporating fluid dynamics, we discovered that advection and viscosity feedback play a significant role in shaping spatial transmission patterns: constant-viscosity flows promote dispersal, while variable-viscosity feedback creates localized stagnation zones that trap pathogens, resulting in uneven and prolonged outbreaks. These results highlighted the critical influence of hydrodynamic feedback and environmental persistence on real-world epidemic behavior.

In the future, we plan to expand our current framework in a few key ways. On the theoretical front, adding stochastic perturbations or considering spatial variations could really enhance our understanding of uncertainty and how patterns form in contaminated environments. This is a relatively new area of exploration, with several researchers recently diving into scenarios where models lack fluid dynamics \cite{mehdaoui2024well,bendahmane2025global}. From a practical standpoint, integrating our model with actual hydrological or climate data would allow us to make precise predictions about contamination routes in rivers, coastal areas, or wastewater systems. Additionally, we could look into optimal control and data assimilation techniques to create effective intervention strategies for tackling waterborne epidemics.
	\subsection*{Acknowledgment}
	M. Mehdaoui would like to express his gratitude to the Euromed University of Fez, UEMF, for its valuable support.
	\subsection*{Data availability}
	Not applicable.
	\subsection*{Funding}
	No funding was received.
	\subsection*{Conflict of interest}
	No conflict of interest to be declared.
	\bibliographystyle{abbrvurl}
	\bibliography{References.bib}	

\begin{thebibliography}{10}

\bibitem{bendahmane2025global}
M.~Bendahmane, M.~Mehdaoui, and M.~Tilioua.
\newblock Global weak martingale solutions to a stochastic two-sidedly
  degenerate aggregation-diffusion equation issued from biology.
\newblock {\em arXiv preprint arXiv:2510.03947}, 2025.

\bibitem{bendahmane2025mathematical}
M.~Bendahmane, Y.~Mezzan, Y.~Ouzrour, and M.~Zagour.
\newblock Mathematical study of a nonlinear ecological system describing
  forager--exploiter interactions in a fluid environment.
\newblock {\em Mathematical Methods in the Applied Sciences}, 2025.

\bibitem{bendahmane2024mathematical2}
M.~Bendahmane, Y.~Ouakrim, Y.~Ouzrour, and M.~Zagour.
\newblock Mathematical analysis of a two-dimensional radiofrequency ablation
  model in cardiac tissue with {L}$^1$ energy dissipation.
\newblock {\em Moroccan Journal of Pure and Applied Analysis}, 10(3):285--309,
  2024.

\bibitem{bendahmane2024mathematical1}
M.~Bendahmane, Y.~Ouakrim, Y.~Ouzrour, and M.~Zagour.
\newblock Mathematical study of a new coupled electro-thermo radiofrequency
  model of cardiac tissue.
\newblock {\em Communications in Nonlinear Science and Numerical Simulation},
  139:108281, 2024.

\bibitem{bendahmane2026mathematical}
M.~Bendahmane, Y.~Ouakrim, Y.~Ouzrour, and M.~Zagour.
\newblock Mathematical analysis and numerical simulation of a nonlinear
  radiofrequency ablation model in cardiac tissue.
\newblock {\em Nonlinear Analysis: Real World Applications}, 87:104412, 2026.

\bibitem{buonomo2012forces}
B.~Buonomo and D.~Lacitignola.
\newblock Forces of infection allowing for backward bifurcation in an epidemic
  model with vaccination and treatment.
\newblock {\em Acta Applicandae Mathematicae}, 122(1):283--293, 2012.

\bibitem{capasso1978global}
V.~Capasso.
\newblock Global solution for a diffusive nonlinear deterministic epidemic
  model.
\newblock {\em SIAM Journal on Applied Mathematics}, 35(2):274--284, 1978.

\bibitem{dong2022global}
L.~Dong, S.~Hou, and C.~Lei.
\newblock Global attractivity of the equilibria of the diffusive {SIR} and
  {SEIR} epidemic models with multiple parallel infectious stages and nonlinear
  incidence mechanism.
\newblock {\em Applied Mathematics Letters}, 134:108352, 2022.

\bibitem{el2020mathematical}
M.~El~Jarroudi, H.~Karjoun, L.~Kouadio, and M.~El~Jarroudi.
\newblock Mathematical modelling of non-local spore dispersion of wind-borne
  pathogens causing fungal diseases.
\newblock {\em Applied Mathematics and Computation}, 376:125107, 2020.

\bibitem{gu2025distinct}
T.~Gu, Y.~Liu, Y.~Wang, H.~Zheng, and L.~Chen.
\newblock Distinct impact of polystyrene microplastics on six species of common
  pathogenic and probiotic bacteria and their boosting support to vibrio
  cholerae proliferation.
\newblock {\em Environmental Science: Processes \& Impacts}, 27(8):2353--2366,
  2025.

\bibitem{hau2006wind}
B.~Hau and C.~de~Vallavieille-Pope.
\newblock Wind-dispersed diseases.
\newblock In {\em The epidemiology of plant diseases}, pages 387--416.
  Springer, 2006.

\bibitem{kermack1927contribution}
W.~O. Kermack and A.~G. McKendrick.
\newblock A contribution to the mathematical theory of epidemics.
\newblock {\em Proceedings of the royal society of london. Series A, Containing
  papers of a mathematical and physical character}, 115(772):700--721, 1927.

\bibitem{lacitignola2016backward}
D.~Lacitignola.
\newblock On the backward bifurcation of a vaccination model with nonlinear
  incidence.
\newblock {\em Nonlinear Analysis: Modelling and Control}, 2016.

\bibitem{ladyvzenskaja1968linear}
O.~A. Lady{\v{z}}enskaja.
\newblock Linear and quasi-linear equations of parabolic type, (russian).
\newblock {\em Izdat.}, 1968.

\bibitem{lipp2002effects}
E.~K. Lipp, A.~Huq, and R.~R. Colwell.
\newblock Effects of global climate on infectious disease: the cholera model.
\newblock {\em Clinical microbiology reviews}, 15(4):757--770, 2002.

\bibitem{mehdaoui2024well}
M.~Mehdaoui.
\newblock Well-posedness results for a new class of stochastic spatio-temporal
  {SIR}-type models driven by proportional pure-jump {L\'e}vy noise.
\newblock {\em Applied Mathematical Modelling}, 126:543--567, 2024.

\bibitem{mehdaouii2023analysis}
M.~Mehdaoui, A.~L. Alaoui, and M.~Tilioua.
\newblock Analysis of a stochastic {SVIR} model with time-delayed stages of
  vaccination and {L\'e}vy jumps.
\newblock {\em Mathematical Methods in the Applied Sciences},
  46(12):12570--12590, 2023.

\bibitem{mehdaoui2023dynamical}
M.~Mehdaoui, A.~L. Alaoui, and M.~Tilioua.
\newblock Dynamical analysis of a stochastic non-autonomous svir model with
  multiple stages of vaccination.
\newblock {\em Journal of Applied Mathematics and Computing}, 69(2):2177--2206,
  2023.

\bibitem{mehdaoui2023optimal}
M.~Mehdaoui, A.~L. Alaoui, and M.~Tilioua.
\newblock Optimal control for a multi-group reaction--diffusion {SIR} model
  with heterogeneous incidence rates.
\newblock {\em International Journal of Dynamics and Control},
  11(3):1310--1329, 2023.

\bibitem{mehdaoui2024optimal}
M.~Mehdaoui, D.~Lacitignola, and M.~Tilioua.
\newblock Optimal social distancing through cross-diffusion control for a
  disease outbreak {PDE} model.
\newblock {\em Communications in Nonlinear Science and Numerical Simulation},
  131:107855, 2024.

\bibitem{mehdaoui2024new}
M.~Mehdaoui and M.~Tilioua.
\newblock A new optimal cross-diffusive control for a class of spatio-temporal
  predator-prey models.
\newblock {\em Optimization}, pages 1--40, 2024.

\bibitem{morse1995factors}
S.~S. Morse.
\newblock Factors in the emergence of infectious diseases.
\newblock {\em Emerging infectious diseases}, 1(1):7, 1995.

\bibitem{ouzrour2025well}
Y.~Ouzrour, M.~Bendahmane, Y.~Ouakrim, and M.~Zagour.
\newblock Well-posedness analysis and numerical simulation of a radiofrequency
  ablation model in a porous medium.
\newblock {\em Journal of Engineering Mathematics}, 154(1):2, 2025.

\bibitem{sabbar2025refining}
Y.~Sabbar.
\newblock Refining extinction criteria in a complex multi-stage epidemic system
  with non-gaussian {L\'e}vy noise.
\newblock {\em Communications in Nonlinear Science and Numerical Simulation},
  page 108911, 2025.

\bibitem{sabbar2024probabilistic}
Y.~Sabbar, M.~Mehdaoui, M.~Tilioua, and K.~S. Nisar.
\newblock Probabilistic analysis of a disturbed {SIQP-SI} model of
  mosquito-borne diseases with human quarantine strategy and independent
  poisson jumps.
\newblock {\em Modeling Earth Systems and Environment}, 10(4):4695--4715, 2024.

\bibitem{shuai2013global}
Z.~Shuai and P.~van~den Driessche.
\newblock Global stability of infectious disease models using lyapunov
  functions.
\newblock {\em SIAM Journal on Applied Mathematics}, 73(4):1513--1532, 2013.

\bibitem{simon1986compact}
J.~Simon.
\newblock Compact sets in the space ${L}^p(0, {T}; {B})$.
\newblock {\em Annali di Matematica pura ed applicata}, 146(1):65--96, 1986.

\bibitem{wang2025effect}
N.~Wang, L.~Zhang, and Z.~Teng.
\newblock The effect of pathogens from environmental breeding and accumulative
  release by the infected individuals on spread dynamics of a {SIRP} epidemic
  model.
\newblock {\em Journal of Mathematical Biology}, 90(3):30, 2025.

\bibitem{wang2016reaction}
X.~Wang, D.~Posny, and J.~Wang.
\newblock A reaction-convection-diffusion model for cholera spatial dynamics.
\newblock {\em Discrete and Continuous Dynamical Systems Series B},
  21:2785--2809, 2016.

\bibitem{yang2025reaction}
L.~Yang and M.~Fan.
\newblock Reaction--advection--diffusion model of highly pathogenic avian
  influenza with behavior of migratory wild birds.
\newblock {\em Journal of Mathematical Biology}, 90(2):18, 2025.

\bibitem{zagour2024time}
M.~Zagour.
\newblock A time-dependent {SIRD} nonlinear cross-diffusion epidemic model:
  Multiscale derivation and computational analysis.
\newblock In {\em Predicting Pandemics in a Globally Connected World, Volume 2:
  Toward a Multiscale, Multidisciplinary Framework through Modeling and
  Simulation}, pages 127--156. Springer, 2024.

\bibitem{zhang2018spatial}
X.~Zhang and Y.~Zhang.
\newblock Spatial dynamics of a reaction-diffusion cholera model with spatial
  heterogeneity.
\newblock {\em Discrete and Continuous Dynamical Systems Series B}, 23(6),
  2018.

\end{thebibliography}
\end{document}